\documentclass[11pt,reqno]{article}

\usepackage{a4wide}
\usepackage{amsmath,amsthm,amssymb}
\usepackage{extarrows}
\usepackage{units}
\usepackage{url}
\usepackage{calc, graphicx} 
\usepackage{bbm}
\usepackage{scalerel}
\usepackage{dsfont}

\usepackage{arxiv}
\usepackage{color}
\usepackage[utf8]{inputenc} 
\usepackage[T1]{fontenc}    
\usepackage[hidelinks]{hyperref}       
\usepackage{url}            
\usepackage{booktabs}       
\usepackage{amsfonts}       
\usepackage{nicefrac}       
  \usepackage{pgfplots}
  \pgfplotsset{compat=newest}
  \usepackage{rotating}
\usepackage{tikz}

\usetikzlibrary{matrix, decorations.pathreplacing}
  \usetikzlibrary{plotmarks}
  \usetikzlibrary{arrows.meta}
  \usepgfplotslibrary{patchplots}
  \usepackage{grffile}
  \usepackage{amsmath}
  \usepackage{nicematrix}
  \usepackage{import}
\usepackage{xifthen}
\usepackage{pdfpages}
\usepackage{transparent}
\usepackage{tikz}
\usepackage{extarrows}
\usepackage{caption}
\DeclareCaptionLabelFormat{cont}{#1~#2\alph{ContinuedFloat}}
\usepackage{microtype}      
\usepackage{lipsum}		
\usepackage{graphicx}
\usepackage{float}
\usepackage{multitoc}
\usepackage{enumitem}
\usepackage{subfigure}
\usepackage{doi}
\usepackage{amsfonts}
\usepackage{mathrsfs}
\usepackage{bm}
 \usepackage{nicematrix}
\usepackage{latexsym}
\usepackage{arydshln}
\usepackage[toc,page]{appendix}
\usepackage{cleveref}
\usepackage{thmtools}
\usepackage{thm-restate}

\usepackage{hyperref}

\usepackage{cleveref}

\Crefformat{equation}{(#2#1#3)}

\newtheorem{theorem}{Theorem}[section]

\newtheorem{corollary}[theorem]{Corollary}
\newtheorem{proposition}[theorem]{Proposition}
\newtheorem{lemma}[theorem]{Lemma}

\numberwithin{equation}{section}
\theoremstyle{definition}
\newtheorem{definitiony}[theorem]{Definition}
\newenvironment{definition}
  {\pushQED{\qed}\definitiony}
  {\popQED\enddefinitiony}

\theoremstyle{remark}
\newtheorem{remarkx}[theorem]{Remark}
\newenvironment{remark}
  {\pushQED{\qed}\remarkx}
  {\popQED\endremarkx}

\newcommand\bc[1]{\left({#1}\right)}
\newcommand\cbc[1]{\left\{{#1}\right\}}
\newcommand\brk[1]{\left\lbrack{#1}\right\rbrack}
\newcommand\abc[1]{\left\langle{#1}\right\rangle}
\newcommand\abs[1]{\left|{#1}\right|}
\newcommand\dabs[1]{\left\|{#1}\right\|}
\newcommand\supp[1]{{\rm supp}\left({#1}\right)}
\newcommand\rk[1]{\rank_{\mathbb{F}}\left({#1}\right)}
\newcommand\trk[1]{\rank_{\mathbb{F}}}
\newcommand\bin[1]{{\rm Bin}\left({#1}\right)}
\newcommand{\sss}{\scriptscriptstyle}
\newcommand\dtv[1]{{\rm d}_{\rm \sss TV}\left({#1}\right)}
\newcommand*{\dif}{\mathop{}\!\mathrm{d}}
\newcommand\poi[1]{{\rm Po}\left({#1}\right)}
\newcommand\syn{{\rm Sym}_n(\FF^*)}
\newcommand\syN{{\rm Sym}_N(\FF^*)}
\newcommand\teo[1]{\ensuremath{\mathds{1}}{\left\{{#1}\right\}}}

\DeclareRobustCommand{\VAN}[3]{#2}

\newcommand\rd{{\rm d}}
\newcommand\vw{\bm w}
\newcommand\vx{\bm x}

\newcommand\vi{\bm i}
\newcommand\vj{\bm j}

\newcommand\vu{\bm u}
\newcommand\vv{\bm v}
\newcommand\vy{\bm y}
\newcommand\vz{\bm z}
\newcommand\vze{\bm \zeta}

\newcommand\bvet{\bar{\bm \eta}}

\newcommand\vtu{\bm \tau}
\newcommand\bvta{\bar{\bm \tau}}

\newcommand\val{\bm{\alpha}}
\newcommand\vbe{\bm{\beta}}

\newcommand\vth{\bm{\theta}}
\newcommand\hval{\bm{\alpha}^T}

\newcommand\bfh{\mathbf h}

\newcommand\oone{\bar{o}_{\PP}(1)}

\newcommand\fB{ {\mathfrak F}^c}
\newcommand\fBT{{\mathfrak F}_{{\rm tr}}^c}
\newcommand\ffB{{\mathfrak F}}
\newcommand\ffBT{ {\mathfrak F}_{{\rm tr}}}

\newcommand\cF{\mathcal{F}}
\newcommand\cX{\mathcal{X}}
\newcommand\cY{\mathcal{Y}}
\newcommand\cZ{\mathcal{Z}}
\newcommand\cU{\mathcal{U}}
\newcommand\cV{\mathcal{V}}
\newcommand\cW{\mathcal{W}}

\newcommand\THETA{\mathbf\Theta}

\newcommand\eul{\mathrm{e}}
\newcommand\eps{\varepsilon}

\newcommand\ZZ{\mathbb{Z}}
\newcommand\FF{\mathbb{F}}
\newcommand\RR{\mathbb{R}}
\newcommand\QQ{\mathbb{Q}}
\newcommand\NN{\mathbb{N}}
\newcommand\PP{\mathbb{P}}

\newcommand\Erw{\mathbb{E}}

\newcommand{\ind}{\ensuremath{\mathds{1}}}

\DeclareMathOperator{\nul}{nul}
\DeclareMathOperator{\rank}{rk}
\DeclareMathOperator{\PR}{PR}

\usepackage{todonotes}
\usepackage{color}

\DeclareSymbolFont{extraup}{U}{zavm}{m}{n}
\DeclareMathSymbol{\varheart}{\mathalpha}{extraup}{86}
\DeclareMathSymbol{\vardiamond}{\mathalpha}{extraup}{87}
\usepackage{indentfirst}
\title{The rank of sparse symmetric matrices over arbitrary fields}

\author{Remco van der Hofstad, Noela Müller, Haodong Zhu}
\begin{document}
\maketitle
\begin{abstract}
Let $\FF$ be an arbitrary field and $(\bm{G}_{n,d/n})_n$ be a sequence of sparse weighted Erd\H{o}s-Rényi random graphs on $n$ vertices with edge probability $d/n$, where weights from $\FF \setminus\{0\}$ are assigned to the edges according to a fixed matrix $J_n$. We show that the normalised rank of the adjacency matrix of $(\bm{G}_{n,d/n})_n$ converges in probability to a constant, and derive the limiting expression. Our result shows that for the general class of sparse symmetric matrices under consideration, the asymptotics of the normalised rank are independent of the edge weights and even the field, in the sense that the limiting constant for the general case coincides with the one previously established for adjacency matrices of sparse (non-weighted) Erd\H{o}s-Rényi matrices over $\RR$ from \cite{bordenave2011rank}. Our proof, which is purely combinatorial in its nature, is based on an intricate extension of the novel perturbation approach from \cite{coja2022rank} to the symmetric setting.


\end{abstract}

\section{Introduction}

\subsection{Background and motivation} \label{sec_intro}
The study of matrices with random entries, going back to the $1950$'s \cite{wigner1955characteristic}, is an important and lively field of modern probability and combinatorics with close ties to a multitude of other scientific disciplines such as theoretical physics, mathematical statistics, computer science, neuroscience or machine learning. Up to this day, the theory of random matrices has developed into a mature field and advanced to a very precise understanding of classical models such as Gaussian Ensembles, Bernoulli matrices or Wishart matrices.

Moreover, in the last decade, there has been a burst of progress in the theoretical understanding of random matrices which appear naturally in  the study of random graphs, such as their adjacency matrices. Especially the adjacency matrix of the classical Erd\H{o}s-Rényi random graph model and its spectral properties have attracted a great deal of attention \cite{benaych2019largest,erdos2012spectral,erdos2013spectral}.
The Erd\H{o}s-Rényi graph $\bm{G}_{n,p_n} = (\{1, \ldots, n\}, \bm{E}_n)$, which is arguably the simplest random graph model, is a graph on $n$ vertices, where each edge is present independently with probability $p_n$. Its adjacency matrix $\bm{A}_{n,p_n}$ is a symmetric $n \times n$-matrix with entries $\bm{A}_{n,p_n}(i,j) = \ind\cbc{\{i,j\} \in \bm{E}_n}$\footnote{For an event $B$, $\ensuremath{\mathds{1}}\{B\}$ denotes the indicator function of $B$. When appropriate, we also use $\ensuremath{\mathds{1}}B$.}.
In particular, it is a symmetric Bernoulli matrix, which, depending on the limiting behaviour of the edge probability $p_n$, displays different asymptotic behaviour: Results by Costello, Tao and Vu \cite{costello2006random} and later by Basak and Rudelson \cite{basak2021sharp} have shown that there is a sharp transition in the invertibility of the adjacency matrix around $\ln n/n+k(n)/n$, for a function $k(n)$ that tends slowly to infinity: When $p_n>\ln n/n+k(n)/n$, with high probability (w.h.p.) the adjacency matrix is nonsingular, while it is singular w.h.p. for $p_n<\ln n/n-k(n)/n$.


Following this threshold result, a natural question is to determine the rank of the adjacency matrix $\bm{A}_{n,p_n}$ when $p_n$ is small enough such that the matrix is singular w.h.p.  In the regime where $p_n\in[c\ln n/n,1/2]$ for $c>1/2$, Costello and Vu \cite{costello2008rank} show that w.h.p., the rank of $\bm{A}_{n,p_n}$ is exactly equal to $n$ minus the number of isolated vertices in the underlying Erd\H{o}s-Rényi random graph. They extend their result to $c>0$ 
and arbitrary deterministic non-zero entries (instead of $1$) in \cite{costello2010rank}. This result shows that w.h.p., the rank only depends on the structure of the graph, regardless of the precise value of the nonzero entries of the adjacency matrix. Finally, when $p_n=d/n$ for fixed $d>0$, Bordenave, Lelarge and Salez \cite{bordenave2011rank} derive an asymptotic rank formula for $\bm{A}_{n,d/n}$ (see \cref{e01b} below). 

While all these results naturally consider the rank of the adjacency matrix $\bm{A}_{n,p_n}$ over $\RR$ (or equivalently, $\QQ$), we will be interested in the rank of $\bm{A}_{n,p_n}$ over arbitrary fields $\FF$ in the sparse regime where $p_n=d/n$ (interpreting a $1$-entry as the multiplicative identity of the field, and  a $0$-entry as its additive identity). Moreover, inspired by \cite{costello2010rank}, we consider the more general class of matrices where the non-zero entries of $\bm{A}_{n,p_n}$ are arbitrary deterministic non-zero elements of $\FF$. Our main \Cref{t0} shows that even under this vast generalisation, the asymptotic rank formula of Bordenave, Lelarge and Salez still remains valid. This result suggests that the rank indeed only depends on the positions of the non-zero entries of the adjacency values, which is reflected in our proof strategy.

Indeed, thanks to observations of Bauer and Golinelli \cite{bauer2001exactly}, there is a by now well-known and purely combinatorial upper bound on the asymptotic rank of $\bm A_{n,d/n}$, which is based on the Karp-Sipser algorithm for finding large matchings \cite{karp1981maximum}: Start with $\bm G_{n,p}$. At each step of the algorithm, recursively, a vertex of degree one along with its unique neighbor is removed. The process stops once only isolated vertices and vertices of degree at least two, the so-called Karp-Sipser core, are left. It is straightforward to check that this ``leaf-removal'' leaves the nullity of the graph invariant (for a proof, see \cite{bauer2001exactly}). Since the nullity of the reduced graph is apparently lower bounded by its number of isolated vertices, this number of isolated vertices provides an upper bound on the rank of the original graph that is completely oblivious to the field or the precise values of the non-zero entries.
 Karp and Sipser \cite{karp1981maximum} also derive a formula for the asymptotic number of isolated vertices in the reduced graph. Moreover, for $d \leq \eul$, all but a vanishing proportion of vertices become isolated after running the Karp-Sipser algorithm on $\bm{G}_{n,d/n}$.
 Thus, for $d \leq \eul$, 
 the question is already completely settled. 
 However, when $d> \eul$, w.h.p., the Karp-Sipser core is not negligible, which complicates matters significantly. 

Since there is already the rank formula of \cite{bordenave2011rank} in the sparse case, a natural take on the problem of the missing lower bound would be to turn to the proof methods of Bordenave, Lelarge and Salez and adapt them to our setting. However, their analysis makes heavy use of spectral properties of \textit{real} symmetric matrices, so to the best of our knowledge, there is no possibility to follow their approach. 


On the other hand, inspired by insights from statistical physics, 
 Coja-Oghlan, Ergür, Gao, Hetterich and Rolvien \cite{coja2022rank} found a new combinatorial approach to derive an asymptotic rank formula for a broad class of \textit{asymmetric} sparse random matrices, generalising earlier results by Cooper, Frieze and Pegden for $\FF_2$ \cite{cooper2019rank}. 
 Correspondingly, the results of \cite{coja2022rank} are valid over any field, regardless of the distribution of the non-zero entries. However, their approach cannot straightforwardly be applied to symmetric random matrices, since these retain much less independence among the positions of their non-zero entries.
 Indeed, the authors note that ``an intriguing question for future research is to extend the techniques from the present paper to symmetric random matrices.'' 

In this paper, we build on several of the core concepts of \cite{coja2022rank} to develop a corresponding combinatorial approach towards rank formulas for sparse symmetric matrices. As in \cite{coja2022rank}, instead of investigating the rank of $\bm A_{n, d/n}$ directly, we work with a {\em perturbed} version of $\bm A_{n,d/n}$. Moreover, as in \cite{coja2022rank}, we use a telescoping argument to lower bound the expected rank and relate the rank difference of matrices whose sizes differ by one to so-called ``frozen'' variables.
However, the symmetry of our matrices poses serious obstructions to any attempt to literally follow in the footsteps of \cite{coja2022rank}, and we therefore introduce quite a number of changes and adaptations.
These changes allow us to give a precise characterization of the rank increase when we add a row and a column, and therefore 
to show that the asymptotic behavior of the rank of a broad class of random matrices, whose non-zero entries are prescribed by the adjacency structure of a sparse Erd\H{o}s-Rényi random graph, over any field $\FF$, is indeed the same as the rank of the simple $0/1$-adjacency matrix of $\bm{G}_{n,d/n}$ over the field $\mathbb{R}$.

This paper is organised as follows: In \Cref{sec_mainr}, we introduce our precise model and main result. A proof overview, together with the most important intermediate steps, can be found in \Cref{sec_prov}. \Cref{sec_mper} collects results on our matrix perturbation. In \Cref{sec_prof}, we investigate various properties of the different variable (or vertex) types introduced earlier, and their relation to the rank. We then derive the fixed point equations for the asymptotic proportions of some of the different types in \Cref{sec_pc}. \Cref{sec_newapp} uses these fixed point equations to derive the desired lower bound on the asymptotic rank. 
In Appendix \ref{app_proal}, we provide important properties of the various functions related to the rank formula. Appendix \ref{app_leaf} explains how to derive an upper bound on the normalised rank from results on the Karp-Sipser leaf-removal algorithm. Finally, Appendix \ref{app_b} contains a proposition which is used to compare different conditional expectations.

\begin{remark}[Notation for random variables]
Throughout the article, we use bold letters to indicate random variables and regular letters to indicate deterministic quantities.
\end{remark} 

\subsection{Main results}\label{sec_mainr}
Let $\FF$ be an arbitrary field and $\FF^\ast:=\FF \setminus\{0\}$ its multiplicative group. For a general matrix $A \in \FF^{m \times n}$, $\rank_\FF(A)$ specifically denotes the rank of $A$ over $\FF$, i.e. the dimension of the linear subspace of $\FF^{n}$ spanned by the columns of $A$. Moreover, we use $\syn$ for the set of all symmetric $n \times n$ matrices with entries in $\FF^\ast$.

In the present article, we study adjacency matrices of sparse Erd\H{o}s-Rényi random graphs with arbitrary non-zero edge weights over $\FF$. 
To define the precise model, let $(J_n)_{n\geq 1}$ be any deterministic sequence of ``template'' matrices such that for all $n \geq 1$, $J_n \in \syn$, and $(\bm{q}(i,j))_{i,j\geq 1}$ be an array of i.i.d.\ uniform random variables in $[0, 1]$. For $p \in [0,1]$, we then define the matrix $\bm{A}_{n,p}$ by setting 
\begin{equation}\label{ec1}
  \bm{A}_{n,p}(i,j)= \begin{cases}
\ensuremath{\mathds{1}}{\cbc{\bm{q}(i,j)<p}}J_n(i,j),&\quad i<j; \\
\ensuremath{\mathds{1}}{\cbc{\bm{q}(j,i)<p}}J_n(j,i),&\quad i>j; \\
0,&\quad i=j. \\
\end{cases}
\end{equation}
$\bm{A}_{n,p}$ can be alternatively regarded as the adjacency matrix of a weighted Erd\H{o}s-Rényi random graph on the vertex set $[n]$, where each potential edge $\{i,j\}$ is present independently with probability $p$. If it is present, it is assigned edge weight $J_n(i,j) = J_n(j,i)$. The construction (\ref{ec1}) also incorporates a natural coupling of the positions of the nonzero entries of the matrices $\bm{A}_{n,p}$ for all choices of $n$ and $p$. 

In the important special case where $J_n(i,j)\equiv 1$ for all $i,j\in\cbc{1,2,\ldots,n}$, $\bm{A}_{n,p}$  coincides with the adjacency matrix of an unweighted Erd\H{o}s-Rényi graph with $n$ vertices and edge probability $p$. An asymptotic rank formula for this model over $\FF=\RR$ in the regime where $p=d/n$ was given by Bordenave, Lelarge and Salez in \cite{bordenave2011rank}: For any $d> 0$, let $\phi_d: [0,1] \to \RR, \phi_d(\alpha):=\exp(d(\alpha-1))$ be the probability generating function of a Poisson random variable with parameter $d$ and 
$R_d\colon [0,1]\to \mathbb{R}$ be defined by setting
    \begin{equation}\label{em0}
    R_d(\alpha)=2-\phi_d\bc{1-\phi_d(\alpha)}-(1+d(1-\alpha))\phi_d(\alpha).
    \end{equation}
 Bordenave, Lelarge and Salez \cite{bordenave2011rank} then show that for any $d>0$, in the coupling given above,
    \begin{equation}\label{e01b}
    \lim_{n\to \infty}\frac{1}{n}\rank_{\RR}\bc{\bm{A}_{n,d/n}}=\min_{\alpha\in [0,1]}R_d(\alpha)\qquad\text{a.s.}
    \end{equation}
The article \cite{bordenave2011rank} also provides asymptotic rank formulas for the adjacency matrices of any sequence of random graphs that converges locally to a rooted Galton-Watson tree whose degree distribution has a finite second moment. 

For general fields $\FF$, of course, $\rank_{\FF}(\bm{A}_{n,d/n})$ need not be identical to $\rank_{\RR}(\bm{A}_{n,d/n})$ (even in the case where  $J_n(i,j)\equiv 1$). For example, if $\FF=\FF_p$ is the finite field with $p$ elements, then generally only the upper bound $\rank_{\FF_p}(\bm{A}_{n,d/n})\leq\rank_{\RR}(\bm{A}_{n,d/n})$ holds true. Moreover, the proof of the rank formula (\ref{e01b}) is based on the rank-nullity theorem and the fact that $\nul_{\FF}(\bm{A}_{n,d/n})$ is identical to the dimension of the eigenspace of $A$ corresponding to $0$. Since for real symmetric matrices, the geometric and algebraic multiplicities of all eigenvalues coincide, the dimension of the eigenspace of $A$ corresponding to $0$ can be studied through an associated spectral measure in this case. On the other hand, for symmetric matrices over $\FF_p$, there is no reason to assume the matrix to be diagonalisable.

Pursuing a purely combinatorial approach that does not rely on the analysis of a spectral measure, our main result generalises the asymptotic rank formula of \cite{bordenave2011rank} to arbitrary fields $\FF$ and general  non-zero entries:

\begin{theorem}\label{t0}
For any $d>0$ and any field $\FF$, $\rank_\FF\bc{\bm{A}_{n,d/n}}/n$ converges in probability to $\min_{\alpha\in [0,1]}R_d(\alpha)$ uniformly in $\bc{J_n}_{n\geq 1}$ in the sense that for any $\varepsilon>0$,
\begin{equation}\label{e01a}
\lim_{n\to\infty}\sup_{J_n\in \syn}\PP\bc{\abs{\frac{1}{n}\rank_\FF\bc{\bm{A}_{n,d/n}}-\min_{\alpha\in [0,1]}R_d(\alpha)}\geq \varepsilon}=0.
\end{equation}
\end{theorem}

\begin{remark}[Almost sure convergence]
In the case where $J_n(i,j)\equiv 1$ and one is interested in convergence of the sequence $(\bm{A}_{n,d/n})_{n \geq 1}$ of adjacency matrices of a sparse Erd\H{o}s-Rényi random graph, the convergence in probability can easily be lifted to almost sure convergence by a standard martingale argument as given in \cite[Appendix 1]{bordenave2011rank}.
\end{remark}

In line with previous results on the rank of sparse random asymmetric matrices \cite{coja2022rank}, \Cref{t0} illustrates that (within the specified framework) the rank formula \cref{e01a} solely depends on $d$, but not on the field $\FF$ or the choice of the sequence $(J_n)_{n\geq 1}$.

\section{Proof overview}\label{sec_prov}
On the following pages, we present an overview of the proof of \Cref{t0}. After fixing some notation, we first reduce the uniform convergence in probability in \cref{e01a} to an upper bound in probability and a lower bound in expectation in \Cref{sec_redu}. While the upper bound is based on the leaf-removal algorithm and the results of \cite{ aronson1998maximum,karp1981maximum}, the lower bound constitutes the main contribution of our article. To lower bound the expected rank of $\bm{A}_{n,d/n}$, we transform it to a ``symmetrised'' matrix and grow the modified matrix from $\varepsilon n$ to $n$ step by step. An essential ingredient in the quantification of the described one-step rank change are the powerful techniques developed in \cite{coja2022rank}, which allow us to focus on the positions of the nonzero entries in the target matrix rather than their precise values. Finally, the rank formula follows by interpreting the sum of the lower bounds as the Riemann sum of an integral, which is analytically tractable.

\subsection{Notation}\label{sec_pre}
This section can be used as a reference for recurring notation that is used throughout the article. 

\paragraph{Sets.}
We write $[\ell] = \cbc{1,2,\ldots,\ell}$ and denote the cardinality of a set $B$ by $|B|$. For two sets $B_1$ and $B_2$, we denote their symmetric difference as 
$B_1 \Delta B_2$ 
and use $\uplus_{i\in I} B_i$ to indicate the union over pairwise disjoint sets $(B_i)_{i\in I}$. If $B$ is a set and $\ell \leq |B|$, we write $\binom{B}{\ell}$ for the collection of $\ell$-subsets of $B$.

\paragraph{Real numbers and fields.} For $a, b \in \RR$, we write $a\vee b = \max\cbc{a,b}$ and $a\wedge b = \min\cbc{a,b}$. $\FF$ is reserved to denote a generic field, and $\FF^\ast = \FF \setminus \{0\}$ its multiplicative group.

\paragraph{Vectors and matrices.} For $A \in \FF^{m \times n}$, we denote its transpose by $A^T$. For a vector $b=\bc{b_1,b_2,\ldots,b_n}\in \mathbb{F}^{1\times n}$, we let ${\rm supp}(b)={\rm supp}(b^T)=\left\{i\in [n]\colon b_i\neq 0\right\}$. We denote by $e_n(i)$ the $i$th standard unit vector in $ \mathbb{F}^{1\times n}$.

For $s=(s_1,s_2,\ldots,s_\ell) \in \RR^{1 \times \ell}$, define $\|s\|_\infty=\sup_{i\in [\ell]}|s_i|$ and $\|s\|_k=(\sum_{i=1}^{\ell}|s_i|^k)^{1/k}$.

For $A \in \FF^{m \times n}$, we denote
\begin{enumerate}[label=(\roman*)]
    \item the $i$th row of $A$ by $A(i,)$ and the $j$th column of $A$ by $A(,j)$.
    
    \item the matrix obtained by removing rows $\ell_1,\ell_2,\ldots,\ell_s$ and columns $\ell_1',\ell_2',\ldots,\ell_t'$ from $A$ by 
    $A\abc{\ell_1,\ell_2,\ldots,\ell_s;\ell_1',\ell_2',\ldots,\ell_t'}$. By a slight abuse of indexing, the $i$th row in the diminished matrix $A\abc{\ell_1,\ell_2,\ldots,\ell_s;\ell_1',\ell_2',\ldots,\ell_t'}$ refers to the row vector $A(i,)\abc{;\ell_1',\ell_2',\ldots,\ell_t'}$, i.e., the \textit{$i$th row of $A$} (minus the entries corresponding to columns $\ell_1',\ell_2',\ldots,\ell_t'$). We use an analogous convention for columns.
\end{enumerate}

\paragraph{Functions.} For a function $f:\Omega \to \RR$, we denote by $f^+$ its positive and by $f^-$ its negative part, i.e. $f^+(x)=0\vee f(x)$ and $f^-(x)=0\vee (-f(x))$ for $x \in \Omega$.

\paragraph{Random variables.} For a finite set $B$, we write $\text{Unif}(B)$ to denote a discrete uniform random variable on $B$, $\bin{n,p}$ to denote a binomial random variable with $n$ trials and success probability $p$ and $\poi{d}$ to denote a Poisson variable with parameter $d$.


For two random variables $X,Y$ taking values in $(\Omega,\mathcal{G})$, we denote the total variation distance between $X$ and $Y$ as
    \[
    {\rm d}_{\sss \rm TV}(X,Y)=\sup_{B\in \mathcal{G}}\abs{\mathbb{P}\bc{X\in B} -\mathbb{P}\bc{Y\in B}}.
    \]

\paragraph{Notions of convergence.}
Throughout the article, 
the order in which limits are taken matters significantly.
For families of real numbers $(a_{n,P,N,J_N})_{n, P,N\in \ZZ^+,J_N\in \syN}$, we write 
\begin{enumerate}[label=(\roman*)]
  \item  \makebox[7em][l]{\mbox{$a_{n,P,N,J_N}=o_n(1)$}} $\qquad \Longleftrightarrow \qquad $ For all $P\geq 1: \quad $ $\lim_{n\to\infty}\sup_{N\geq n,J_N\in \syN}\abs{a_{n,P,N,J_N}}=0$;
  \item \makebox[7em][l]{\mbox{$a_{n,P,N,J_N}=o_{n,P}(1)$}} $ \qquad \Longleftrightarrow \qquad  \limsup_{P\to\infty}\limsup_{n\to\infty}\sup_{N\geq n,J_N\in \syN}\abs{a_{n,P,N,J_N}}=0$.
\end{enumerate}
Given a family of real numbers $(c_{n,P,N,J_N,t})_{n, P,N\in \ZZ^+,J_N\in \syN,t\in [0,d]}$, we say that 
\begin{enumerate}[label=(\roman*)]
  \item  \makebox[18em][l]{\mbox{$c_{n,P,N,J_N,t}=o_n(1)$ uniformly in $t\in [0,d]$}} $\qquad \Longleftrightarrow \qquad $ $\sup_{t\in [0,d]}c_{n,P,N,J_N,t}=o_n(1)$;
  \item \makebox[18em][l]{\mbox{$c_{n,P,N,J_N,t}=o_{n,P}(1)$ uniformly in $t\in [0,d]$}} $ \qquad \Longleftrightarrow \qquad  \sup_{t\in [0,d]}c_{n,P,N,J_N,t}=o_{n,P}(1)$.
\end{enumerate}
For a family of \textit{uniformly bounded} random variables $(\bm{b}_{n,P,N,J_N,t})_{n, P,N\in \ZZ^+,J_N\in \syN, t\in [0,d]}$, we write 
\begin{enumerate}[label=(\roman*)]
  \item $  \bm{b}_{n,P,N,J_N,t}=\oone\qquad \Longleftrightarrow \qquad \mathbb{E}\abs{\bm{b}_{n,P,N,J_N,t}}=o_{n,P}(1)$ uniformly in $t\in [0,d]$;
  \item $\bm{b}_{n,P,N,J_N,t}\geq \oone\qquad \Longleftrightarrow \qquad \bc{\bm{b}_{n,P,N,J_N,t}}^-=\oone$.
  \item $\bm{b}_{n,P,N,J_N,t}\leq \oone\qquad \Longleftrightarrow \qquad \bc{\bm{b}_{n,P,N,J_N,t}}^+ = \oone$.
\end{enumerate}
For a family of events $(\mathfrak B_{n,P,N,J_N,t})_{n, P,N\in \ZZ^+,J_N\in \syN,t\in[0,d]}$, we say that $\mathfrak B_{n,P,N,J_N,t}$ occurs w.h.p. if  $\mathbb{P}\bc{\mathfrak B_{n,P,N,J_N,t}}=1+o_{n,P}(1)$ uniformly in $t\in [0,d]$.

We extend the above notions of convergence to families of numbers and events that only depend on subsets of the parameters. For example, for a family of real numbers $(c_{n,P})_{n, P\in\ZZ^+}$, by treating it as constant on the unspecified parameters, we write $c_{n,P}=o_{n,P}(1)$ whenever $\limsup_{P\to\infty}\lim_{n\to\infty}c_{n,P}=0$.

\subsection{Deduction of \Cref{t0} from suitable upper and lower bounds}\label{sec_redu}
Our main result, \Cref{t0}, is a statement about convergence in probability of the normalised rank sequence $\rk{\bm{A}_{n,d/n}}/n$ that holds uniformly in $(J_n)_{n \geq 1}$. 
In this section, we show how \Cref{t0} readily follows from the following upper bound in probability and the subsequent lower bound in expectation: 

\begin{theorem}[Upper bound in probability]\label{t_upper}
Let $d>0$ and $\FF$ be any field. Then for any $\varepsilon>0$,
\begin{equation}\label{e01a_upper}
\lim_{n\to \infty}\PP\bc{\sup_{J_n\in\syn} \frac{\rk{\bm{A}_{n,d/n}}}{n}\leq \min_{\alpha\in[0,1]}R_d(\alpha) +\varepsilon}=1.
    \end{equation}
\end{theorem}

\begin{theorem}[Lower bound in expectation]
\label{t1}
For any $d>0$ and any field $\FF$, \begin{equation}\label{e01}
    \liminf_{n\to \infty}\inf_{J_n\in \syn}\mathbb{E}\brk{\frac{\rk{\bm{A}_{n,d/n}}}{n}}\geq \min_{\alpha\in [0,1]}R_d(\alpha).
    \end{equation}
\end{theorem}

While \Cref{t_upper} straightforwardly follows from the fact that  the nullity of an adjacency matrix remains invariant under ``leaf-removal'' (see \cite{bauer2001exactly}) and the results of \cite{karp1981maximum}\footnote{See Appendix \ref{app_leaf}.}, the derivation of \Cref{t1} is the main contribution of our work. The central steps towards \cref{e01} are laid out in the remainder of \Cref{sec_prov}.
With \Cref{t1,t_upper} in hand, we are in the position to prove \Cref{t0}:

\begin{proof}[Proof of \Cref{t0} subject to \Cref{t1,t_upper}]
Let 
\[\bm{s}_n=\bm{s}_n(J_n)=\frac{\rk{\bm{A}_{n,d/n}}}{n}- \min_{\alpha\in[0,1]}R_d(\alpha).\]
Then $\abs{\bm{s}_n}\leq 1+\abs{\min_{\alpha\in[0,1]}R_d(\alpha)}$. By \Cref{t_upper}, for any $\varepsilon>0$, 
\begin{align*}
  >0   \limsup_{n\to \infty}\sup_{J_n\in\syn}\Erw\brk{\bm{s}_n^+}&\leq\limsup_{n\to \infty}\Erw\brk{\sup_{J_n\in\syn}\bm{s}_n^+}\\
    &\leq \varepsilon+\bc{1+\abs{\min_{\alpha\in[0,1]}R_d(\alpha)}}\limsup_{n\to \infty}\PP\bc{\sup_{J_n\in\syn} \bm{s}_n^+\geq \varepsilon}=\varepsilon.
\end{align*}
Since $\varepsilon$ can be chosen arbitrarily small, we conclude that $\limsup_{n\to \infty}\sup_{J_n\in\syn}\Erw\brk{\bm{s}_n^+}=0$.
On the other hand, by \Cref{t1}, $\liminf_{n\to\infty}\inf_{J_n\in \syn}\Erw\brk{\bm{s}_n}\geq 0$. Since $\bm{s}_n=\bm{s}_n^+-\bm{s}_n^-$,
\[\limsup_{n\to\infty}\sup_{J_n\in \syn}\Erw\brk{\bm{s}_n^-}\leq\limsup_{n\to\infty}\sup_{J_n\in \syn}\Erw\brk{\bm{s}_n^+}-\liminf_{n\to\infty}\inf_{J_n\in \syn}\Erw\brk{\bm{s}_n}\leq 0.\]
As a consequence, $\limsup_{n\to\infty}\sup_{J_n\in \syn}\Erw\brk{\abs{\bm{s}_n}}=0$. The uniform convergence in probability now follows from Markov's inequality. 
\end{proof}


We conclude that it remains to prove \Cref{t1} and outline the main steps in the following subsections.

\subsection{The lower bound: Building the matrix}\label{sec_ass}
Instead of proving \Cref{t1} for the sequence $(\bm{A}_{n,d/n})_{n\geq 1}$ directly, we work with a ``symmetrised'' version that possesses a suitable form of joint row and column exchangeability. To define the auxiliary matrices, fix a number $N\in \NN_{\geq 1}$ and let $\vtu$ be a uniform permutation of $[N]$. For $n \in [N]$, define the matrix $\bm{T}_{n,p}^{(N)} \in \FF^{n \times n}$ by setting 
\begin{equation}\label{ec1_2}
  \bm{T}_{n,p}^{(N)}(i,j)= \begin{cases}
\ensuremath{\mathds{1}}{\cbc{\bm{q}(\vtu(i),\vtu(j))<p}}J_N(\vtu(i),\vtu(j)),&\quad i<j; \\
\ensuremath{\mathds{1}}{\cbc{\bm{q}(\vtu(j),\vtu(i))<p}}J_N(\vtu(j),\vtu(i)),&\quad i>j; \\
0,&\quad i=j. \\
\end{cases}
\end{equation}
For any $N \in \NN_{\geq 1}$, this construction yields $N$ matrices $\bm{T}_{1,p}^{(N)}, \bm{T}_{2,p}^{(N)}, \ldots, \bm{T}_{N,p}^{(N)}$ of growing dimension. Specifically, we have $\bm{T}_{N,p}^{(N)}(i,j)=\bm{A}_{N,p}(\vtu(i),\vtu(j))$ and 
$\rk{\bm{T}_{N,p}^{(N)}}=\rk{\bm{A}_{N,p}}$, so that 
\Cref{t1} would follow from the lower bound
\[\liminf_{n\to \infty}\inf_{J_n\in \syn}\mathbb{E}\brk{\frac{1}{n}\rk{\bm{T}_{n,d/n}^{(n)}}}\geq \min_{\alpha\in [0,1]}R_d(\alpha).\]
However, for technical reasons that will become apparent later, we actually show the stronger statement
\[\liminf_{n\to \infty}\inf_{N \geq n}\inf_{J_N\in \syN}\mathbb{E}\brk{\frac{1}{n}\rk{\bm{T}_{n,d/n}^{(N)}}}\geq \min_{\alpha\in [0,1]}R_d(\alpha).\]

Correspondingly, in the following, we focus on the derivation of a lower bound on $\Erw[\text{rk}_{\FF}(\bm{T}_{n,d/n}^{(N)})]/n$ for $N \geq n$. Nonetheless, for a lighter notation, we omit the superscript $N$ in the matrices below. 
The basic idea in this derivation is rather simple: Fix a small number $\varepsilon \in (0,1)$ and trace the rank change when the matrix  $\bm{T}_{\varepsilon n,d/n}$ is grown to $\bm{T}_{n,d/n}$ step by step. Then, by a telescoping sum,
\begin{align}
    \frac{1}{n}\Erw\brk{\rk{\bm{T}_{n,d/n}}}
    \geq &\frac{1}{n}\sum_{m=\varepsilon n}^{n-1}\bc{\Erw\brk{\rk{\bm{T}_{m+1,d/n}}}-\Erw\brk{\rk{\bm{T}_{m,d/n}}}}.\label{eq_rkdifsum0}
\end{align}
The last expression thus reduces the problem of lower bounding $\Erw[\rk{\bm{T}_{n,d/n}}]/n$ to lower bounding $\sum_{m=\varepsilon n}^{n-1}\bc{\Erw\brk{\rk{\bm{T}_{m+1,d/n}}}-\Erw\brk{\rk{\bm{T}_{m,d/n}}}}/n$. 

Since this bound is based on a comparison of the two matrices $\bm{T}_{m+1,d/n}$ and $\bm{T}_{m,d/n}$ whose sizes differ by one, our approach might superficially resemble the Aizenman-Sims-Starr scheme from mathematical physics, which had previously found its application in the study of the rank of random matrices in \cite{coja2022rank}. The Aizenman-Sims-Starr scheme, whose basic idea is to compare a system of $n$ variables to a system of $n+1$ variables and to study the influence on the $(n+1)$st variable, has originally been developed to tackle the Sherrington-Kirkpatrick spin glass model \cite{aizenman2003extended}. However, our approach cannot straightforwardly be interpreted as a cavity computation for the original matrix sequence, since we do not (directly or indirectly) compare two matrices of the form $\bm{A}_{n,d/n}$ and $\bm{A}_{n-1,d/(n-1)}$. Instead, we compare matrices $\bm{T}_{m,d/n}$ and $\bm{T}_{m+1,d/n}$ whose sizes differ by one, but who are of a purely auxiliary nature and do not represent copies of the original matrix model.

\subsection{Taming linear relations} \label{sec_LR}
While a comparison of the rather similar matrices $\bm{T}_{m+1,d/n}$ and $\bm{T}_{m,d/n}$ might look innocuous at first glance, obtaining good control over the ensuing rank change is not a simple task, since it requires detailed knowledge of the intricate linear dependencies of the matrix $\bm{T}_{m,d/n}$. Following and extending core ideas of \cite{coja2022rank}, this section collects the main tools that are necessary to deal with these relations and to accurately describe the change in rank from $\bm{T}_{m,d/n}$ to $\bm{T}_{m+1,d/n}$.

The following definition from \cite{coja2022rank} contains a collection of terminology that will turn out useful in the coming considerations on linear dependencies.
\begin{definition}[Linear relations: {\cite[Definition 2.1]{coja2022rank}}]\label{d2}
  Let $A \in \FF^{m \times n}$.
\begin{enumerate}[label=(\roman*)]
  \item A set $\emptyset \not= I\subseteq [n]$ is a \textbf{relation} of $A$ if there exists a row vector $y\in \mathbb{F}^{1\times m}$ such that $\emptyset\neq{\rm supp}(y A) \subseteq I$. If furthermore ${\rm supp}(y A) = I$, then we call $y$ a \textbf{representation} of $I$ in $A$.
  \item If $I=\cbc{i}$ is a relation of $A$, then we call $i$ a \textbf{frozen variable} in $A$. Let $\mathcal{F}(A)$ be the set of all frozen variables. 
  \item A relation $I \subseteq [n]$ is a \textbf{proper relation} of $A$ if $I\backslash \mathcal{F}(A)$ is a relation of $A$.
  \item For $\delta>0,\ell\geq 2$, we say that $A$ is \textbf{$(\delta,\ell)$-free} if there are no more than $\delta n^\ell$ proper relations $I \subseteq [n]$ of size $\abs{I}=\ell$.
\end{enumerate}
\end{definition}

\begin{remark}[Frozen variables]\label{rem_froz}
\begin{enumerate}[label=(\roman*)]
\item The terminology \textit{frozen variable} refers to the role that the corresponding coordinate plays in the kernel of $A$: 
Frozen variables are exactly those coordinates that are invariably $0$ in all vectors of $\ker_{\FF}(A)$ (see \cite[Fact 2.2]{coja2022rank}). 
\item In \Cref{p4} (also in \cite[Lemma 4.7]{demichele2021rank}), we will see yet another convenient characterization of frozen variables in terms of column removal as follows:
\begin{equation}\label{adf}
i\in \mathcal{F}(A) \Longleftrightarrow  \rank_{\FF}\bc{A}-\rank_{\FF}\bc{A\abc{;i}}=1 .
\end{equation}
\end{enumerate}
\end{remark}

Let $A \in \FF^{m \times n}$ be any matrix and $b \in \FF^{1 \times n}$ be a non-zero row vector, and suppose that we want to attach $b$ to $A$ and characterise the ensuing rank change. This is a simpler operation than what we actually need (attaching both a row and a column), but still instructive. In terms of frozen variables and proper relations, we can say the following about the rank increase of attaching $b$ to $A$: If all variables of $\supp{b}$ are frozen, then surely $b$ lies in the linear span of the rows of $A$, since it can be linearly combined using the representations of its non-zero coordinates. On the other hand, if $b$ is contained in the linear span of the rows of $A$, then because of the existence of a linear combination, either all variables of $\supp{b}$ are frozen or they form a proper relation of $A$. As a consequence, we have the following key implications:
\begin{equation}\label{eq00}
\begin{aligned}
  \supp{b}\subseteq \mathcal{F}(A)\quad\Longrightarrow\quad& \text{$b$ is in the span of the rows of $A$}\\
    \Longrightarrow\quad  & \text{$\supp{b}\subseteq \mathcal{F}(A)$ or $\supp{b}$ is a proper relation of $A$.}
    \end{aligned}
\end{equation}

These implications are useful for our purposes since 
the concept of a relation only 
takes into account the locations of non-zero entries, but not their 
entries. 
However, unfortunately, \cref{eq00} does not come in form of an equivalence, 
since $\supp{b}$ being a proper relation of $A$ does not imply that $b$ lies in the span of the rows of $A$. 

To remedy this issue, based on ideas from \cite{coja2022rank}, we use a matrix perturbation that greatly reduces the overall number of short proper linear relations in the resulting matrix, such that morally, an equivalence of the form ``$\supp{b}\subseteq \mathcal{F}(A) \Longleftrightarrow b$ is in the span of the rows of $A$'' holds. 
While the perturbation from \cite{coja2022rank} is based on the attachment of unit rows, we will augment this definition by the attachment of unit columns to account for the symmetry of our matrices. The basic idea is that the attachment of unit rows at the bottom of a given matrix $A$ can eliminate short proper relations in the augmented matrix, while the attachment of unit columns to the left of $A$ can eliminate short proper relations in its transpose. 

The details of the perturbation are considerably more subtle. We split its definition into two main parts, since it involves two stages of randomness. In the first definition, we present the basic row and column attachment matrices. Their non-zero entries may be confined to fixed \textit{initial segments} of the column set $[n]$ and row set $[m]$, respectively:

\begin{definition}[Perturbation matrices]\label{d1}
\begin{enumerate}[label=(\roman*)]
\item Let $\theta_r, n_1, n_2 \in \mathbb{N}$ with $n_1 \leq n_2$. The \textbf{row-perturbation matrix} $\THETA_r[\theta_r, n_1| n_2] \in \{0,1\}^{\theta_r \times n_2}$ with parameters $\theta_r, n_1, n_2$ is defined by setting exactly one entry in each of its $\theta_r$ rows equal to $1$, where the choice of this entry is uniform among the first $n_1$ out of its $n_2$ columns. More precisely, the unique $1$-entry of row $k \in [\theta_r]$ is in column $\vj_k$, where  $\vj_1, \ldots, \vj_{\theta_r} \in [n_1]$ are i.i.d.\ uniformly distributed random variables.
\item Let $\theta_c, m_1, m_2 \in \mathbb{N}$ with $m_1 \leq m_2$. The \textbf{column-perturbation matrix} $\THETA_c[m_1\vert m_2, \theta_c] \in \{0,1\}^{m_2 \times \theta_c}$ with parameters $\theta_c, m_1, m_2$ is defined by setting exactly one entry in each of its $\theta_c$ columns equal to $1$, where the choice of this entry is uniform among the first $m_1$ out of its $m_2$ rows.  More precisely, the unique $1$-entry of column $k \in [\theta_c]$ is in row $\vi_k$, where $\vi_1, \ldots, \vi_{\theta_c} \in [m_1]$ are i.i.d.\ uniformly distributed random variables.
\end{enumerate}
\end{definition}



\begin{figure}[H]
\centering
\includegraphics[width = 0.5\textwidth]{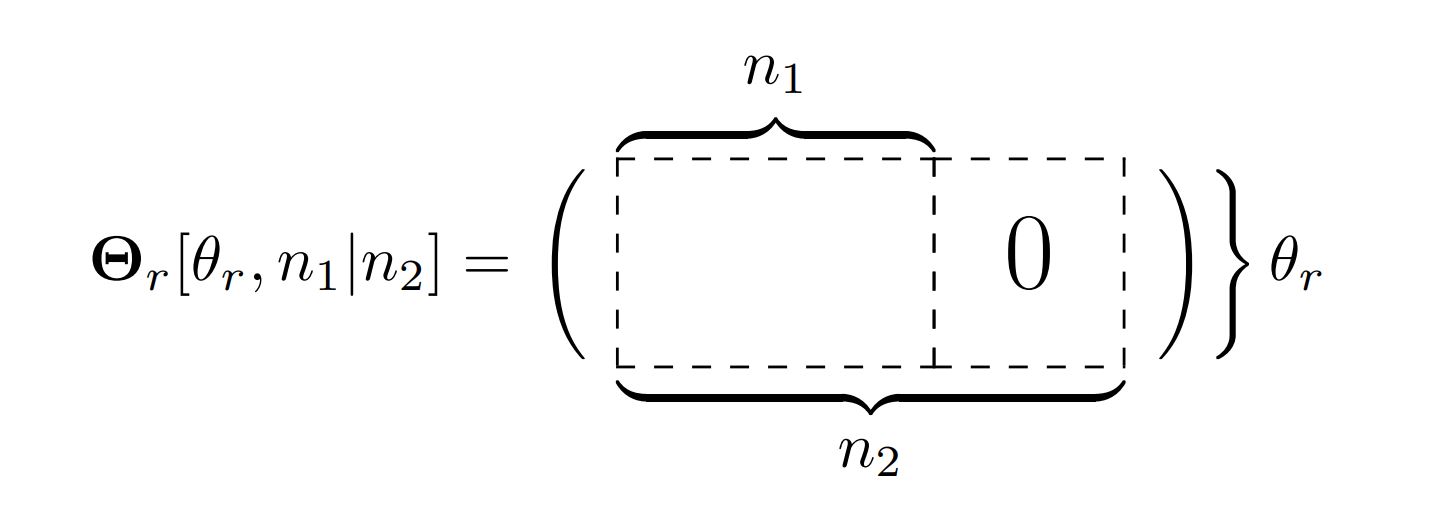}
\caption{Schematic representation of the row-perturbation matrix $\bm{\Theta}_r[\theta_r, n_1| n_2]$.}
\end{figure}

For $A \in \FF^{m \times n}$ and a row perturbation matrix $\THETA_r[\theta_r, n_1|n]$ with non-zero column-coordinates $\vj_1, \ldots, \vj_{\theta_r} \in [n_1]$ of its $\theta_r$ rows, consider the perturbed matrix
$$ \bm{A}' := \begin{pmatrix}
    A \\
    \THETA_r[\theta_r, n_1| n]
    \end{pmatrix}.$$
Then in $\bm{A}'$, the non-zero columns $\vj_1, \ldots, \vj_{\theta_r} \in [n_1]$ of $\THETA_r[\theta_r, n_1|n]$ are part of the set of frozen variables: Since $\vj_s$ is the index of the only non-zero entry in the $(m+s)$th row, the Boolean row vector $e_{m+\theta_r}(m+s)$ is a representation of $\vj_s$ in $\bm{A}'$. In this sense, one can view the attachment of $\THETA_r[\theta_r, n_1| n]$ at the bottom of a matrix as explicitly freezing the variables corresponding to non-zero columns. 

On the other hand, appending $\THETA_c[m_1\vert m, \theta_c]$ to the right of $A$ has quite a contrary and more subtle effect upon the set of frozen variables: In a sense, additional columns have the same impact as row removals and therefore can ``unfreeze'' coordinates (see \Cref{l8.4} for a proof). The necessity of column perturbation matrices constitutes the main difference to the previously employed perturbation from \cite{coja2022rank}.

Before we introduce a second level of randomness to the perturbation, in the next lemma, we construct a coupling of the row-perturbation matrices $\THETA_r[\theta_r,n_1|n_2]$ for all possible sizes $\theta_r \times n_2$ and subsets of \textit{freezable} coordinates $[n_1]\subseteq [n_2]$. The benefit of this coupling is twofold. First, perturbation matrices of increasing size, but with fixed subset of freezable coordinates, will be nested. Second, the probability that matrices of fixed dimension, but with different subsets of freezable coordinates, disagree, can be bounded explicitly. This coupling ensures that with high probability, we can apply the same perturbation to both $\bm{T}_{m,d/n}$ and $\bm{T}_{m+1,d/n}$, and still get the desired properties:

\begin{lemma}[Coupling of perturbation matrices]\label{lc1}
There is a coupling of the family $\{\THETA_r[\theta_r, n_1| n_2]: \theta_r, n_2 \geq 1, n_1\in [n_2]\}$ with the following properties:
\begin{enumerate}[label=(\roman*)]
\item  For any $\theta_r, n_2 \geq 1$ and $n_1 \in [n_2]$, $\THETA_r[\theta_r,n_1|n_2+1]\abc{;n_2+1}=\THETA_r[\theta_r,n_1|n_2]$.
\item  For any $\theta_r, n_2 \geq 1$ and $n_1 \in [n_2]$, $\THETA_r[\theta_r+1,n_1|n_2]\abc{\theta_r+1;}=\THETA_r[\theta_r,n_1|n_2]$.
\item  For any $\theta_r, n_2 \geq 1$ and $n_0 \leq n_1 \leq n_2$, $\mathbb{P}\left(\THETA_r[\theta_r,n_0|n_2]= \THETA_r[\theta_r,n_1|n_2]\right)=(n_0/n_1)^{\theta_r}$.
\end{enumerate}
Similarly, there is a coupling with analogous properties for the family $\{\THETA_c[m_1| m_2,\theta_c]: \theta_c, m_2 \geq 1, m_1\in [m_2]\}$. 
\end{lemma}

\begin{remark}
From now on, we assume that the perturbation families $\{\THETA_r[\theta_r, n_1| n_2]: \theta_r, n_2 \geq 1, n_1\in [n_2]\}$ and $\{\THETA_c[m_1| m_2,\theta_c]: \theta_c, m_2 \geq 1, m_1\in [m_2]\}$ are coupled as in \Cref{lc1} and independent of each other. 
\end{remark}

\Cref{lc1} is proved in \Cref{sec_proof_lc1}. Based on the ensembles $\{\THETA_r[\theta_r, n_1| n_2]: \theta_r, n_2 \geq 1, n_1\in [n_2]\}$ and $\{\THETA_c[m_1| m_2,\theta_c]: \theta_c, m_2 \geq 1, m_1\in [m_2]\}$ from \Cref{lc1}, we finally introduce the central perturbation of this article:

\begin{definition}[Canonical perturbation]\label{d34}
%
%
 For $A \in \FF^{m \times n}$ and $\theta = (\theta_r,\theta_c) \in \NN^2$, we write 
$$A[\theta]=\begin{pmatrix}A&\THETA_c[m|m, \theta_c]\\\THETA_r[\theta_r, n \vert n] & 0_{\theta_r\times \theta_c} \end{pmatrix}.$$
For the canonical choice $\bm{\theta} =(\bm{\theta}_r,\bm{\theta}_c) \sim \text{Unif}([P]^2)$, where $P \in \NN$ is fixed and $\bm \theta$ is independent of the couplings $\{\THETA_r[\theta_r, n_1| n_2]: \theta_r, n_2 \geq 1, n_1\in [n_2]\}$ and $\{\THETA_c[m_1| m_2,\theta_c]: \theta_c, m_2 \geq 1, m_1\in [m_2]\}$, we simply write $A[\bm \theta]$.
\end{definition}

\begin{remark}
In the rest of this paper, $\bm{\theta}$ always denotes a random vector chosen uniformly at random from $[P]^2$. It is important to keep in mind that the random vector $\bm{\theta}$ is always understood to depend on the parameter $P$, even though this is omitted from the notation (in line with the notation in \cite{coja2022rank}). 
\end{remark}
    

As advertised earlier, perturbation typically greatly reduces the number of short proper relations. The next proposition shows that, for any fixed $L\in\NN_{\geq 2}$,  $A[\bm{\theta}]$ and $A[\bm{\theta}]^T$ are w.h.p.  $(\delta,\ell)$-free for all $2\leq \ell\leq L$ (observe that even if $A$ is symmetric, the perturbed matrix $A[\bm\theta]$ generally is not). This makes the matrices $A[\bm{\theta}]$ and $A[\bm{\theta}]^T$ {\em much} more convenient to study in comparison to $A$: 

\begin{proposition}[Perturbation eliminates most short proper relations]\label{p1}
Fix $\delta>0, L\in\NN_{\geq 2}$ and $s\in \ZZ$. Then 
    \begin{equation}\label{e1_rep}
    \sup_{A\in \FF^{(n+s) \times n}}\mathbb{P}\left( \text{$A[\bm{\theta}]$ or $A[\bm{\theta}]^T$ is not $(\delta,\ell)$-free for some $2\leq \ell\leq L$}\right) = o_{n,P}(1).
    \end{equation}
\end{proposition}
The proof of \Cref{p1} is given in \Cref{sec_proof_p1}.  \Cref{p1} is the symmetric version of \cite[Proposition 2.3]{coja2022rank}. It is remarkable in the sense that it shows that the simple perturbation of attaching a bounded number of unit rows and columns eliminates a large proportion of short proper relations both column- and row-wise. 

Rather than lower bounding $\Erw\brk{\rk{\bm{T}_{n,d/n}}}/n$ as indicated in (\ref{eq_rkdifsum0}), in the next subsections, we will outline how to lower bound the expected normalised rank of the perturbed matrix $\bm{T}_{n,d/n} [\bm{\theta}]$. This also gives a lower bound for $\Erw\brk{\rk{\bm{T}_{n,d/n}}}/n$, 
since if we add a row or a column to a matrix, its rank stays unchanged or increases by $1$, and therefore,
    \begin{equation}\label{ee1}
    \rank(A)\leq \rank(A[\bm\theta]) \leq \rank(A)+\bm\theta_r+\bm\theta_c.
    \end{equation}
Thus, as long as $\bm \theta_r, \bm \theta_c$ are bounded random variables, all results on the asymptotic rank of the perturbed matrices transfer to the unperturbed ones.

\subsection{Rank increase for the perturbed matrix and obstructions due to symmetry}\label{sec_rifpm}
At this point, our strategy rests on lower bounding the differences
\begin{align*}
     \Erw\brk{\rk{\bm{T}_{m+1,d/n} [\bm{\theta}]}}-\Erw\brk{\rk{\bm{T}_{m,d/n}[\bm{\theta}]}}
\end{align*}
for $m \geq \eps n$. Since $m$ grows linearly in $n$, a reparametrisation yields the more convenient expression
\begin{align}\label{eq_rank_difference}
     \Erw\brk{\rk{\bm{T}_{n+1,t/n} [\bm{\theta}]}}-\Erw\brk{\rk{\bm{T}_{n,t/n}[\bm{\theta}]}},
\end{align}
where now $t \in [\varepsilon d,d]$.
While all the perturbed matrices use the same vector $\bm{\theta}$ that fixes the dimensions of the perturbation, the positions of the non-zero entries in the perturbation part may change from matrix to matrix. Conveniently, this does not happen frequently, since thanks to the coupling from \Cref{lc1}, 
with high probability,
    \begin{equation}\label{e001}
    \bm{T}_{n+1,t/n} [\bm{\theta}]\abc{n+1;n+1}=\bm{T}_{n,t/n} [\bm{\theta}].
    \end{equation}
On the event \cref{e001}, observation \cref{adf} on column removal and frozen variables implies that
\begin{align}\label{eq_rank_diff}
&\rk{\bm{T}_{n+1,t/n} [\bm{\theta}]} - \rk{\bm{T}_{n,t/n} [\bm{\theta}]}\\
= &\ind\{n+1\in \mathcal{F}(\bm{T}_{n+1,t/n} [\bm{\theta}]^T)\}+\ind\{n+1\in \mathcal{F}(\bm{T}_{n+1,t/n} [\bm{\theta}]\abc{n+1;})\},\nonumber
\end{align}
where we first remove the $(n+1)$st row and then  the $(n+1)$st column to go from $\bm{T}_{n+1,t/n} [\bm{\theta}]$ to $\bm{T}_{n,t/n} [\bm{\theta}]$. 
In the analysis of \cref{eq_rank_diff}, both the benefits of working with the matrix $\bm{T}_{n,t/n}$ and then its  perturbation $\bm{T}_{n,t/n}[\bm{\theta}]$ become apparent.
We next explain how these ideas can be used effectively in the evaluation of the r.h.s. of \cref{eq_rank_diff}.



For $p \in (0,1)$, let $\bm{\alpha}_{n,p}$ and $\hval_{n,p}$ be the proportions of frozen variables $i\in[n]$ in $\bm{T}_{n,p}[\bm{\theta}]$ and $\bm{T}_{n,p}[\bm{\theta}]^T$, respectively. By the distributional invariance of the matrix $\bm{T}_{n+1,t/n} [\bm{\theta}]$ under joint row- and column-relabelling\footnote{For the precise arguments, see \Cref{sec_exchange}.}, conditionally on $\hval_{n+1,t/n}$,  the probability that $n+1$ is frozen in $\bm{T}_{n+1,t/n} [\bm{\theta}]^T$ is simply given by 
\begin{align*}
\mathbb{P}(n+1\in \mathcal{F}(\bm{T}_{n+1,t/n} [\bm{\theta}]^T)\mid\hval_{n+1,t/n}) = \hval_{n+1,t/n}.
\end{align*}

This provides a simple expression for the first indicator in r.h.s. of \cref{eq_rank_diff}. We next consider the second indicator that $n+1$ is frozen in $\bm{T}_{n+1,t/n} [\bm{\theta}]\abc{n+1;}$. Again by observation \cref{adf}, this event is the same as the event the $(n+1)$st column of $\bm{T}_{n+1,t/n}[\bm{\theta}]\abc{n+1;}$  lies in the span of the columns of $\bm{T}_{n,t/n}[\bm{\theta}]$. Considering the transposed matrix, this translates to the event that the $(n+1)$st row of $\bm{T}_{n+1,t/n}[\bm{\theta}]\abc{n+1;}^T$  lies in the span of the rows of $\bm{T}_{n,t/n}[\bm{\theta}]^T$. By this chain of equivalences, we have turned the original event into one that we can handle very well
thanks to \cref{eq00} and the perturbation: Since the perturbation effectively excludes the possibility that the non-zero components of the $(n+1)$st row of $\bm{T}_{n+1,t/n}[\bm{\theta}]\abc{n+1;}^T$ form a proper relation, the event in question roughly corresponds to the event that all the non-zero components of the $(n+1)$st row of $\bm{T}_{n+1,t/n}[\bm{\theta}]\abc{n+1;}^T$ are frozen in $\bm{T}_{n,t/n}[\bm{\theta}]^T$. 

For a lighter notation, we abbreviate $\bm{b}:=\bm{T}_{n+1,t/n}[\bm \theta](n+1,)$, so that $\bm{b}$ is the $(n+1)$st row of $\bm{T}_{n+1,t/n}[\bm \theta]$. 
Since the positions of the non-zero entries of $\bm b$ are chosen uniformly at random and independently of $\bm{T}_{n,t/n}[\bm{\theta}]$, conditionally on $ {\bm \alpha}_{n,t/n}$ and $|\supp{\bm b}|$, the probability that all the non-zero components of the $(n+1)$st row of $\bm{T}_{n+1,t/n}[\bm{\theta}]\abc{n+1;}^T$ are frozen in $\bm{T}_{n,t/n}[\bm{\theta}]^T$ should be close to
$ 1 - {\bm \alpha}_{n,t/n}^{|\supp{\bm b}|}$. On the other hand, $|\supp{\bm b}|$ asymptotically follows a $\poi{t}$ distribution, so that after taking expectation with respect to $|\supp{\bm b}|$, we arrive at the approximation
\begin{align*}
\mathbb{P}(n+1\in \mathcal{F}(\bm{T}_{n+1,t/n} [\bm{\theta}]\abc{n+1;})|\hval_{n,t/n}) \approx 1-\phi_t(\hval_{n,t/n}).
\end{align*}
In the above, recall that $\phi_t$ is the probability generating function of a $\poi{t}$ variable.
 Thus, on a heuristic level,
\begin{align}\label{heur_alpha_eq}
\Erw\brk{\rk{\bm{T}_{n+1,t/n} [\bm{\theta}]}\vert \hval_{n+1,t/n}} - \Erw\brk{\rk{\bm{T}_{n,t/n} [\bm{\theta}]}\vert \hval_{n,t/n}} \approx \hval_{n+1,t/n} + 1 - \phi_t\bc{\hval_{n,t/n}}. 
\end{align}
This expression has two flaws: First of all, rather than depending on one random variable, it depends on both $\hval_{n,t/n}$ and $\hval_{n+1,t/n}$. Secondly, even though we trace the rank change in $\bm{T}_{n,t/n} [\bm{\theta}]$, the left hand side of \cref{heur_alpha_eq} comes in terms of the proportions in the transposed matrices. Fortunately, in \Cref{sec_prof}, we will show that in expectation, the difference $\val_{n+1,t/n} - \val_{n,t/n}$ is small, which allows us to reduce the r.h.s. of \cref{heur_alpha_eq} to one parameter. On the other hand, $\hval_{n+1,t/n}$ and $\val_{n+1,t/n}$ are identically distributed, so the second problem is solved as well. With $h_t:[0,1] \to \RR$,
\begin{align}\label{def_ht}
h_t\bc{\alpha}:=\alpha+1-\phi_t\bc{\alpha},
\end{align}
we have thus heuristically derived the following result:
\begin{restatable}[The rank increase]{proposition}{randif}\label{lem_rkdif_rank}
For any $d>0$, 
\begin{equation}\label{eq_rkdif_rank}
\begin{aligned}
    \Erw\brk{\rk{\bm{T}_{n+1,t/n}[\bm{\theta}]}-\rk{\bm{T}_{n,t/n}[\bm{\theta}]}}
    =\Erw\brk{h_t\bc{\val_{n,t/n}}}+o_{n,P}(1),\ \mbox{uniformly in $t\in[0,d]$.}
\end{aligned}\end{equation}
\end{restatable}
We give a full proof of \cref{eq_rkdif_rank} in \Cref{lem_rkdif_rank} in \Cref{sec_newapp}.
\Cref{lem_rkdif_rank} lays the basis for the targeted lower bound on $\Erw[\rank_{\FF}(\bm{T}_{n,d/n})]/n$. In view of the rank formula \cref{e01b}, it might be tempting to just take the minimum over all $\alpha \in [0,1]$ on the r.h.s. of \cref{eq_rkdif_rank}. Unfortunately, this is not sufficient to arrive at \cref{e01b}, and we need means to restrict the potential values of $\val_{n,t/n}$. 


It thus ``only'' remains to get our hands on $\val_{n,t/n}$. With \Cref{eq_rkdif_rank} in mind, it is natural to suspect that $\val_{n,t/n}$ converges and to try to  calculate its limit. However, the situation is not that simple, and based on results for a similar class of asymmetric sparse matrices \cite{coja2022sparse}, it is not reasonable to expect $\val_{n,t/n}$ to stabilise. Instead, our strategy will be to
derive an asymptotic fixed point equation for $\val_{n,t/n}$. 
The ensuing characterisation will finally allow us to make the connection to the rank formula \cref{e01b}. 

To motivate the desired equation for $\val_{n,t/n}$, 
we again take a look at the evaluation of the second indicator in the derivation of \cref{heur_alpha_eq} above:
\begin{align*}
   \mathbb{P}(n+1\in \mathcal{F}(\bm{T}_{n+1,t/n} [\bm{\theta}]\abc{n+1;})|\val_{n,t/n},\hval_{n,t/n}) \approx  1-\phi_t(\hval_{n,t/n}).
\end{align*}
Since the matrix $\bm{T}_{n+1,t/n} [\bm{\theta}]\abc{n+1;}$ is rather similar to $\bm{T}_{n,t/n} [\bm{\theta}]$, one might make the bold assumption that
\begin{align*}
\mathbb{P}(n+1\in \mathcal{F}(\bm{T}_{n+1,t/n} [\bm{\theta}]\abc{n+1;})|\val_{n,t/n},\hval_{n,t/n}) \approx \mathbb{P}(n+1\in \mathcal{F}(\bm{T}_{n+1,t/n} [\bm{\theta}])|\val_{n,t/n},\hval_{n,t/n}).
\end{align*}
On the other hand,
\begin{align*}
\mathbb{P}(n+1\in \mathcal{F}(\bm{T}_{n+1,t/n} [\bm{\theta}])|\val_{n+1,t/n},\hval_{n+1,t/n}) \approx \val_{n+1,t/n}. 
\end{align*}
Based on the previous assumption, we can again argue that $\val_{n+1,t/n}\approx \val_{n,t/n}$ and use a handy proposition on the comparison of conditional expectations\footnote{See \Cref{pexp} in the appendix.} to conclude that 
$$ 1-\phi_t(\hval_{n,t/n}) \approx \val_{n,t/n}.$$
Along the same lines, we can conclude that $1-\phi_t(\val_{n,t/n}) \approx \hval_{n,t/n}$. Combining the two approximations, we heuristically deduce that $\val_{n,t/n}$ should approximately satisfy the equation 
\begin{align}\label{eq_heuristic_fix}
\val_{n,t/n}\approx 1-\phi_t(1-\phi_t(\val_{n,t/n})).
\end{align}
While \cref{eq_heuristic_fix} is surely based on a plausible line of arguments, crucially, the very first step in its derivation might have been too bold. Indeed, this approximation was in essence based on the assumption that w.h.p., for any fixed $i \in [n]$,
\begin{align}\label{conj_fra}
    \mbox{$i\not\in \mathcal{F}(\bm{T}_{n,t/n} [\bm{\theta}]\abc{i;})\Delta \mathcal{F}(\bm{T}_{n,t/n} [\bm{\theta}])$.}
\end{align}

Does \cref{conj_fra} hold w.h.p.? We believe so\footnote{Our belief is underpinned by the fact that removal of row $i$ has the same effect as attachment of a unit column (see \Cref{l8.4}), which is akin to a pinning operation.}. Sadly, we cannot prove it, and therefore \cref{eq_heuristic_fix} is just a conjecture at this point. Nevertheless, the heuristic approximation illustrates the pivotal role of events of the form \cref{conj_fra} for symmetric matrices, which motivates a more fine-grained description of frozen variables as introduced in the following section. This description will finally allow us to find another, more indirect route towards \cref{eq_heuristic_fix}, while still, the belief in \cref{conj_fra} lies at the heart of the argument.

\subsection{Frozen variables revisited} \label{sec_frozen_rev}
As discussed in \Cref{sec_rifpm}, we cannot prove that w.h.p., removal of row $i$ from $\bm{T}_{n,t/n}[\bm \theta]$ does not unfreeze $i$. To keep track of those ``problematic'' variables where removal of row $i$ unfreezes variable $i$, we now give a name to them:
\begin{definition}[Frailly, firmly and completely frozen variables]\label{dsc}
For any matrix $A\in \mathbb{F}^{m\times n}$ and $i \in [m\wedge n]$, we say that
\begin{enumerate}[label=(\roman*)]
  \item $i$ is \textbf{ frailly frozen} in $A$ if $i\in\mathcal{F}(A)\backslash \mathcal{F}\bc{A\abc{i;}}$\footnote{This is equivalent to what we need, see \Cref{r80}.};
  \item $i$ is \textbf{firmly frozen} in $A$ if $i\in \mathcal{F}\bc{A\abc{i;}}$;
  \item $i$ is \textbf{completely frozen} in $A$ if $i$ is firmly frozen in both $A$ and $A^T$.
  \end{enumerate}
\end{definition}
In addition, in \Cref{sec_prelifv} we show that variables which are frailly frozen in $A$ are also frailly frozen in the transpose $A^T$. So indeed, we can partition the set of coordinates into five disjoint sets as follows:
\begin{restatable}{definition}{definitionxyzuv}(Typecasting of variables)\label{dxyzuv}
For any matrix $A\in \mathbb{F}^{m\times n}$, we partition the set $[m\wedge n]$ into
\begin{enumerate}[label=(\roman*)]  
    \item the set $\cX(A)$ of frailly frozen variables;
  \item the set $\cY(A)$ of completely frozen variables;
  \item the set $\cZ(A)$ of variables that are neither frozen in $A$ or $A^T$;
  \item the set $\cU(A)$ of variables that are not frozen in $A$ and firmly frozen in $A^T$;
  \item the set $\cV(A)$ of variables that are firmly frozen in $A$ and not frozen in $A^T$.
\end{enumerate}
For each $i \in [m \wedge n]$, we refer to the category it belongs to with respect to the above partition as its \textbf{type}.  
\end{restatable}
This distinction between different types of frozen variables is a chief ingredient in our calculation of the lower bound, and the main difference with respect to the preceding works \cite{coja2022sparse, coja2022rank}. Notably, it allows us to extend core ideas of these articles to symmetric matrices. For example, with the terminology of \Cref{dxyzuv}, we can now express the rank increase of interest alternatively as
\begin{align*}
    \rk{\bm{T}_{n+1,t/n} [\bm{\theta}]} - \rk{\bm{T}_{n,t/n} [\bm{\theta}]} & = \ind\{n+1\in \mathcal{X}(\bm{T}_{n+1,t/n} [\bm{\theta}])\}+2 \cdot \ind\{n+1\in \mathcal{Y}(\bm{T}_{n+1,t/n} [\bm{\theta}])\} \\
    & \quad +  \ind\{n+1\in \mathcal{U}(\bm{T}_{n+1,t/n} [\bm{\theta}])\} +  \ind\{n+1\in \mathcal{V}(\bm{T}_{n+1,t/n} [\bm{\theta}])\}
    \end{align*}
(for a proof of this identity, see  \Cref{lr}). 

Returning to the discussion at the end of \Cref{sec_rifpm}, the typecasting allows us to the derive fixed point equations not for $\val_{n,t/n}$, but for \textit{some} of the proportions of the finer types. Thereby, we gain a better understanding of the proportion of frailly frozen variables and of $\val_{n,t/n}$. 
And, what is more, these fixed point equations provide enough information to derive the desired lower bound on $h_t(\val_{n,t/n})$ as given below in \Cref{imp_pro}, and therefore to bypass \cref{conj_fra},
which is precisely what we need. The derivation of the fixed point equations is the content of \Cref{sec_pc}.

\subsection{The heuristic fixed point equation and its connection to $R_d(\alpha)$}\label{sec_cripoi}
Let us return to the heuristic fixed point equation (\ref{eq_heuristic_fix}), which suggests that only zeroes of the function $G_d:[0,1] \mapsto \RR$,
\begin{equation}\label{fun_g}
   G_d(\alpha):=\alpha+\phi_d\bc{1-\phi_d\bc{\alpha}}-1,
    \end{equation}
constitute viable candidates for $\val_{n, t/n}$.
And indeed, for any $d\geq 0$, $G_d$ has at least one zero: If $\alpha_0(d) \in [0,1]$ is such that $\alpha_0(d)=1-\phi_d\bc{\alpha_0(d)}$\footnote{The existence and uniqueness of $\alpha_0(d)$ are straightforward to check, see \Cref{lem_easy_fix} in Appendix \ref{app_proal}.}, then 
\begin{align}\label{zero_F}
    G_d(\alpha_0(d))=\alpha_0(d)+\phi_d\bc{1-\phi_d\bc{\alpha_0(d)}}-1=-\phi_d(\alpha) + \phi_d(\alpha) = 0.
\end{align}
 Unfortunately, for some $d \geq 0$, $G_d$ has more zeroes: Let $\alpha_\star(d)$ and $\alpha^\star(d)$ denote the smallest and largest zeroes of $G_d(\alpha)$ in $[0,1]$, respectively. The existence of $\alpha_\star(d)$ and $\alpha^\star(d)$ is guaranteed by \Cref{zero_F}.  A detailed analysis of the function $G_d$, its zeroes and relation to the function $1-R_d$ is carried out in \cite{coja2022sparse}, where the asymmetric counterpart of $\bm{A}_{n,p}$ with all non-zero entries being identical to $1$ was studied.
In \cite{coja2022sparse}, the authors show that $G_d$ has at most the three zeroes $\alpha_\star(d) \leq \alpha_0(d) \leq \alpha^\star(d)$. 

From the analysis of the finer types as described in \Cref{sec_frozen_rev}, it will become apparent that in the limit, only the two zeroes $\alpha_\star(d)$ and $\alpha^\star(d)$ correspond to possible values of $\val_{n,d/n}$. For the asymmetric case, where no perturbation is necessary, the connection between $G_d$ and the proportion of frozen variables has been studied in \cite{coja2022sparse}. While we cannot derive a picture as detailed as in \cite{coja2022sparse}, we can show that $h_t(\val_{n,t/n})$ is no less than $h_t$ evaluated at one of the zeroes, which provides a  sufficient substitute for the exact asymptotic characterisation of $\val_{n,t/n}$:

\begin{restatable}[Lower bound on the rank increase]{proposition}{imp}\label{imp_pro}
For any $d>0$, 
\begin{align}\label{eq_eqeqh}
h_t\bc{\val_{n,t/n}}\geq h_t\bc{\alpha^\star(t)} + \oone.
   \end{align}
\end{restatable}

Recall that the principal aim of (\ref{eq_heuristic_fix}) was to establish a connection to the rank formula \cref{e01b}, which comes in terms of an optimization problem over $[0,1]$. 
The function $R_d$ attains its minimum on $[0,1]$ either for $\alpha \in \{0,1\}$ or for $\alpha \in (0,1)$ such that 
  \begin{equation}\label{eq_difr}
  R_d'(\alpha)=d^2\phi_d\bc{\alpha}\bc{\alpha+\phi_d\bc{1-\phi_d\bc{\alpha}}-1}=0.
  \end{equation}
The little calculation of \cref{eq_difr} shows that $R_d'(\alpha)=0$ if and only if $G_d(\alpha)=0$. 
Indeed, $\alpha_\star(d)$ and $\alpha^\star(d)$ are the two only minimizers of $R_d$ on $[0,1]$:
\begin{align}\label{extre_R}
R_d\bc{\alpha^\star(d)}=R_d\bc{\alpha_\star(d)}=\min_{\alpha\in[0,1]}R_d(\alpha).
\end{align}
Thus doubtlessly, the lower bound \cref{eq_eqeqh} establishes a connection to the minimizers of $R_d$. Given \Cref{lem_rkdif_rank,imp_pro}, it is now a matter of analysis to prove \cref{e01}, which we complete next.



\begin{figure}[H]
\centering
\subfigure{\label{f1}
\scalebox{.5}{\input{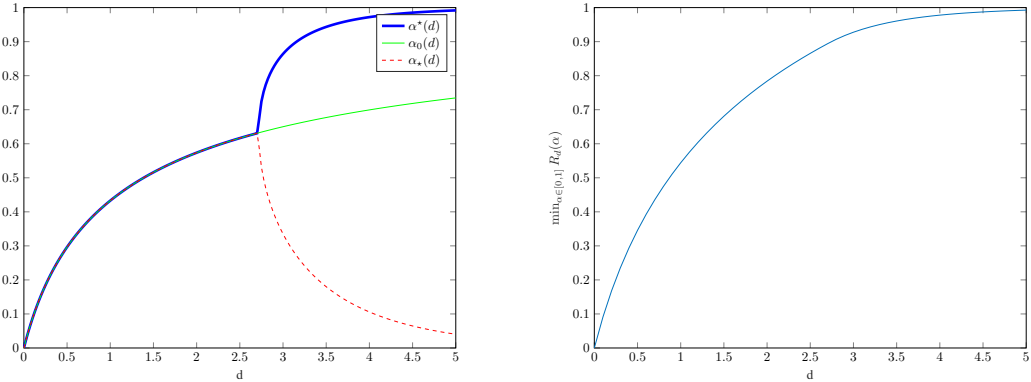}}}
\subfigure{
\scalebox{.5}{
%
%
\definecolor{mycolor1}{rgb}{0.00000,0.44700,0.74100}%
\begin{tikzpicture}

\begin{axis}[%
width=4.521in,
height=3.566in,
at={(0.758in,0.481in)},
scale only axis,
xmin=0,
xmax=5,
xlabel style={font=\color{white!15!black}},
xlabel={d},
ymin=0,
ymax=1,
ylabel style={font=\color{white!15!black}},
ylabel={$\min_{\alpha\in[0,1]}R_d(\alpha)$},
axis background/.style={fill=white}
]
\addplot [color=mycolor1, forget plot]
  table[row sep=crcr]{%
0	0\\
0.1	0.0911554126772786\\
0.2	0.168177234493145\\
0.3	0.234787670787647\\
0.4	0.293389566533976\\
0.5	0.345631947744951\\
0.6	0.392698953393658\\
0.7	0.435471212185482\\
0.8	0.474622220208094\\
0.9	0.510678996956901\\
1	0.544061907323596\\
1.1	0.575111733095978\\
1.2	0.60410862364368\\
1.3	0.631285697574958\\
1.4	0.656839018261486\\
1.5	0.680935049179545\\
1.6	0.703716319102483\\
1.7	0.725305790991791\\
1.8	0.745810275979815\\
1.9	0.765323133050103\\
2	0.783926426954236\\
2.1	0.80169267004409\\
2.2	0.818686240874352\\
2.3	0.83496454907268\\
2.4	0.850578999113114\\
2.5	0.865575793294474\\
2.6	0.879996605093397\\
2.7	0.893879147224236\\
2.8	0.906785552979839\\
2.9	0.917975771442649\\
3	0.927687457885459\\
3.1	0.936138844003126\\
3.2	0.943511979734529\\
3.3	0.949959448118232\\
3.4	0.95560969749858\\
3.5	0.960571303645604\\
3.6	0.964936397559604\\
3.7	0.968783438260272\\
3.8	0.972179467963217\\
3.9	0.975181955702018\\
4	0.977840311818603\\
4.1	0.980197137789812\\
4.2	0.982289262122125\\
4.3	0.984148602468242\\
4.4	0.985802885921016\\
4.5	0.987276253048518\\
4.6	0.988589766222212\\
4.7	0.989761838839015\\
4.8	0.990808598906904\\
4.9	0.991744197970119\\
5	0.992581074354835\\
};
\end{axis}

\begin{axis}[%
width=5.833in,
height=4.375in,
at={(0in,0in)},
scale only axis,
xmin=0,
xmax=1,
ymin=0,
ymax=1,
axis line style={draw=none},
ticks=none,
axis x line*=bottom,
axis y line*=left
]
\end{axis}
\end{tikzpicture}%
}}

\caption{Left: Plot of $\alpha^\star(d)$, $\alpha_0(d)$ and $\alpha_\star(d)$, which are distinct for $d>\eul$.  Right: Plot of the function $d \mapsto \min_{\alpha\in [0,1]}R_d(\alpha)$.}
\end{figure}

\subsection{Lower bound on the expected rank: Proof of \Cref{t1} subject to \Cref{lem_rkdif_rank,imp_pro}} \label{sec_lower_exp}
An application of \Cref{lem_rkdif_rank,imp_pro} now gives the following lower bound for $\frac{1}{n}\Erw\brk{\rk{\bm{T}_{n,d/n}[\bm{\theta}]}}$:
\begin{align}\label{eq_rkdifsum}
    \frac{1}{n}\Erw\brk{\rk{\bm{T}_{n,d/n} [\bm{\theta}]}}
    \geq &\frac{1}{n}\sum_{m=\varepsilon n}^{n-1}\bc{\Erw\brk{\rk{\bm{T}_{m+1,(dm/n)/m} [\bm{\theta}]}}-\Erw\brk{\rk{\bm{T}_{m,(dm/n)/m}[\bm{\theta}]}}}\\
    =&\frac{1}{n}\sum_{m=\varepsilon n}^{n-1} \Erw\brk{h_{dm/n}\bc{\val_{m,d/n}}}+o_{n,P}(1)\geq\frac{1}{n}\sum_{m=\varepsilon n}^{n-1} h_{dm/n}\bc{\alpha^\star(dm/n)}+o_{n,P}(1).\nonumber
\end{align}

The sum $\frac{1}{n}\sum_{m=\varepsilon n}^{n-1}h_{md/n}\bc{\alpha^\star(md/n)}$ can be treated as a Riemann sum, i.e.,
    \[
    \frac{1}{n}\mathbb{E}\brk{\rk{\bm{T}_{n,d/n}[\vth]}}\geq \int_\varepsilon^1 h_{ds}\bc{{\alpha}^*\bc{ds}}\dif s+o_{n,P}(1)=\frac{1}{d}\int_{\varepsilon d}^d h_t\bc{\alpha^\star(t)}\dif t+o_{n,P}(1).
    \]
Taking the appropriate limits on both sides gives
\[\liminf_{P\to\infty}\liminf_{n\to\infty}\inf_{\substack{N\geq n,\\J_N\in \syN}}\frac{1}{n}\mathbb{E}\brk{\rk{\bm{T}_{n,d/n}[\vth]}}\geq\frac{1}{d}\int_{\varepsilon d}^d h_t\bc{\alpha^\star(t)}\dif t\]
and since we can choose $\varepsilon$ arbitrarily small, we conclude that
\[\liminf_{P\to\infty}\liminf_{n\to\infty}\inf_{\substack{N\geq n,\\J_N\in \syN}}\frac{1}{n}\mathbb{E}\brk{\rk{\bm{T}_{n,d/n}[\vth]}}\geq\frac{1}{d}\int_{0}^{d}h_t\bc{\alpha^\star(t)}\dif t.\]
Then indeed, as we prove in \Cref{sec_eqqdrd}, the derived integral expression coincides with the desired rank formula:
\begin{restatable}[Integral evaluation]{lemma}{qeqaulr}\label{fact_1}
For any $d\geq 0$,
\[\int_{0}^{d}h_t\bc{\alpha^\star(t)}\dif t= d \cdot R_d\bc{\alpha^\star(d)}.\]
\end{restatable}
%
Now, the combination of 
\cref{extre_R} and \Cref{fact_1} yields that
\[\frac{1}{d}\int_{0}^{d}h_t\bc{\alpha^\star(t)}\dif t = R_d\bc{\alpha^\star(d)}=\min_{\alpha\in[0,1]}R_d(\alpha)\]
and therefore by \cref{ee1} that
\begin{equation}\label{eq_t_r}
     \liminf_{n\to \infty}\inf_{\substack{N\geq n,\\J_N\in \syN}}\frac{1}{n}\Erw\brk{\rank_\FF\bc{\bm{T}_{n,d/n}}}\geq \liminf_{P\to\infty}\liminf_{n\to\infty}\inf_{\substack{N\geq n,\\J_N\in \syN}}\frac{1}{n}\mathbb{E}\brk{\rk{\bm{T}_{n,d/n}[\vth]}}\geq \min_{\alpha\in[0,1]}R_d(\alpha).
\end{equation}
By definition of  $\bm{T}_{n,p} = \bm{T}_{n,p}^{(N)}$, we have $\rk{\bm{T}_{N,p}^{(N)}} = \rk{\bm A_{N,p}}$ and consequently
\begin{equation}\label{eq_t_a}
\inf_{J_n\in \syn}\frac{1}{n}\Erw\brk{\rank_\FF\bc{\bm{A}_{n,d/n}}}=\inf_{\substack{N= n,\\J_N\in \syN}}\frac{1}{n}\Erw\brk{\rank_\FF\bc{\bm{T}_{n,d/n}^{(N)}}}\geq \inf_{\substack{N\geq n,\\J_N\in \syN}}\frac{1}{n}\Erw\brk{\rank_\FF\bc{\bm{T}_{n,d/n}^{(N)}}}.
\end{equation}
\Cref{t1} now follows from the combination of \cref{eq_t_r,eq_t_a}. \qed

\subsection{Discussion}

The understanding of the ensemble of adjacency matrices of Erd\H{o}s-Rényi random graphs has seen major advances during the last two decades, in particular with respect to its real rank and spectral properties. Prominently, $\ln(n)/n$ is a threshold for the singularity of these matrices \cite{basak2021sharp,costello2006random}. More generally, in the regime where $p_n \in [c \ln(n)/n,1/2]$ for $c>0$ and for more general real matrix entries as considered in the current article, Costello and Vu \cite{costello2010rank} show that with high probability, the nullity of $\bm A_{n,p_n}$ is exactly equal to the number of isolated vertices in the underlying Erd\H{o}s-Rényi random graph. In the same spirit, DeMichele, Moreira and Glasgow \cite{demichele2021rank} show that for $p_n=\omega(1/n)$, with high probability, the nullity of $\bm A_{n,p_n}$ coincides with the number of isolated vertices in the graph that arises from $\bm{G}_{n,d/n}$ after an application of the Karp-Sipser algorithm described in \Cref{sec_intro}. In an associated random matrix process where edges are revealed one after the other, Addario-Berry and Eslava \cite{addario2014hitting} derive a hitting time theorem in the sense that with high probability, the matrix becomes singular at the exact moment when there are no zero rows and columns left. 

In the challenging sparse regime where $p_n=d/n$ for fixed $d>0$, much less is known. Notably, there is the asymptotic rational rank formula \cref{e01b} for $\bm A_{n,d/n}$ by Bordenave, Lelarge and Salez \cite{bordenave2011rank}. Recently, building on the machinery of \cite{bordenave2011rank}, Ferber et al. \cite{ferber2021singularity} have shown that the $k$-core for $k \geq 3$ is non-singular with high probability, thereby resolving an open conjecture of Vu from 2014.

This work has been inspired by recent advances on the rank of random matrices in the context of random constraint satisfaction problems, in particular work on the $k$-XORSAT problem \cite{ayre2020satisfiability,coja2023XORSAT,dietzfelbinger2010tight,dubois2002XORSAT,pittel2016sat} and a model inspired by random code ensembles \cite{coja2022rank}. In this context, it is natural to consider the matrices not only over the reals, but as binary matrices or more generally, matrices over finite fields. 

Correspondingly, this article crucially builds on the methodology developed in \cite{ayre2020satisfiability,coja2022rank}. However, because of the symmetry of our model, virtually all core ides have to be developed differently in comparison to \cite{ayre2020satisfiability,coja2022rank}. First of all, we modify the perturbation according to \Cref{d34}. 
While the basic idea of a perturbation as in \Cref{d34} in the context of random graphical models goes back to information theory \cite{andrea2008estimating}, it has since been successfully applied to the study of random inference problems and random factor graphs \cite{coja2021cut,coja2017info}. 
The basis for an application to asymmetric sparse random matrices, in combination with the conceptualisation of linear relations, has been laid out in \cite{coja2022rank}. In comparison to this previous application, the perturbation in \Cref{d34} is of a slightly different flavour, since it cannot be straightforwardly interpreted as the addition of unary factor nodes in the underlying graphical model. In \cite{coja2022rank} and the earlier version \cite{ayre2020satisfiability}, as well as in results on random factor graphs, the perturbation has proven to be particularly useful when combined with the Aizenman-Sims-Starr scheme from mathematical physics, which brings us to our next modification: Instead of combining the pinning operation with the Aizenman-Sims-Starr scheme, we apply the telescoping argument \cref{eq_rkdifsum0} and therefore compare the matrices $\bm{T}_{m+1,d/n}$ and $\bm{T}_{m,d/n}$ rather than $\bm{T}_{n+1,d/(n+1)} $ and $\bm{T}_{n,d/n}$. This is due to the fact that an application of the Aizenman-Sims-Starr scheme as in \cite{coja2022rank} would require knowledge about the event \cref{conj_fra} that we do not have, and comes at the price of pursuing a different route to characterise $\val_{n,t/n}$. We therefore introduce frailly frozen variables, which are probably the most essential difference between this article and the previous work on asymmetric matrices \cite{ayre2020satisfiability,coja2022sparse,coja2022rank}.


Finally, we believe that the methods developed in this article will generalise to broader symmetric matrix structures. It would also be interesting to see whether the fraction of frozen variables in the unperturbed matrix $\bm A_{n,d/n}$ satisfies an anti-concentration result as its asymmetric counterpart \cite{coja2022sparse}, or whether the two models behave differently. In hindsight, the rank formula \Cref{t0} gives us some information about the perturbed matrix and \cref{conj_fra}. The proof of \Cref{imp_pro} shows that there are essentially two cases: In the first case, 
the proportion of frailly frozen variables $\vx_{n,t/n}$ is approximately zero and the proportion of frozen variables $\val_{n,t/n}$ is approximately $\alpha_\star(d)$ or $\alpha^\star(d)$. In the second case, the proportion of frailly frozen variables $\vx_{n,t/n}$ is approximately $\alpha^{\star}(d) - \alpha_\star(d)$ and the proportion of frozen variables $\val_{n,t/n}$ is approximately $\alpha^\star(d)$. 
From simulations, it seems likely that only the first case corresponds to the actual asymptotic behaviour of the perturbed matrices under consideration, but we cannot exclude the second case at present.


\section{Matrix perturbations}\label{sec_mper}
In this short section, we prove the two most important properties of the matrix perturbation introduced in \Cref{d34}: In \Cref{sec_proof_lc1}, we construct the coupling from \Cref{lc1}, which ensures that w.h.p., for any two large square matrices that differ by one in their size, their canonical perturbation is based on the same row- and column-perturbation matrices (compare (\ref{e001})). 
In \Cref{sec_proof_p1}, we then prove \Cref{p1} on the joint deletion of short proper relations in both the perturbed $A$ and its transpose.

\subsection{Coupling of perturbation matrices: Proof of \Cref{lc1}} \label{sec_proof_lc1}

To couple the matrices $\THETA_r[\theta_r,n_1|n_2]$, we couple the locations of their non-zero entries row by row. For a given row $k$, the basic idea is to construct a coupling $(\vj_{k,n_1})_{n_1 \geq 1}$ of uniformly distributed random variables $\vj_{k,n_1} \sim \text{Unif}([n_1])$ on increasing integer intervals, such that for any two random variables, $\PP\bc{\vj_{k,n_0} \neq\vj_{k,n_1}} = \dtv{\text{Unif}([n_0]),\text{Unif}([n_1])}$. For the overall coupling, we then take the product distribution over the rows.
More precisely, let $(\vu_{k,\ell})_{k,\ell\geq 1}$ be an array of independent random variables such that for all $k, \ell \geq 1$, $\vu_{k,\ell}$ is uniformly distributed on $[\ell]$. For any $n_1 \in \NN$, set
\begin{align}
\vj_{k,n_1} = \max\{\ell \in [n_1]: \vu_{k,\ell} = \ell\}.
\end{align}
Since $\vu_{k,1}=1$, the set above is nonempty, and it is straightforward to verify that $\vj_{k,n_1} \sim$ Unif$([n_1])$.

For any $\theta_r, n_2 \in \NN$, $n_1 \in [n_2]$, let $\THETA_r[\theta_r, n_1| n_2] \in \mathbb{F}^{\theta_r \times n_2}$ be the matrix where row $k \in [\theta_r]$ has its unique non-zero entry in column $\vj_{k,n_1}$.  Since the definition of $\vj_{k,n_1}$ only depends on $n_1$, but not on $\theta_r$ or $n_2$, this coupling satisfies properties (i) and (ii). 

Consider now $n_0 \leq n_1 \leq n_2$. Then $\THETA_r[\theta_r, n_0|n_2] = \THETA_r[\theta_r, n_1|n_2] $ if and only if $\vj_{k,n_0} = \vj_{k,n_1}$ for all $k \in [\theta_r]$, or equivalently $\vu_{k,n_0+1}<n_0+1, \ldots, \vu_{k,n_1} < n_1$.  Therefore,
\[\mathbb{P}\left(\THETA_r[\theta_r,n_0|n_2]=\THETA_r[\theta_r,n_1|n_2]\right)=\prod_{k=n_0+1}^{n_1}\bc{\frac{k-1}{k}}^{\theta_r}=\bc{\frac{n_0}{n_1}}^{\theta_r},\]
so that the coupling satisfies (iii).

The coupling of $\{\THETA_c[m_1| m_2,\theta_c]: \theta_c, m_2 \geq 1, m_1\in [m_2]\}$ can be constructed along the same lines.\qed

\subsection{Perturbation eliminates most short proper relations: Proof of \Cref{p1}}\label{sec_proof_p1}

Recall the definition of the canonical perturbation from \Cref{d34}.
In this section, we prove \Cref{p1}, which ensures that for any $\delta>0$ and $\ell \in \NN_{\geq 2}$, the canonical perturbation of any (almost) square matrix $A$, as well as its transpose, are $(\delta, \ell)$-free with probability arbitrarily close to one, provided that the matrix dimension and the perturbation parameter $P$ are chosen large enough.

The main ingredient in the proof of \Cref{p1} is the following lemma:
\begin{lemma}[{\cite[Proposition 2.3]{coja2022rank}}]\label{l0}
Let $\delta>0$ and $\ell \in \NN_{\geq 2}$. Then there exists $P'=P'(\delta,\ell) \in \NN$ such that for any $P \geq P'$ the following holds: For any matrix $A \in \FF^{m\times n}$
\begin{align}\label{lem_pin}
\mathbb{P}\left(\text{$\begin{pmatrix}A\\ \THETA_r[\bm{\theta}_r, n \vert n]\end{pmatrix}$ is $(\delta,\ell)$-free}\right) \geq1-\delta,
\end{align}
provided that $\bm{\theta}_r \sim \text{Unif}([P])$ and is independent of the coupling $\{\THETA_r[\theta_r, n_1| n_2]: \theta_r, n_2 \geq 1, n_1\in [n_2]\}$. 
\end{lemma}

\begin{remark}
\Cref{l0} is a minor adaptation of \cite[Proposition 2.3]{coja2022rank}. While the exact wording is for $P = P'(\delta,\ell)$ rather than \textit{all} $P \geq P'(\delta,\ell)$, its proof 
shows that all choices of $P > 4\ell^3/\delta^4$ imply (\ref{lem_pin}).
\end{remark}


Before we prove \Cref{p1}, we observe the following simple consequence of \Cref{l0}:

\begin{corollary}\label{c1}
Let $\delta>0$ and $L \in \NN_{\geq 2}$. Then there exists $P'=P'(\delta,L) \in \NN$ such that for any $P \geq P'$ the following holds: For any matrix $A \in \FF^{m\times n}$ and $n_1 \in [n]$ 
\[\mathbb{P}\left(\text{$\begin{pmatrix}A\\\bm{\Theta}_r[\bm{\theta}_r,n_1|n]\end{pmatrix}$ is $(\delta,\ell)$-free for $\ 2\leq \ell\leq L$}\right) \geq\bc{\frac{n_1}{n}}^P-\delta,\]
provided that $\bm{\theta}_r \sim \text{Unif}([P])$ and is independent of the coupling $\{\THETA_r[\theta_r, n_1| n_2]: \theta_r, n_2 \geq 1, n_1\in [n_2]\}$. 
\end{corollary}

\begin{proof}
Fix $\delta>0$ and $L \in \NN_{\geq 2}$. For any $\ell \in \{2, \ldots, L\}$, \Cref{l0} guarantees the existence of $P_\ell = P_\ell(\delta/L,\ell)\in \NN$ such that for any $P\geq P_\ell$ and $\bm{\theta}_r \sim \text{Unif}([P])$,
\begin{align}\label{eq_c1_0}
\mathbb{P}\left(\text{$\begin{pmatrix}A\\\bm{\Theta}_r[\bm{\theta}_r, n \vert n]\end{pmatrix}$ is $(\delta,\ell)$-free}\right) \geq \mathbb{P}\left(\text{$\begin{pmatrix}A\\\bm{\Theta}_r[\bm{\theta}_r, n \vert n]\end{pmatrix}$ is $(\delta/L,\ell)$-free}\right) \geq1-\delta/L.
\end{align}
Let $P'=\max_{2\leq \ell\leq L}P_{\ell}$. Then for any $P \geq P'$ and $\bm{\theta}_r \sim \text{Unif}([P])$, by (\ref{eq_c1_0}) and a union bound,
\begin{align}\label{eq_c1_1}
\mathbb{P}\left(\text{$\begin{pmatrix}A\\\bm{\Theta}_r[\bm{\theta}_r, n \vert n]\end{pmatrix}$ is $(\delta,\ell)$-free for $2\leq \ell\leq L$}\right)\geq 1-\sum_{\ell=2}^{L}\mathbb{P}\left(\text{$\begin{pmatrix}A\\\bm{\Theta}_r[\bm{\theta}_r, n \vert n]\end{pmatrix}$ is not $(\delta,\ell)$-free}\right) \geq1-\delta.
\end{align}

By \Cref{lc1} (iii), $\PP\bc{\bm{\Theta}_r[\bm{\theta}_r, n|n] = \Theta_r[\bm{\theta}_r,n_1|n]}= \Erw\brk{(n_1/n)^{\bm{\theta}_r}} \geq (n_1/n)^P$. 
Therefore,
\begin{align}\label{eq_c1_2}
\PP\bc{\begin{pmatrix}A\\\bm{\Theta}_r[\bm{\theta}_r, n \vert n]\end{pmatrix}  (\delta,\ell)\text{-free for } 2\leq \ell\leq L} 
\leq \mathbb{P}\bc{\text{$\begin{pmatrix}A\\\bm{\Theta}_r[\bm{\theta}_r,n_1|n]\end{pmatrix}$  $(\delta,\ell)$-free for $2\leq \ell\leq L$}}+1-\bc{\frac{n_1}{n}}^{P}.
\end{align}
Combining (\ref{eq_c1_1}) and (\ref{eq_c1_2}) yields the claim.
\end{proof}

\begin{proof}[Proof of \Cref{p1}]
Fix $\delta >0, L \in \mathbb{N}_{\geq 2}$ and $s \in \ZZ$. For $n, P \in \NN$, let $\bm{\theta} = (\bm{\theta}_r, \bm{\theta}_c) \sim \text{Unif}([P]^2)$ and $A\in\FF^{(n+s)\times n}$. With the coupling from \Cref{lc1} and ${\bm{A}'}:=\begin{pmatrix}A&\bm{\Theta}_c[n+s|n+s,\bm{\theta}_c]\end{pmatrix}$, 
\begin{align}\label{eq_p1_0}
A[\bm{\theta}]
=\begin{pmatrix}A&\bm{\Theta}_c[n+s|n+s, \bm{\theta}_c]\\\bm{\Theta}_r[\bm{\theta}_r, n \vert n] & 0_{\bm{\theta}_r\times \bm{\theta}_c} \end{pmatrix}
=\begin{pmatrix}{\bm{A}'}\\\bm{\Theta}_r[\bm{\theta}_r,n|n+\bm{\theta}_c]\end{pmatrix}.
\end{align}
Conditionally on $\bm{A}'$ and $\bm{\theta}_c$, because of independence of the row and column perturbations, $\bm{\Theta}_r[\bm{\theta}_r,n|n+\bm{\theta}_c]$ is distributed as the perturbation in \Cref{c1} with the ensuing choice of $n_1$ and $n$. Thus, for any $a>1$, if $ P(\delta/a, L)$ is chosen large enough, conditioning on $\bm{A}'$ and $\bm{\theta}_c$ in (\ref{eq_p1_0}) yields that for $P\geq P(\delta/a, L)$,
\begin{align}\label{eq_p1_1}
\mathbb{P}\left(\text{$A[\bm{\theta}]$ is $(\delta,\ell)$-free for $2\leq \ell\leq L$}\right) \geq \mathbb{P}\left(\text{$A[\bm{\theta}]$ is $(\delta/a,\ell)$-free for $2\leq \ell\leq L$}\right) \geq \bc{\frac{n}{n+P}}^{P}-\delta/a.
\end{align}
By an analogous argument, also for $P \geq P(\delta/a,L)$,
\[\mathbb{P}\left(\text{$A[\bm{\theta}]^T$ is $(\delta,\ell)$-free for $2\leq \ell\leq L$}\right) \geq \bc{\frac{n+s}{n+s+P}}^{P}-\delta/a.\]
Since $ \mathbb{P}(\mathfrak B_1\cap \mathfrak B_2)\geq \mathbb{P}(\mathfrak B_1)+\mathbb{P}(\mathfrak B_2)-1$ for any two events $\mathfrak B_1, \mathfrak B_2$, we conclude that
\[\mathbb{P}\left(\text{Both $A[\bm{\theta}]$ and $A[\bm{\theta}]^T$ are $(\delta,\ell)$-free for $2\leq \ell\leq L$}\right) \geq \bc{\frac{n+s}{n+s+P}}^{P}+\bc{\frac{n}{n+P}}^{P}-2\delta/a-1.\]
In particular,
\[\limsup_{P \to \infty}\limsup_{n \to \infty}\sup_{A \in \FF^{(n+s)\times n}} \mathbb{P}\left(\text{$A[\bm{\theta}]$ or $A[\bm{\theta}]^T$ is not $(\delta,\ell)$-free for some $2\leq \ell\leq L$}\right) \leq 2\delta/a.\]
Since this is upper bound holds for any $a>1$, \cref{e1_rep} follows.
\end{proof}

\begin{remark}
It is natural to wonder whether there is a possibility to perturb a symmetric matrix $A$ such that the perturbed matrix $A[\bm\theta]$ is symmetric as well \textit{and}  \Cref{p1} holds. However, simply choosing $\THETA_c[n|n,\bm \theta_c]=\THETA_r[\bm \theta_r, n|n]^T$ does not have the desired effect: For (\ref{eq_p1_1}) to hold, it is crucial that both the number of rows as well as the columns of the non-zero indices of $\THETA_r[\bm \theta_r, n|n+\bm\theta_c]$ are chosen uniformly given $\bm{A}'=\begin{pmatrix}A&\THETA_c[n|n,\bm \theta_c]\end{pmatrix}$. Thus, the above perturbation technique necessarily destroys the matrix symmetry. 
\end{remark}

\section{Frozen variables: General properties \& stability}\label{sec_prof}
 The principle aim of this section is to derive general properties of the various types of frozen variables as well as to prove stability of the proportions of types in the transition from $\bm{T}_{n,t/n}[\bm \theta]$ to $\bm{T}_{n+1,t/n}[\bm \theta]$.
 In this sense, {our main result of this section, \Cref{l91},} asserts that the proportions of the various types remain nearly unchanged when we grow the matrix from $n$ to $n+1$. 

\subsection{How the type of a variable encodes rank change under row- and column removal}\label{sec_prelifv}
We first present basic \textit{deterministic} implications of the type of a variable that are used throughout the article{, and that} indicate the significance of the types of \Cref{dxyzuv}. More specifically, we are ultimately interested in the rank decrease upon simultaneous removal of row $i$ and column $i$ from a given matrix $A \in \FF^{m \times n}$. In this section, we prove that the type of $i$ according to \Cref{dxyzuv} completely determines the ensuing rank change. 
The starting point is the following lemma on frozen variables: Living up to their name, in \Cref{rem_froz}, frozen variables were characterised as coordinates that take the value zero in any kernel vector. The following lemma shows how the rank of any given matrix changes, if a column that corresponds to a frozen variable is removed from it:

\begin{lemma}[{\cite[Lemma 4.7]{demichele2021rank}}]\label{p4}
Let $A\in \mathbb{F}^{m\times n}$ and $i\in [n]$. Then
\[i\in \mathcal{F}(A) \qquad \Longleftrightarrow  \qquad \rank\bc{A}-\rank\bc{A\abc{;i}}=1 .\]
\end{lemma}

\begin{proof}
Recall that we denote the $i$th standard unit vector in $\FF^{1 \times n}$ by $e_n(i)$. While the linear dependencies of column $i$ of $A$ with the other columns of $A$ may be intricate, attaching $e_n(i)$ at the bottom of $A$ surely renders column $i$ 
linearly independent of all the other columns. Thus
\[\rank\begin{pmatrix}A\\e_n(i)\end{pmatrix}= 
1+\rank\bc{A\abc{;i}}.\]
On the other hand, by \Cref{d2}, $i$ is frozen in $A$ if and only if $e_n(i)$ is in the row span of $A$, so
\[
\rank\begin{pmatrix}A\\e_n(i)\end{pmatrix} = \rank\bc{A} + \ind\cbc{i \notin \cF(A)}.\]
\end{proof}

The next lemma demonstrates that, generally, column removal and row addition cannot ``unfreeze'' variables:

\begin{lemma}\label{l1.1b}
Let $A\in \mathbb{F}^{m\times n}$, $b\in\mathbb{F}^{m\times 1}$, $c\in\mathbb{F}^{1\times n}$ and $i\in [n]$. Then
\begin{enumerate}[label=(\roman*)]
  \item $i\in \mathcal{F}\bc{\begin{pmatrix}A&b\end{pmatrix}} \quad \Longrightarrow \quad i\in \mathcal{F}(A)$;
  \item $i\in \mathcal{F}(A) \quad \Longrightarrow \quad  i\in \mathcal{F}\bc{\begin{pmatrix}A\\c\end{pmatrix}}$.
\end{enumerate}
\end{lemma}

\begin{proof}
Both statements immediately follow from the characterisation of frozen variables from \Cref{p4}: variable $i$ is frozen in $A$ if and only if column $i$ does not lie in the linear span of the other columns of $A$.
\end{proof}

While \Cref{l1.1b} shows that addition of rows can only enlarge the set of frozen variables, the next lemma studies the consequences of row removal. Indeed, \Cref{l8.4} illustrates that the removal of a row has the same effect as addition of a unit vector (which effectively forbids to use the corresponding row in representations):

\begin{lemma}\label{l8.4}
For any matrix $A \in \FF^{m \times n}$, $i\in [n]$ and $j\in [m]$,
\[i\in\mathcal{F}(A\abc{j;})\qquad  \Longleftrightarrow \qquad i\in\mathcal{F}\bc{\begin{pmatrix}A&e_m(j)^T\end{pmatrix}}.\]
\end{lemma}

\begin{proof}
Throughout the proof, we abbreviate $B = \begin{pmatrix}A&e_m(j)^T\end{pmatrix}$.

Assume that $i\in \mathcal{F}\bc{A\abc{j;}}$. Since $B\abc{j;}$ only has a zero-column appended at the right in comparison to $A\abc{j;}$, $i$ is also frozen in $B\abc{j;}$. By \Cref{l1.1b} (ii), adding a row cannot unfreeze variables, so
$i\in \mathcal{F}\bc{B}.$

{Next, }assume that $i\in \mathcal{F}(B)$ and let $y=(y_1,\ldots,y_m)$ be a representation of $\cbc{i}$ in $B$. Since no row of $B$ apart from $j$ has a non-zero entry in column $n+1$, $y_j = (yB)_{n+1} =0$, which implies that $y\abc{;j}$ is a representation of $\cbc{i}$ in $A\abc{j;}$. 
%
\end{proof}

We next take a closer look at the frailly frozen variables, which were characterised as those variables that unfreeze under removal of the identically indexed row (see \Cref{dsc} (i)). Since on the other hand, variables can never freeze under row removal, we obtain the following corollary of \Cref{l1.1b} (ii), which expresses that the frailly frozen variables are exactly those variables that are classified differently in the matrix with one appropriately chosen row less than in the original matrix:

\begin{corollary}\label{r80}
For any $A \in \FF^{m \times n}$ and $i \in [m\land n]$, 
\[i \text{ is frailly frozen in } A \qquad \Longleftrightarrow \qquad i\in \mathcal{F}(A)\Delta \mathcal{F}\bc{A\abc{i;}}.\]
\end{corollary}


In \Cref{dxyzuv}, we have claimed that for any matrix $A \in \FF^{m \times n}$, the set $[m \wedge n]$ can be partitioned into five types of variables. The next proposition shows that this claim is justified, since any variable in $[m \wedge n]$ is either frailly frozen, firmly frozen or unfrozen in $A$, and if it is frailly frozen in $A$, then it must also be frailly frozen in $A^T$:
\begin{proposition}\label{c81}
Let $A \in \FF^{m \times n}$ and $i \in [m\land n]$. Then 
\begin{align*}
   i \text{ is  frailly frozen in }A \qquad  \Longleftrightarrow \qquad i \text{ is  frailly frozen in } A^T.
\end{align*}
\end{proposition}

We prove \Cref{c81} by means of \Cref{p4} and the following observation:
\begin{lemma}\label{l8}
For any matrix $A \in \FF^{m \times n}$, {vectors} $b \in \FF^{m \times 1}$, $c \in \FF^{1 \times n}$ and $f \in \FF$,
\begin{equation}\label{el81}
\rank\begin{pmatrix}A\\c\end{pmatrix}-\rank(A)=0 \quad \text{and} \quad \rank\begin{pmatrix}A&b\\c&f\end{pmatrix}-\rank\begin{pmatrix}A&b\end{pmatrix}=1
\end{equation}
if and only if
\begin{equation}\label{el82}
\rank\begin{pmatrix}A&b\end{pmatrix}-\rank(A)=0 \quad \text{and} \quad \rank\begin{pmatrix}A&b\\c&f\end{pmatrix}-\rank\begin{pmatrix}A\\c\end{pmatrix}=1.
\end{equation}
\end{lemma}

\begin{proof}[Proof of \Cref{l8}]
Denote $\rank\begin{pmatrix}A&b\\c&f\end{pmatrix}-\rank(A)$ by $h$ and assume that (\ref{el81}) holds. Then
\begin{align} \label{eq_l8_1}
    h=\rank\begin{pmatrix}A&b\end{pmatrix}-\rank(A)+\rank\begin{pmatrix}A&b\\c&f\end{pmatrix}-\rank\begin{pmatrix}A&b\end{pmatrix}\geq \rank\begin{pmatrix}A&b\\c&f\end{pmatrix}-\rank\begin{pmatrix}A&b\end{pmatrix}=1,
\end{align}
as well as
\begin{align} \label{eq_l8_2}
    h=\rank\begin{pmatrix}A&b\\c&f\end{pmatrix}-\rank\begin{pmatrix}A\\c\end{pmatrix}+\rank\begin{pmatrix}A\\c\end{pmatrix}-\rank(A)\leq 1+\rank\begin{pmatrix}A\\c\end{pmatrix}-\rank(A)=1.
\end{align}
Therefore, $h=1$, and we must have equality throughout (\ref{eq_l8_1}) and (\ref{eq_l8_2}). \cref{el82} then follows. 
The converse implication can be shown to be true analogously.
\end{proof}

\begin{proof}[Proof of \Cref{c81}]
The assertion is an immediate consequence of the characterisation of frozen variables in terms of rank decrease upon column removal from \Cref{p4} in combination with \Cref{l8} applied to $A\abc{i,i}, A\abc{i;}, A\abc{;i}$ and $A$, since the rank of a matrix is identical to that of its transpose.
\end{proof}

The final result of this section, \Cref{lr}, connects the five variable categories $\cX(A),\cY(A),\cZ(A),\cU(A),\cV(A)$ from \Cref{dxyzuv} to the following rank changes under symmetric row and column removal:

\begin{lemma}\label{lr}
For any $A \in \FF^{m \times n}$ and $ i\in [m\land n]$, 
\begin{enumerate}[label=(\roman*)]
  \item \makebox[14em][l]{$i \in \cY(A)$}$ \Longleftrightarrow  \qquad\rank(A)-\rank(A\abc{i;i})=2$;
  \item \makebox[14em][l]{$i \in \cX(A) \cup \cU(A) \cup \cV(A) $}$ \Longleftrightarrow \qquad \rank(A)-\rank(A\abc{i;i})=1$;
  \item \makebox[14em][l]{$i \in \cZ(A)$}$  \Longleftrightarrow  \qquad\rank(A)-\rank(A\abc{i;i})=0$.
\end{enumerate}
Thus,
\begin{equation}\label{lr1}
\begin{aligned}
\rank(A)-\rank(A\abc{i;i})&=\ind\{i \in \cX(A)\}+2\cdot \ind\{i \in \cY(A)\}+\ind\{i \in \cU(A)\}+\ind\{i \in \cV(A)\}\\
&=1+\ind\{i \in \cY(A)\}- \ind\{i \in \cZ(A)\}.
\end{aligned}
\end{equation}
\end{lemma}
\begin{proof}
Let $i\in [m\land n]$. \Cref{p4} yields the representation
\begin{align} \label{eq_lr1}
\rank(A)-\rank(A\abc{i;i})=\rank(A)-\rank(A\abc{i;})+\rank(A\abc{i;})-\rank(A\abc{i;i})=\ind\cbc{i\in \mathcal{F}(A^T)}+\ind\cbc{i\in\mathcal{F}(A\abc{i;})}.
\end{align}
Identities (i)-(iii) now follow from (\ref{eq_lr1}) by an application of \Cref{c81}.
\end{proof}

\subsection{Appending a row to a $(\delta,\ell)$-free matrix}\label{sec_aar}

In the present section, we discuss how the rank of a $(\delta,\ell)$-free matrix $A$ changes upon the attachment of a single row $\bm{b}$ with exactly $\ell$ non-zero entries, which are chosen \textit{uniformly} from a subset of the columns of $A$. Recall that in (\ref{eq00}), we had observed that for a general vector $b$ to be in the row span of $A$, it is sufficient that $\supp{b} \subseteq \cF(A)$ and necessary that either $\supp{b} \subseteq \cF(A)$ or $\supp{b}$ forms a proper relation in $A$. These considerations show that in the \textit{complete absence} of ``short'' proper relations in $A$, rank stagnation upon attachment of a vector $\bm{b}$ with $\ell$ non-zero entries can be equivalently described by the event that all variables of $\supp{\bm{b}}$ are frozen in $A$. 


\Cref{l5}, which revisits an argument from the proof of \cite[Lemma 5.4]{coja2022rank}, shows how to transfer the above reasoning to matrices with \textit{few} short proper relations, where the dominant reason for a rank stagnation upon attachment of a vector should still be the event that all variables in its support are frozen.
For convenience of the reader, we revisit the main step of the argument in \cite{coja2022rank}. For this, let $A \in \FF^{m \times n}$ and
\begin{align}
\PR_\ell(A) = \{I \subseteq [n]:  I \text{ is a proper relation of } A \text{ with } |I|= \ell\}\quad \text{and}\quad \PR(A)=\cup_{\ell=2}^{\infty}\PR_\ell(A)
\end{align}
be the set of proper relations of $A$ of size $\ell \geq 2$ as well as the set of all proper relations of $A$, respectively. 

\begin{lemma}[\cite{coja2022rank}]\label{l5}
Fix $\delta >0$, $\ell \in \NN_{\geq 2}$ and $s\in \NN_0$. For any sequence 
$((\bm{b}_{n-s,n}, \bm{b}_{n-s+1,n}, \ldots, \bm{b}_{n,n}))_{n \in \NN}$ such that for all $n$ and $ n_1 \in [n]\setminus [n-s-1]$, $\bm{b}_{n_1,n} \in \mathbb{F}^{1 \times n}$ and 
supp$(\bm{b}_{n_1,n})$ is uniformly distributed over all $\ell$-subsets of $[n_1]$,
\begin{equation}\label{el51}
\sup_{\substack{m \in \{n-s,\ldots, n+s\}\\ n_1 \in \{n-s, \ldots, n\}}} \sup_{\substack{ A\in \FF^{m \times n}:\\ A \text{ is } (\delta, \ell)-\text{free}}}\mathbb{P}\bc{\supp{\bm{b}_{n_1,n}} \in {\rm \PR_\ell}(A)} \leq \delta \ell! + o_n(1),
\end{equation}
and
\begin{equation}\label{el52}
\sup_{\substack{m \in \{n-s,\ldots, n+s\}\\ n_1 \in \{n-s, \ldots, n\}}} \sup_{\substack{ A\in \FF^{m \times n}:\\ A \text{ is } (\delta, \ell)-\text{free}}}\left|\mathbb{E}\brk{\rank_{\FF}\begin{pmatrix}A\\\bm{b}_{n_1,n}\end{pmatrix}}-\rank_{\FF}(A)-\bc{1-\bc{\frac{\abs{\mathcal{F}(A) \cap [n_1]}}{n_1}}^\ell}\right|\leq \delta \ell! + o_n(1).
\end{equation}
\end{lemma}  

\begin{proof}
Observe that
\begin{equation}\label{el53}
\mathbb{E}\brk{\rank_{\FF}\begin{pmatrix}A\\\bm{b}_{n_1,n}\end{pmatrix}}-\rank _{\FF}(A)=1-\mathbb{P}\bc{\text{$\bm{b}_{n_1,n}$ is in the span of the rows of $A$}}.
\end{equation}
As discussed in the beginning of the subsection, \cref{eq00} gives
\begin{equation}\label{el54}
\begin{aligned}
\mathbb{P}\bc{\supp{\bm{b}_{n_1,n}}\subseteq \mathcal{F}(A)}\leq& \mathbb{P}\bc{\text{$\bm{b}_{n_1,n}$ is in the span of the rows of $A$}}\\
\leq& \mathbb{P}\bc{\supp{\bm{b}_{n_1,n}}\in \PR_\ell(A)}+\mathbb{P}\bc{\supp{\bm{b}_{n_1,n}}\subseteq \mathcal{F}(A)}.
\end{aligned}
\end{equation}
For any $(\delta,\ell)$-free matrix $A \in \FF^{m \times n}$, $\abs{\PR_\ell(A)}\leq \delta n^\ell$, and therefore
\begin{align}\label{eq_prop}
\mathbb{P}\bc{\supp{\bm{b}_{n_1,n}}\in \PR_\ell(A)}\leq \frac{\abs{\PR_\ell(A)}}{\binom{n_1}{\ell}}\leq \delta \ell !\bc{\frac{n}{n-s-\ell}}^\ell.
\end{align}
Taking the supremum over all $(\delta,\ell)$-free matrices $A \in \FF^{m \times n}$, then $m \in \{n-s,\ldots, n+s\}$ and $n_1 \in \{n-s, \ldots, n\}$ yields (\ref{el51}).
To estimate $\mathbb{P}\bc{\supp{\bm{b}_{n_1,n}}\subseteq \mathcal{F}(A)}$, let $\alpha(A)=\abs{\mathcal{F}(A) \cap [n_1]}/n_1$ be the proportion of frozen variables of $A$ among $[n_1]$. Then 
\begin{align}\label{el55}
\abs{\mathbb{P}\bc{\supp{\bm{b}_{n_1,n}}\subseteq \mathcal{F}(A)} - \alpha(A)^\ell}=&\abs{\binom{n_1\alpha(A)}{\ell}\Big/\binom{n_1}{\ell}-\alpha(A)^\ell}  = O(1/n)
\end{align}
uniformly in $m, n_1, A$.
Combining (\ref{el53}) - (\ref{el55}) yields (\ref{el52}).
\end{proof}

\subsection{Stability of types}\label{Sec_change_stat}
As outlined in \Cref{sec_rifpm}, a central ingredient in  our proof strategy is to show that the proportion of frozen variables in $\bm{T}_{n,t/n}[\bm{\theta}]$ is close to that in $\bm{T}_{n+1,t/n}[\bm{\theta}]$, which is the core theme of this section.
Specifically, we look at the extended variable types $\mathcal{X}, \mathcal{Y}, \mathcal{Z}, \mathcal{U}$ and $\mathcal{V}$ from \Cref{dxyzuv}. For each of these types, we define the share it has among the variables of a matrix with perturbation, where the artificial row-perturbation columns are not taken into account:

\begin{definition}[Proportions of types]\label{def_proportions}
\begin{enumerate}[label=(\roman*)]
\item For $A \in \FF^{n \times n}$ and $\mathcal W \in \{\mathcal{X}, \mathcal{Y}, \mathcal{Z}, \mathcal{U}, \mathcal{V}\}$, we use the non-calligraphic lowercase letter $w$ to denote the proportion of variables $i \in [n]$ of the corresponding type:
\begin{align*}
    w(A[\bm{\theta}])=\frac{\abs{\mathcal{W}(A[\bm{\theta}]) \cap [n]}}{n}.
\end{align*}
\item For $A \in \FF^{n \times n}$, we denote the vector of \textit{all} proportions by
\begin{align*}
  \zeta(A[\bm{\theta}]) = \bc{x(A[\bm{\theta}]), y(A[\bm{\theta}]), z(A[\bm{\theta}]), u(A[\bm{\theta}]), v(A[\bm{\theta}])}.
\end{align*}
\item For $A=\bm{T}_{n,t/n}$ and $\vw_{n,t/n} \in \{ \vx_{n,t/n}, \vy_{n,t/n}, \vz_{n,t/n}, \vu_{n,t/n}, \vv_{n,t/n}, \vze_{n,t/n}\}$, we simply write
\begin{align*}
    \vw_{n,t/n}= w\bc{\bm{T}_{n,t/n}[\bm{\theta}]}.
\end{align*}
\end{enumerate}
\end{definition}

\begin{remark}[Summation of proportions]\label{ra}
 By definition, for any matrix $A \in \FF^{n \times n}$, 
\begin{align}\label{eq_sumone}
  \|\zeta(A[\bm{\theta}])\|_1=  x(A[\bm{\theta}])+ y(A[\bm{\theta}])+ z(A[\bm{\theta}])+ u(A[\bm{\theta}])+ v(A[\bm{\theta}])=1.
\end{align}
Moreover, recall that $\val_{n,p}$ and $\hval_{n,p}$ denote the proportions of frozen variables among $[n]$ in ${\bm{T}}_{n,p}[\bm{\theta}]$ and ${\bm{T}}_{n,p}[\bm{\theta}]^T$, respectively. With the above definitions,
\begin{align}\label{eq_alphaxyuv}
    \val_{n,t/n}=\vx_{n,t/n}+\vy_{n,t/n}+\vv_{n,t/n} \qquad \text{and} \qquad \hval_{n,t/n}=\vx_{n,t/n}+\vy_{n,t/n}+\vu_{n,t/n}.
\end{align}
\end{remark}

With the notation of \Cref{def_proportions}, the main result of the remainder of \Cref{sec_prof} is the following  proposition:

\begin{proposition}[Stability of types]\label{l91}
For any $d>0$, and $\vw \in \{ \vx, \vy, \vz, \vu, \vv, \vze\}$,
\[\mathbb{E}\dabs{\vw_{n,t/n}-\vw_{n+1,t/n}}_1 = o_{n,P}(1), \ \mbox{uniformly in $t\in[0,d]$.}\]
\end{proposition}

In light of \Cref{l91} and \cref{eq_alphaxyuv}, it is tempting to conjecture that the proportions $\val_{n,d/n}$ of frozen variables of $\bm{T}_{n,d/n}[\bm{\theta}]$ converge in a suitable sense. Unfortunately, this conjecture turns out to be {incorrect}, and one of the implications of our present proof is that $\val_{n,d/n}$ does not converge for $d>\eul$. Despite this complication, the strictly weaker statement of \Cref{l91} is sufficient for our purposes.

The rest of \Cref{Sec_change_stat} is organized as follows: In \Cref{sec_exchange}, we study the impact of symmetric row- and column relabelling on the proper relations and variable types of a given matrix. In \Cref{sec_changefrozen}, building on Sections \ref{sec_prelifv}, \ref{sec_aar} and \ref{sec_exchange}, we prove that any fixed variable is unlikely to change from frozen to unfrozen, or the other way round, under one-step matrix growth of $\bm{T}_{n,d/n}[\bm{\theta}]$ and a related matrix. Finally, we present the proof of \Cref{l91} in \Cref{sec_sop}.

\subsubsection{Row- and column exchangeability}\label{sec_exchange}
In the following proofs, exchangeability arguments play an important role. We prepare these arguments in the current section. Throughout this section, for $k \in \NN$, let $\mathcal{S}_{k}$ denote the symmetric group of $[k]$.

\begin{definition}
Let $A\in\FF^{n \times n}$.
    For a permutation $\pi \in \mathcal{S}_{n}$, define the matrix $A^\pi$ by setting
    \begin{align}
        A^\pi(i,j) = A(\pi^{-1}(i),\pi^{-1}(j)), \qquad \text{ for } i,j \in [n].
    \end{align}
\end{definition}
$A^\pi$ is the matrix that arises from $A$ through joint relabelling of the rows and columns according to $i \mapsto \pi^{-1}(i)$. 

\begin{lemma}\label{lem_exc0}
For any $\pi \in \mathcal{S}_n$ and $p \in [0,1]$, $\bm{T}_{n,p}^{\pi}[\bm{\theta}] \stackrel{\text{d}}{=} \bm{T}_{n,p}[\bm{\theta}]$. 
\end{lemma}

\begin{proof}
Recall the definition of $\bm{T}_{n,p}$ in \cref{ec1_2}, according to which
\[\bm{T}_{n,p}(i,j)=\bm{A}_{N,p}(\vtu(i),\vtu(j)) \qquad \text{ for } i,j \in [n].\]
Hence,
\[\bm{T}^\pi_{n,p}(i,j)=\bm{T}_{n,p}(\pi^{-1}(i),\pi^{-1}(j))=\bm{A}_{N,p}(\vtu\circ \pi^{-1}(i),\vtu\circ \pi^{-1}(j)).\]
Since $\vtu$ is a uniform permutation of $[N]$, also $\vtu\circ \pi^{-1}$ is a uniform permutation of $[N]$, where we view $\pi$ as a permutation of $[N]$ that leaves $\{n+1, \ldots, N\}$ fixed. Thus,
\[\bm{T}_{n,p}^{\pi} \stackrel{\text{d}}{=} \bm{T}_{n,p}.\]
Independence of $\bm{T}_{n,p}$, $\vtu$ and the row- and column-perturbation matrices now implies that
$\bm{T}_{n,p}^{\pi}[\bm{\theta}] \stackrel{\text{d}}{=} \bm{T}_{n,p}[\bm{\theta}],$ 
as desired.
\end{proof}

\begin{corollary}\label{lem_exc}
Let $\pi\in \mathcal{S}_n$ and $I=\cbc{i_1,\ldots,i_k}\subset [n]$. Setting $I^\pi=\cbc{\pi(i_1),\ldots,\pi(i_k)}$, 
\begin{align*}
    \PP(\mbox{$I$ is a proper relation of $\bm{T}_{n,p}[\bm{\theta}]$})=\PP(\mbox{$I^\pi$ is a proper relation of $\bm{T}_{n,p}[\bm{\theta}]$}).
\end{align*}
\end{corollary}
\begin{proof}
    Note that
    \begin{align*}
        \mbox{$I$ is a proper relation of $\bm{T}_{n,p}[\bm{\theta}]$}\quad\Longleftrightarrow\quad \mbox{$I^\pi$ is a proper relation of $\bm{T}_{n,p}^\pi[\bm{\theta}]$}.
    \end{align*}
    The desired result now follows from \Cref{lem_exc0}.
\end{proof}

\begin{lemma}\label{cla_exc}
For any $A \in \FF^{n \times n}$, $\pi \in \mathcal{S}_n$, and $\mathcal W \in \{\mathcal{X}, \mathcal{Y}, \mathcal{Z}, \mathcal{U}, \mathcal{V}\}$,
\[i\in \mathcal{W}(A[\bm{\theta}])\quad\Longleftrightarrow\quad \pi(i)\in \mathcal{W}(A^\pi[\bm{\theta}]).\]
As a consequence,
\[\zeta(A[\bm{\theta}])=\zeta(A^\pi[\bm{\theta}]).\]
\end{lemma}
\begin{proof}
By the determinantal rank characterisation and the Leibniz determinant formula, the rank of $A$ stays unchanged under the permutation $\pi$.  By \Cref{p4},
\begin{align}\label{eqv_pi1}
    i\in \mathcal{F}(A)\  \Longleftrightarrow  \ \rank\bc{A}-\rank\bc{A\abc{;i}}=1\ \Longleftrightarrow  \  \rank\bc{A^\pi}-\rank\bc{A^\pi\abc{;\pi(i)}}=1 \ \Longleftrightarrow  \  \pi(i)\in \mathcal{F}(A^\pi).
\end{align}
Analogously, 
\begin{align}\label{eqv_pi2}
i\in \mathcal{F}(A\abc{i;})\quad \Longleftrightarrow \quad \pi(i)\in \mathcal{F}(A^\pi\abc{\pi(i);}).
\end{align}
The desired results now follow from \cref{eqv_pi1} and \cref{eqv_pi2}.
\end{proof}

In particular, we will make frequent use of the following corollary:

\begin{corollary}\label{cor_proportions}
    For any $i \in [n]$, $p \in [0,1]$ and $\mathcal W \in \{\mathcal{X}, \mathcal{Y}, \mathcal{Z}, \mathcal{U}, \mathcal{V}\}$,
\begin{align*}
    \PP\bc{i \in \cW\bc{\bm{T}_{n,p}[\bm{\theta}]}} = \Erw\brk{\bm w_{n,p}}.
\end{align*}
\end{corollary}

\begin{proof}
Let $\pi \in \mathcal S_n$ be the transposition of $n+1$ and $i$. Then Lemmas \ref{cla_exc} and \ref{lem_exc0} together imply that 
$$\PP\bc{n+1 \in \cW\bc{\bm{T}_{n,p}[\bm{\theta}]}} = \PP\bc{\pi(n+1) \in \cW\bc{\bm{T}_{n,p}^{\pi}[\bm{\theta}]}} = \PP\bc{i \in \cW\bc{\bm{T}_{n,p}[\bm{\theta}]}}$$   
and therefore 
\[\PP\bc{n+1 \in \cW\bc{\bm{T}_{n,p}[\bm{\theta}]}} =\Erw\brk{\frac{1}{n+1}\sum_{i=1}^{n+1}\ind\cbc{i \in \cW\bc{\bm{T}_{n,p}[\bm{\theta}]}}}=\Erw\brk{\vw_{n,p}}.\]
\end{proof}

\subsubsection{Freezing and unfreezing under row- and column removal}\label{sec_changefrozen}

Building upon the symmetry arguments of \Cref{sec_exchange}, we now prove the two main lemmas that are needed to attack \Cref{l91}, which makes a statement about the expected differences of the various proportions of variable types in $\bm{T}_{n,t/n}[\bm{\theta}]$ and $\bm{T}_{n+1,t/n}[\bm{\theta}]$. 
Since the type of variable $i \in [n]$ with respect to {the} matrix $A \in \FF^{n \times n}$ is defined solely in terms of the membership of $i$ in each of the sets $\cF(A), \cF(A^T)$, $\cF(A\abc{i;})$ and $\cF(A^T\abc{i;})$ (see Definitions \ref{dsc} and \ref{dxyzuv}) and  the matrices $\bm{T}_{n,t/n}[\bm{\theta}]$ and $\bm{T}_{n+1,t/n}[\bm{\theta}]$ are reasonably alike, it seems like a viable strategy to show that any given variable is unlikely to change its membership in each of the aforementioned sets of frozen variables in the transition from $\bm{T}_{n,t/n}[\bm{\theta}]$ to $\bm{T}_{n+1,t/n}[\bm{\theta}]$, which is precisely what we {show}.   
In this sense, \Cref{l7} shows that any fixed variable is unlikely to be frozen in  exactly one of the matrices 
$\bm{T}_{n,t/n}[\bm{\theta}]$ or $\bm{T}_{n+1,t/n}[\bm{\theta}]$:

\begin{lemma}[One-step matrix growth, original matrix]\label{l7}
Fix $d, \delta >0$ and $L \in \NN_{\geq 2}$. Then for any $i\in[n]$,
\[\mathbb{P}\left(i\in \mathcal{F}\bc{\bm{T}_{n,t/n}[\bm{\theta}]}\Delta\mathcal{F}\bc{\bm{T}_{n+1,t/n}[\bm{\theta}]}\right)\leq 2(L+1)!\delta+\mathbb{P}\left({\rm Po}(d)\geq L\right)+o_{n,P}(1),\ \mbox{uniformly in $t\in[0,d]$}. \]
\end{lemma}

Conveniently, the pair $(\bm{T}_{n,t/n}[\bm{\theta}],\bm{T}_{n+1,t/n}[\bm{\theta}])$ is identically distributed to $(\bm{T}_{n,t/n}[\bm{\theta}]^T,\bm{T}_{n+1,t/n}[\bm{\theta}]^T)$, so it is enough to work with non-transposed matrices in the above considerations. In the same spirit, \Cref{l71} shows that any fixed variable is unlikely to be frozen in exactly one of the matrices 
$\bm{T}_{n,t/n}[\bm{\theta}]\abc{i;}$ or $\bm{T}_{n+1,t/n}[\bm{\theta}]\abc{i;}$:

\begin{lemma}[One-step matrix growth, row-deleted matrix]\label{l71}
Fix $d, \delta >0$ and $L \in \NN_{\geq 2}$. Then for any $i\in[n]$,
\[\mathbb{P}\left(i\in\mathcal{F}\bc{\bm{T}_{n,t/n}[\bm{\theta}]\abc{i;}}\Delta \mathcal{F}\bc{\bm{T}_{n+1,t/n}[\bm{\theta}]\abc{i;}}\right) \leq 2(L+2)!\delta+(L+2)\mathbb{P}\bc{\poi{d}\geq L}+o_{n,P}(1),\]
uniformly in $t\in[0,d]$.
\end{lemma}

While \Cref{l71} is structurally similar to \Cref{l7}, its proof proceeds differently. This is due to the fact that the removal of row $i$ makes this index special, and the exchangeability arguments that are used in the proof of \Cref{l7} do not apply directly to the modified setting. We use \Cref{p4,l5} to overcome this problem.

Finally, we also prove a third lemma which shows that a small deterministic increase in $\bm{\theta}$ is unlikely to change whether variable $i$ is frozen or not. \Cref{l1.1} is not used in the proof of \Cref{l91} and will only become relevant in Section \ref{sec_pc}, but since it is similar in spirit to the previous two lemmas, we include it here:

\begin{lemma}[Deterministic perturbation shift]\label{l1.1}
Let $\mu=(\mu_r,\mu_c)\in \mathbb{N}^2$, $A\in \mathbb{F}^{m\times n}$ and $i\in [n]$. Then 
\[\mathbb{P}\left(i\in \mathcal{F}\bc{A[\bm{\theta}]}\Delta \mathcal{F}\bc{A[\bm{\theta}+\mu]}\right) \leq \frac{\mu_r+\mu_c}{P}.\]
\end{lemma}

While the proofs of \Cref{l7,l71} heavily depend on the structure of $\bm{T}_{n,t/n}$,
the proof of \Cref{l1.1} only uses properties of the perturbation, and thus the result is true for arbitrary matrices.

\textbf{Two good events.} Before we turn to the proofs of \Cref{l7,l71,l1.1}, we define two \textit{good}  events that will be used here and later throughout the article. For $p \in [0,1]$, let
\begin{align}\label{event_R}
    \mathfrak{R}_{n,p}=\cbc{\text{{both} $\bm{T}_{n,p}[\bm{\theta}]$ and $\bm{T}_{n,p}[\bm{\theta}]^T$ are $(\delta,\ell)$-free for  $2\leq \ell \leq L$}},
\end{align}
and
\begin{align}\label{event_P}
\mathfrak{P}_n=\cbc{\Theta_r[\bm{\theta}_r,n|n]=\Theta_r[\bm{\theta}_r,n+1|n+1]\abc{;n+1},\Theta_c[n|n,\bm{\theta}_c]=\Theta_c[n+1|n+1,\bm{\theta}_c]\abc{n+1;}}.
\end{align}
$\mathfrak R_{n,p}$ ensures that the rank increase upon attaching rows and columns can be controlled as in \Cref{l5}, while the benefit of $\mathfrak{P}_n$ is that when growing the matrix from $n$ to $n+1$, the perturbation stays unchanged. 
By \Cref{p1} and \Cref{lc1},
\begin{equation}\label{oe}
\mathbb{P}\left({\mathfrak{R}_{n,p}}\right)\geq 1+o_{n,P}(1)
\qquad \text{and} \qquad \mathbb{P}\left({\mathfrak{P}_n}\right)= 1+o_n(1).
\end{equation}
The bound $\mathbb{P}\left({\mathfrak{R}_{n,p}^c}\right)\leq o_{n,P}(1)$ holds uniformly in $p \in [0,1]$, since it is based on \Cref{p1}.
In the following, we frequently work on the intersection of $\mathfrak R_{n,t/n}$ and $\mathfrak{P}_n$, which is a sufficiently likely event by (\ref{oe}).

We now prove \Cref{l7,l71,l1.1} in their order of appearance:

\begin{proof}[Proof of \Cref{l7}]
Since $i$ can either freeze or unfreeze when the matrix is grown, 
\begin{align*}
    &\ \PP\bc{i\in \mathcal{F}\bc{\bm{T}_{n,t/n}[\bm{\theta}]}\Delta\mathcal{F}\bc{\bm{T}_{n+1,t/n}[\bm{\theta}]}}\\
    =&\ \PP\bc{i\in \mathcal{F}\bc{\bm{T}_{n,t/n}[\bm{\theta}]}\backslash \mathcal{F}\bc{\bm{T}_{n+1,t/n}[\bm{\theta}]}} + \PP\bc{i\in \mathcal{F}\bc{\bm{T}_{n+1,t/n}[\bm{\theta}]}\backslash \mathcal{F}\bc{\bm{T}_{n,t/n}[\bm{\theta}]}}.
\end{align*}
We bound both cases separately.

\begin{enumerate}[label=(\roman*),left=1pt]
  \item \textbf{Unfreezing}: To bound the probability that $i$ is frozen in $\bm{T}_{n,t/n}[\bm{\theta}]$, but not in $\bm{T}_{n+1,t/n}[\bm{\theta}]$, 
  we first show that on $\mathfrak{P}_n$,
  \begin{align}\label{implication_unfreeze}
  i\in \mathcal{F}\bc{\bm{T}_{n,t/n}[\bm{\theta}]}\backslash \mathcal{F}\bc{\bm{T}_{n+1,t/n}[\bm{\theta}]} \qquad \Longrightarrow \qquad \cbc{i,n+1} \text{ is a proper relation in } \bm{T}_{n+1,t/n}[\bm{\theta}].\end{align}

  Assume that $\mathfrak P_n$ holds and $i \in \mathcal{F}\bc{\bm{T}_{n,t/n}[\bm{\theta}]}\backslash \mathcal{F}(\bm{T}_{n+1,t/n}[\bm{\theta}])$. Then $\bm{T}_{n+1,t/n}[\bm{\theta}]$ arises from $\bm{T}_{n,t/n}[\bm{\theta}]$ through the symmetric attachment of a row and a column, which we may break into two steps. By \Cref{l1.1b}, attaching a row cannot unfreeze $i$, i.e.,
  \[i\in \mathcal{F}\bc{\bm{T}_{n,t/n}[\bm{\theta}]} \qquad  \Longrightarrow \qquad i\in \mathcal{F}\bc{\bm{T}_{n+1,t/n}[\bm{\theta}]\abc{;n+1}}.\]
  In particular, there exists a representation $y$ of $\{i\}$ in $\bm{T}_{n+1,t/n}[\bm{\theta}]\abc{;n+1}$. Attaching column $n+1$ and using the representation $y$ on the resulting matrix $\bm{T}_{n+1,t/n}[\bm{\theta}]$ yields
  \[\cbc{i}\subseteq \supp{y\bm{T}_{n+1,t/n}[\bm{\theta}]} \subseteq \cbc{i,n+1}.\]
  Since $i\not\in \mathcal{F}\bc{\bm{T}_{n+1,t/n}[\bm{\theta}]}$ by assumption, we conclude that
  \[\supp{y\bm{T}_{n+1,t/n}[\bm{\theta}]} = \cbc{i,n+1},\]
which implies that $\{i,n+1\}$ is a proper relation in $\bm{T}_{n+1,t/n}[\bm{\theta}]$, since the existence of the representation $y$ ensures that $n+1$ cannot be frozen in $\bm{T}_{n+1,t/n}[\bm{\theta}]$ without $i$ being frozen in $\bm{T}_{n+1,t/n}[\bm{\theta}]$. This proves \cref{implication_unfreeze}.

The next step is to show that, on the  good  event $\mathfrak R_{n+1,t/n}$, the probability that $\{i,n+1\}$ forms a proper relation in $\bm{T}_{n+1,t/n}[\bm{\theta}]$ is small: 
This is an immediate consequence of \Cref{lem_exc}, which asserts that the probability to be a proper relation is the same for any pair $\{i_1,i_2\}$ for $1\leq i_1<i_2\leq n+1$ and the observation that on $\mathfrak R_{n+1,t/n}$, there are at most $\delta (n+1+P)^2$ proper relations of length two. 
Therefore,
\begin{align*}
\mathbb{P}\bc{\cbc{i,n+1} \text{ is a proper relation in } \bm{T}_{n+1,t/n}[\bm{\theta}], \mathfrak R_{n+1,t/n}}
\leq 2\delta +o_{n,P}(1),\ \mbox{uniformly in $t \in [0,d]$}
\end{align*}
and thus also
\begin{equation}\label{el701}
\mathbb{P}\left(i\in \mathcal{F}\bc{\bm{T}_{n,t/n}[\bm{\theta}]} \setminus \mathcal{F}\bc{\bm{T}_{n+1,t/n}[\bm{\theta}]}, \mathfrak P_n, \mathfrak R_{n+1,t/n}\right)\leq 2\delta+o_{n,P}(1),\ \mbox{uniformly in $t \in [0,d]$}.
\end{equation}

  \item \textbf{Freezing}: 
 To bound the probability that $i$ is frozen in $\bm{T}_{n+1,t/n}[\bm{\theta}]$, but not in $\bm{T}_{n,t/n}[\bm{\theta}]$, 
  we show that on $\mathfrak P_n$,
 \begin{align}\label{implication_freeze}
    & i\in \mathcal{F}\bc{\bm{T}_{n+1,t/n}[\bm{\theta}]}\backslash \mathcal{F}\bc{\bm{T}_{n,t/n}[\bm{\theta}]},  \quad  i\notin \supp{\bm{T}_{n+1,t/n}(n+1,)} \nonumber \\
     \Longrightarrow  \quad &\cbc{i}\cup \supp{\bm{T}_{n+1,t/n}[\bm{\theta}](n+1,)} \text{ is a proper relation in } \bm{T}_{n,t/n}[\bm{\theta}].
 \end{align}
 Assume that $\mathfrak P_n$ holds, $i \in \mathcal{F}(\bm{T}_{n+1,t/n}[\bm{\theta}])\backslash \mathcal{F}\bc{\bm{T}_{n,t/n}[\bm{\theta}]}$ and $i\notin \supp{\bm{T}_{n+1,t/n}(n+1,)}$. Then the matrix $\bm{T}_{n,t/n}[\bm{\theta}]$ arises from $\bm{T}_{n+1,t/n}[\bm{\theta}]$ through symmetric removal of a column and a row, which we may break into two steps. 
 By \Cref{l1.1b}, removing a column cannot unfreeze $i$:
  \[i\in \mathcal{F}\bc{\bm{T}_{n+1,t/n}[\bm{\theta}]} \qquad  \Longrightarrow \qquad i\in \mathcal{F}\bc{\bm{T}_{n+1,t/n}[\bm{\theta}]\abc{;n+1}}.\]
  This implies in particular that there exists a representation $y$ of $\{i\}$ in $\bm{T}_{n+1,t/n}[\bm{\theta}]\abc{;n+1}$. 
If there was a representation $y$ of $\{i\}$ in $\bm{T}_{n+1,t/n}[\bm{\theta}]\abc{;n+1}$ with $y_{n+1} =0$, {then }shortening $y$ to $y\abc{;n+1}$ would be a representation of $\{i\}$ in $\bm{T}_{n,t/n}[\bm{\theta}]$, in contrast to our assumption that $i$ is not frozen in $\bm{T}_{n,t/n}[\bm{\theta}]$. Thus, all representations $y$ of $\{i\}$ in $\bm{T}_{n+1,t/n}[\bm{\theta}]\abc{;n+1}$ have their $(n+1)$st coordinate different from zero. Since we assume that $i\notin \supp{\bm{T}_{n+1,t/n}(n+1,)}$, we conclude that
  \[\supp{y\abc{;n+1}\bm{T}_{n,t/n}[\bm{\theta}]} = \cbc{i}\cup\supp{\bm{T}_{n+1,t/n}[\bm{\theta}](n+1,)}.\]
  This implies that $\cbc{i}\cup\supp{\bm{T}_{n+1,t/n}[\bm{\theta}](n+1,)}$ is a proper relation of $\bm{T}_{n,t/n}[\bm{\theta}]$, since it contains the non-frozen variable $i$. Thus, \cref{implication_freeze} holds.

The next step is to show that on the  good event $\mathfrak R_{n,t/n}$, the probability that $\{i\}\cup\supp{\bm{T}_{n+1,t/n}[\bm{\theta}](n+1,)}$ forms a proper relation in $\bm{T}_{n,t/n}[\bm{\theta}]$ is small.
We first upper-bound the probability that row $n+1$ has too many non-zero entries, which is due to the sparsity of the matrix $\bm{T}_{n+1,t/n}$. By \cite[Theorem 2.10]{van2017random}, we can upper bound
 \begin{align*}
  \mathbb{P}\bc{\abs{\supp{\bm{T}_{n+1,t/n}[\bm{\theta}](n+1,)}}\geq L}\leq&\mathbb{P}\bc{\bin{n,t/n}\geq L}+ \PP\bc{\mathfrak P_n^c}\\
  \leq& 
  \mathbb{P}\bc{\poi{t}\geq L}+ \PP\bc{\mathfrak P_n^c}+ o_n(1) \leq \mathbb{P}\bc{\poi{d}\geq L} + o_n(1)
  \end{align*}
  uniformly in $t\in[0,d]$. 
On the other hand, by \Cref{lem_exc}, the probability to be a proper relation in $ \bm{T}_{n,t/n}[\bm{\theta}]$ is the same for any subset of $[n]$ of cardinality $\abs{\{i\}\cup\supp{\bm{T}_{n+1,t/n}[\bm{\theta}](n+1,)}}$.
If row $n+1$ has at most $L-1$ non-zero entries 
and $\mathfrak R_{n,t/n}$ holds, {then} $ \bm{T}_{n,t/n}[\bm{\theta}]$ is 
$(\delta,\abs{\{i\}\cup\supp{\bm{T}_{n+1,t/n}[\bm{\theta}](n+1,)}})$-free, and 
{\begin{align*}
&\mathbb{P}\bc{\{i\}\cup\supp{\bm{T}_{n+1,t/n}[\bm{\theta}](n+1,)}\in \PR\bc{ \bm{T}_{n,t/n}[\bm{\theta}]}, \abs{\{i\}\cup\supp{\bm{T}_{n+1,t/n}[\bm{\theta}](n+1,)}}\leq L, \mathfrak R_{n,t/n}} \\ 
& \leq L! \delta +o_{n,P}(1), \ \mbox{uniformly in $t \in [0,d]$}.
\end{align*}}
Finally, since 
 {\begin{align*}
      \PP\bc{i\in \supp{\bm{T}_{n+1,t/n}(n+1,)}}=\PP\bc{\bm{T}_{n+1,t/n}(n+1,i)= 1}=t/n\leq d/n = o_n(1)  \ \mbox{uniformly in $t \in [0,d]$},
  \end{align*}}
we conclude {from \cref{implication_freeze}} that
  \begin{equation}\label{el702}
\mathbb{P}\left(i\in \mathcal{F}\bc{\bm{T}_{n+1,t/n}[\bm{\theta}]} \setminus \mathcal{F}\bc{\bm{T}_{n,t/n}[\bm{\theta}]}, \mathfrak P_n, \mathfrak R_{n+1,t/n}\right)\leq (L+1)!\delta +\mathbb{P}\left({\rm Po}(d)\geq L\right)+o_{n,P}(1)
\end{equation}
uniformly in $t \in [0,d]$.
\end{enumerate}

Combining (\ref{oe}), (\ref{el701}) and (\ref{el702}) 
finishes the proof of \Cref{l7}.
\end{proof}

\begin{proof}[Proof of \Cref{l71}]
Again, we relate a status change of $i$ to the existence of a proper relation: On a sufficiently likely event $\mathfrak S_n$,
\begin{align}\label{eq_rel_imp}
    &i\in \mathcal{F}\bc{\bm{T}_{n,t/n}[\bm{\theta}]\abc{i;}} \Delta \mathcal{F}\bc{\bm{T}_{n+1,t/n}[\bm{\theta}]\abc{i;}} \nonumber \\
    \Longrightarrow \quad & \supp{\bm{T}_{n,t/n}(,i)} \text{ is a proper relation in } \bm{T}_{n,t/n}[\bm{\theta}]\abc{i;i}^T \text{ or in } \bm{T}_{n+1,t/n}[\bm{\theta}]\abc{i;i}^T.
\end{align}
\textit{Definition of $\mathfrak S_n$.} The event that we work on is composed of three parts: First, we define
\begin{align*}
    \mathfrak{S}_{n,1}= \mathfrak P_n \cap \cbc{\bm{T}_{n+1,t/n}[\bm{\theta}](n+1,i)=0, \Theta_r[\bm{\theta}_r, n \vert n](,i)=0_{\bm{\theta}_r \times 1}, \Theta_c[n \vert n,\bm{\theta}_c](i,k)=0_{1\times \bm{\theta}_c}}.
\end{align*}

On $\mathfrak S_{n,1}$, 
the non-zero entries of column (or equivalently row) $i$ in all involved matrices are contained in $[n]$, i.e.,
\[\supp{\bm{T}_{n,t/n}(,i)}=\supp{\bm{T}_{n,t/n}[\bm{\theta}](,i)}=\supp{\bm{T}_{n+1,t/n}[\bm{\theta}](,i)}=\supp{\bm{T}_{n+1,t/n}(,i)}.\]
From the construction of the perturbation \Cref{lc1} and the definition of $\bm{T}_{n+1,t/n}$, it is immediate that
\[\PP\bc{\mathfrak S_{n,1}^c}\leq \frac{2P}{n+1}+\frac{d}{n}+\frac{2P}{n}=o_n(1), \ \mbox{uniformly in $t \in [0,d]$}. \]


Next, let
\begin{align*}
    \mathfrak{S}_{n,2}=\cbc{\text{{for} all } j\in {\rm supp}(\bm{T}_{n,t/n}[\bm{\theta}](,i)): j \notin \mathcal{F}\bc{\bm{T}_{n,t/n}[\bm{\theta}]\abc{i;i}^T} \Delta \mathcal{F}\bc{\bm{T}_{n+1,t/n}[\bm{\theta}]\abc{i;i}^T}}
\end{align*}
be the event that no element of the support of column $i$ has a different status in $\bm{T}_{n+1,t/n}[\bm{\theta}]\abc{i;i}^T$ than in $\bm{T}_{n,t/n}[\bm{\theta}]\abc{i;i}^T$. 
Since $(\bm{T}_{n,t/n}[\bm{\theta}]\abc{i;i}^T,\bm{T}_{n+1,t/n}[\bm{\theta}]\abc{i;i}^T)$ conditionally on $\mathfrak{S}_{n,1}$ and $\bc{\bm{T}_{n-1,t/n}[\bm{\theta}],\bm{T}_{n,t/n}[\bm{\theta}]}$ conditionally on $\mathfrak{P}_{n-1}$ have the same law, by \Cref{l7} and \Cref{lem_exc},
\begin{align*}
 \PP(\mathfrak S_{n,2}^c) 
  & \leq L\mathbb{P}\left(1 \in\mathcal{F}\bc{\bm{T}_{n-1,t/n}[\bm{\theta}]}\Delta \mathcal{F}\bc{\bm{T}_{n,t/n}[\bm{\theta}]}\right) + \PP(|\text{supp}(\bm{T}_{n,t/n}[\bm{\theta}](,i))| > L) + \PP(\mathfrak{S}_{n,1}^c) +\PP\bc{\mathfrak{P}_{n-1}^c} \\
  &\leq (2+(L+1)!)L\delta+(L+1)\mathbb{P}\left({\rm Po}(d)\geq L\right)+o_{n,P}(1), \ \mbox{uniformly in $t \in [0,d]$}.
  \end{align*}
Finally, let
\begin{align*}
    \mathfrak{S}_{n,3}=\cbc{\bm{T}_{n,t/n}[\bm{\theta}]\abc{i;i}^T \text{ and } \bm{T}_{n+1,t/n}[\bm{\theta}]\abc{i;i}^T \text{ are } (\delta,\ell)\text{-free for } 2\leq \ell\leq L}.
\end{align*}
Since $(\bm{T}_{n,t/n}[\bm{\theta}]\abc{i;i}^T,\bm{T}_{n+1,t/n}[\bm{\theta}]\abc{i;i}^T)$ conditionally on $\mathfrak{S}_{n,1}$ and $\bc{\bm{T}_{n-1,t/n}[\bm{\theta}],\bm{T}_{n,t/n}[\bm{\theta}]}$ conditionally on $\mathfrak{P}_{n-1}$ have the same law, \Cref{p1} implies that
\[\PP(\mathfrak S_{n,3}^c) = o_{n,P}(1), \ \mbox{uniformly in $t \in [0,d]$}. \]
We set $\mathfrak S_n = \mathfrak S_{n,1} \cap \mathfrak S_{n,2} \cap \mathfrak S_{n,3}$, so that
\begin{align}\label{eq_l71_1}
    \PP(\mathfrak S_n^c) \leq (2+(L+1)!)L\delta+(L+1)\mathbb{P}\left({\rm Po}(d)\geq L\right)+o_{n,P}(1), \ \mbox{uniformly in $t \in [0,d]$}.
\end{align}

\textit{Proof of implication (\ref{eq_rel_imp}).} Now suppose that $\mathfrak S_n$ holds and that supp$(\bm{T}_{n,t/n}(,i))$ is neither a proper relation in $ \bm{T}_{n,t/n}[\bm{\theta}]\abc{i;i}^T$ nor in  $ \bm{T}_{n+1,t/n}[\bm{\theta}]\abc{i;i}^T$. Since the support may contain frozen variables, there are four cases:
\begin{itemize}
    \item[\bf{Case 1:}] supp$(\bm{T}_{n,t/n}(,i))$ neither has a representation in  $ \bm{T}_{n,t/n}[\bm{\theta}]\abc{i;i}^T$ nor in $ \bm{T}_{n+1,t/n}[\bm{\theta}]\abc{i;i}^T$. \\
    The non-existence of a representation of supp$(\bm{T}_{n,t/n}\abc{i;}(,i))$ in $\bm{T}_{n,t/n}[\bm{\theta}]\abc{i;i}^T$ in particular implies that column $i$ of $\bm{T}_{n,t/n}[\bm{\theta}]\abc{i;}$ is not in the linear span of the other columns of $\bm{T}_{n,t/n}[\bm{\theta}]\abc{i;}$. Thus, by \Cref{p4}, $i$ is frozen in $\bm{T}_{n,t/n}[\bm{\theta}]\abc{i;}$. The same reasoning implies that $i$ is frozen in $\bm{T}_{n+1,t/n}[\bm{\theta}]\abc{i;}$ as well. 
    \item[\bf{Case 2:}]   supp$(\bm{T}_{n,t/n}(,i))$ has a representation both in  $ \bm{T}_{n,t/n}[\bm{\theta}]\abc{i;i}^T$ and in $ \bm{T}_{n+1,t/n}[\bm{\theta}]\abc{i;i}^T$. \\
   Since we assume that supp$(\bm{T}_{n,t/n}(,i))$ is neither a proper relation in $ \bm{T}_{n,t/n}[\bm{\theta}]\abc{i;i}^T$ nor in  $ \bm{T}_{n+1,t/n}[\bm{\theta}]\abc{i;i}^T$, all variables in $\emptyset \not=$ supp$(\bm{T}_{n,t/n}(,i))$ must be frozen both in $ \bm{T}_{n,t/n}[\bm{\theta}]\abc{i;i}^T$ and in $ \bm{T}_{n+1,t/n}[\bm{\theta}]\abc{i;i}^T$. In this case, the existence of the respective representations ensures that  column $i$ of $\bm{T}_{n,t/n}[\bm{\theta}]\abc{i;}$ is contained in the linear span of the other columns of $\bm{T}_{n,t/n}[\bm{\theta}]\abc{i;}$ and that  column $i$ of $\bm{T}_{n+1,t/n}[\bm{\theta}]\abc{i;}$ is contained in the linear span of the other columns of $\bm{T}_{n+1,t/n}[\bm{\theta}]\abc{i;}$. Thus, by \Cref{p4}, $i$ is neither frozen in $\bm{T}_{n,t/n}[\bm{\theta}]\abc{i;}$ nor in $\bm{T}_{n+1,t/n}[\bm{\theta}]\abc{i;}$.
    \item[\bf{Case 3:}]  supp$(\bm{T}_{n,t/n}(,i))$ has a representation in  $ \bm{T}_{n,t/n}[\bm{\theta}]\abc{i;i}^T$, but none in $ \bm{T}_{n+1,t/n}[\bm{\theta}]\abc{i;i}^T$. \\
   Again, all variables in $\emptyset \not=$ supp$(\bm{T}_{n,t/n}(,i))$ must be frozen in $ \bm{T}_{n,t/n}[\bm{\theta}]\abc{i;i}^T$, but there must {exist} a variable that is not frozen in $ \bm{T}_{n+1,t/n}[\bm{\theta}]\abc{i;i}^T$.
    This possibility is excluded by $\mathfrak S_{n,2}$.
    \item[\bf{Case 4:}] supp$(\bm{T}_{n,t/n}(,i))$ has a representation in  $ \bm{T}_{n+1,t/n}[\bm{\theta}]\abc{i;i}^T$, but none in $ \bm{T}_{n,t/n}[\bm{\theta}]\abc{i;i}^T$. \\
    By the same reasoning as in case 3, this cannot happen on $\mathfrak S_{n,2}$.
\end{itemize}
Cases 1 to 4 imply (\ref{eq_rel_imp}), which gives
\begin{align*}
&\PP(i\in \mathcal{F}\bc{\bm{T}_{n,t/n}[\bm{\theta}]\abc{i;}} \Delta \mathcal{F}\bc{\bm{T}_{n+1,t/n}[\bm{\theta}]\abc{i;}})\\
\leq &
     \PP\bc{\mathfrak S_n, \text{supp}(\bm{T}_{n,t/n}(,i)) \text{ is a proper relation in } \bm{T}_{n,t/n}[\bm{\theta}]\abc{i;i}^T \text{ or in } \bm{T}_{n+1,t/n}[\bm{\theta}]\abc{i;i}^T} + \PP\bc{\mathfrak S_n^c}.
\end{align*}
Finally, since $\bm{T}_{n,t/n}[\bm{\theta}]\abc{i;i}^T$ and $\bm{T}_{n+1,t/n}[\bm{\theta}]\abc{i;i}^T$ are $(\delta,\ell)$-free for $2\leq \ell\leq L$ on $\mathfrak S_n$, by \Cref{l5},
\begin{equation}\label{eq_l71_2}\begin{aligned}
&\PP(\mathfrak S_n, \text{supp}(\bm{T}_{n,t/n}(,i)) \text{ is a proper relation in } \bm{T}_{n,t/n}[\bm{\theta}]\abc{i;i}^T \text{ or in } \bm{T}_{n+1,t/n}[\bm{\theta}]\abc{i;i}^T) \\
\leq &2L!\delta+\mathbb{P}\left({\rm Po}(d)\geq L\right) + o_n(1),  \ \mbox{uniformly in $t \in [0,d]$}.
\end{aligned}
\end{equation}
Combining (\ref{eq_l71_1}) and (\ref{eq_l71_2}) yields the claim.
\end{proof}

\begin{proof}[Proof of \Cref{l1.1}]
The matrix $A[\bm{\theta}+\mu]$ arises from $A[\bm{\theta}]$ through the attachment of $\mu_r$ independent unit rows and $\mu_c$ independent unit columns. We split this row- and column-attachment into two steps. Since $A$ is non-random,
\begin{align}\label{eq_l1.1_1}
\dtv{\cF(A[\bm{\theta}]),\cF(A[\bm{\theta}+(0,\mu_c)])}\leq \dtv{\bm{\theta},\bm{\theta}+(0,\mu_c)} \leq \frac{\mu_c}{P},
\end{align}
and
\begin{align}\label{eq_l1.1_2}
  \dtv{\cF(A[\bm{\theta}+(0,\mu_c)]),\cF(A[\bm{\theta}+\mu])}\leq \frac{\mu_r}{P}.  
\end{align}
By \Cref{l1.1b}, increasing the number of columns can only diminish the number of frozen variables among the first $n$. 
Therefore, for $i \in [n]$, using (\ref{eq_l1.1_1}),
\begin{align}\label{eq_l1.1_3}
\mathbb{P}\left(i\in \mathcal{F}\bc{A[\bm{\theta}]}\Delta \mathcal{F}\bc{A[\bm{\theta}+(0,\mu_c)]}\right)= \mathbb{P}\left(i\in \mathcal{F}\bc{A[\bm{\theta}]}\right)-\mathbb{P}\left(i\in\mathcal{F}\bc{A[\bm{\theta}+(0,\mu_c)]}\right) \leq \frac{\mu_c}{P}. 
\end{align}
Similarly, also by \Cref{l1.1b}, increasing the number of rows can only enlarge the number of frozen variables among the first $n$. 
Therefore, using (\ref{eq_l1.1_2}),
\begin{align}\label{eq_l1.1_4}
\mathbb{P}\left(i \in \mathcal{F}\bc{A[\bm{\theta}+(0,\mu_c)]}\Delta \mathcal{F}\bc{A[\bm{\theta}+\mu]}\right)= \mathbb{P}\left(i\in\mathcal{F}\bc{A[\bm{\theta}+\mu)]}\right) - \mathbb{P}\left(i\in\mathcal{F}\bc{A[\bm{\theta}+(0,\mu_c)]}\right) \leq \frac{\mu_r}{P}. 
\end{align}
The claim follows by combining (\ref{eq_l1.1_3}) and (\ref{eq_l1.1_4}).
\end{proof}

As the final result of this subsection, we note an immediate consequence of \Cref{l1.1}, that will be used in the proof of \Cref{l12} {below}. \Cref{l9} shows that for any $i \in [n]$, removal of a bounded number of uniformly chosen rows is unlikely to unfreeze $i$, even if row $i$ is forbidden to be among the removed rows:

\begin{corollary}[Random row-removal]\label{l9}
For any $A \in \FF^{m\times n}$, $k \in [m], i \in [n]$ and a uniformly chosen $k$-subset $\bm{\mathcal J} \subseteq [m]\backslash\cbc{i}$,
\begin{align}\label{l9_part1}
    \mathbb{P}\left(i\in\mathcal{F}\bc{A[\bm{\theta}]}\Delta \mathcal{F}\bc{A[\bm{\theta}]\abc{\bm{\mathcal J};}}\right)\leq \frac{k}{P}+\bc{1-\frac{1}{m}}^k\frac{k(k-1)}{2m} + \frac{k}{m}.
\end{align}
\end{corollary}

\begin{proof}
Let $i \in [n]$. The proof is based on \Cref{l8.4,l1.1}: First, \Cref{l8.4} shows that row removal has the same effect on whether $i$ is frozen or not as addition of a unit column vector. The latter operation then can be treated as a slight change in the column perturbation, and therefore falls under the scope of \Cref{l1.1}. 

We first replace $\bm{\mathcal J}$ by a set obtained from sampling with replacement.
Let $\vj'_1, \ldots, \vj'_{k} \in [m]$ be 
i.i.d.\ uniform indices and $\bm{\mathcal J}'=\cup_{s=1}^{k}\left\{\vj'_s\right\}$ such that $\PP(\bm{\mathcal J} \not= \bm{\mathcal J}') = {\rm d}_{\rm TV}(\bm{\mathcal J},\bm{\mathcal J}')$ (i.e., we take an optimal coupling). Then
\begin{align}\label{el9.1}
\mathbb{P}\left(i\in\mathcal{F}\bc{A[\bm{\theta}]}\Delta \mathcal{F}\bc{A[\bm{\theta}]\abc{\bm{\mathcal J};}}\right)\leq \mathbb{P}\left(i\in\mathcal{F}\bc{A[\bm{\theta}]}\Delta \mathcal{F}\bc{A[\bm{\theta}]\abc{\bm{\mathcal J}';}}\right)+{\rm d}_{\rm TV}(\bm{\mathcal J},\bm{\mathcal J}').
\end{align}
Furthermore, an application of \Cref{l8.4} to (\ref{el9.1}) gives
\begin{align}\label{eq_l9_last}
  \mathbb{P}\left(i\in\mathcal{F}\bc{A[\bm{\theta}]}\Delta \mathcal{F}\bc{A[\bm{\theta}]\abc{\bm{\mathcal J};}}\right)\leq \PP(i\in \mathcal{F}\bc{A[\bm{\theta}]}\Delta \mathcal{F}\bc{A[\bm{\theta}+(0,k)]})+{\rm d}_{\rm TV}(\bm{\mathcal J},\bm{\mathcal J}').
\end{align}
Since ${\rm d}_{\rm TV}(\bm{\mathcal J},\bm{\mathcal J}') \leq (1-1/m)^k k(k-1)/(2m) + k/m$ (see \cite{freedman1977sampling} in combination with the observation that $\bm{\mathcal J}'$ samples from $[m]$ rather than $[m]\setminus \{i\}$, for example), the claim now follows from (\ref{eq_l9_last}) and \Cref{l1.1}.
\end{proof}

\subsubsection{Stability of types: Proof of \Cref{l91}}\label{sec_sop}

With \Cref{l7,l71}, we are now in the position to prove \Cref{l91}:

\begin{proof}[Proof of \Cref{l91}]
Observe that for any ${\mathcal W} \in \{\mathcal X, \mathcal Y, \mathcal Z, \mathcal U, \mathcal V\}$, the sequence $(\ind\cbc{i \in \mathcal W\bc{\bm{T}_{n,t/n}[\bm{\theta}]}}-\ind\cbc{i \in \mathcal W\bc{\bm{T}_{n+1,t/n}[\bm{\theta}]}})_{i\in [n]}$ consists of identically distributed random variables: This is a consequence of the fact that $\bm{T}_{n,t/n}$ is a submatrix of $\bm{T}_{n+1,t/n}$, \Cref{lem_exc0} and \Cref{cla_exc}.
Therefore,
\begin{align}\label{eq_l91_x1}
\mathbb{E}\abs{\vw_{n,t/n}-\vw_{n+1,t/n}}=&
\frac{1}{n}\mathbb{E}\abs{\sum_{i=1}^n \bc{\ind\cbc{i \in \mathcal W\bc{\bm{T}_{n,t/n}[\bm{\theta}]}}-\ind\cbc{i \in \mathcal W\bc{\bm{T}_{n+1,t/n}[\bm{\theta}]}}}}+o_n(1) \nonumber\\
& \leq \mathbb{E}\abs{\ind\cbc{1 \in \mathcal W\bc{\bm{T}_{n,t/n}[\bm{\theta}]}}-\ind\cbc{1 \in \mathcal W\bc{\bm{T}_{n+1,t/n}[\bm{\theta}]}}}+o_n(1).
\end{align}
We next bound \cref{eq_l91_x1} for the different types separately:

\textbf{Case 1}: ${\bm w}_{n,t/n} = \vx_{n,t/n}$.

\Cref{r80} yields the identity
\begin{align}\label{eq_l91_x2}
    \ind\cbc{1 \in \mathcal X\bc{\bm{T}_{n,t/n}[\bm{\theta}]}} =  \ind\cbc{1 \in \mathcal F\bc{\bm{T}_{n,t/n}[\bm{\theta}]} \Delta \mathcal F\bc{\bm{T}_{n,t/n}[\bm{\theta}] \abc{1;}} }.
\end{align}
Using $(\mathfrak B_1\Delta \mathfrak B_2)\Delta(\mathfrak B_3\Delta \mathfrak B_4)\subseteq (\mathfrak B_1\Delta \mathfrak B_3)\cup (\mathfrak B_2\Delta \mathfrak B_4)$ for any sets $\mathfrak B_1, \mathfrak B_2, \mathfrak B_3, \mathfrak B_4$ and plugging (\ref{eq_l91_x2}) into (\ref{eq_l91_x1}) yields the upper bound 
\begin{equation}\label{eq_l91_x3}
\begin{aligned}
 \hspace{0.1 cm}\mathbb{P}\bc{1\in\mathcal{F}\bc{\bm{T}_{n,t/n}[\bm{\theta}]}\Delta \mathcal{F}\bc{\bm{T}_{n+1,t/n}[\bm{\theta}]}} +\mathbb{P}\bc{1\in\mathcal{F}\bc{\bm{T}_{n,t/n}[\bm{\theta}]\abc{1;}}\Delta \mathcal{F}\bc{\bm{T}_{n+1,t/n}[\bm{\theta}]\abc{1;}}}+o_n(1)
\end{aligned}
\end{equation}
for $ \mathbb{E}|\vx_{n,t/n}-\vx_{n+1,t/n}|$. \Cref{l7,l71} now imply that
\begin{align}\label{eq_xn0}
\mathbb{E}|\vx_{n,t/n}-\vx_{n+1,t/n}|\leq  4(L+2)!\delta+(L+3)\mathbb{P}\bc{\poi{d}\geq L}+o_{n,P}(1),\ \mbox{uniformly in $t\in[0,d]$.}
\end{align}
In particular,
\begin{align}\label{eq_xn_lim}
\limsup_{P\to\infty}\limsup_{n\to\infty}\sup_{N\geq n,J_N\in\syN}\sup_{t\in[0,d]} \mathbb{E}|\vx_{n,t/n}-\vx_{n+1,t/n}|\leq  4(L+2)!\delta+3L\mathbb{P}\bc{\poi{d}\geq L}.
\end{align}
Since the {left} hand side of \cref{eq_xn_lim} does not depend on $L$ and $\delta$, we {can send $\delta\downarrow 0$ followed by $L\to \infty$ to} conclude that 
\begin{align}\label{eq_xn01}
\limsup_{P\to\infty}\limsup_{n\to\infty}\sup_{N\geq n,J_N\in\syN}\sup_{t\in[0,d]} \mathbb{E}|\vx_{n,t/n}-\vx_{n+1,t/n}|= 0
\end{align}
or equivalently, $\mathbb{E}|\vx_{n,t/n}-\vx_{n+1,t/n}|=o_{n,P}(1)$ uniformly in $t\in[0,d]$.

\textbf{Case 2}: ${\bm w}_{n,t/n} = \vy_{n,t/n}$.\\
\Cref{dsc} of completely frozen variables and \Cref{l1.1b} (ii) on row addition yield the identity
\begin{align}\label{eq_l91_y2}
    \ind\cbc{1 \in \mathcal Y\bc{\bm{T}_{n,t/n}[\bm{\theta}]}} =  \ind\cbc{1 \in \mathcal F\bc{\bm{T}_{n,t/n}[\bm{\theta}]\abc{1;}} \cap \mathcal F\bc{\bm{T}_{n,t/n}[\bm{\theta}]^T\abc{1;}} }.
\end{align}
Using $(\mathfrak B_1\cap \mathfrak B_2)\Delta(\mathfrak B_3\cap \mathfrak B_4)\subseteq (\mathfrak B_1\Delta \mathfrak B_3)\cup (\mathfrak B_2\Delta \mathfrak B_4)$ for any sets $\mathfrak B_1, \mathfrak B_2, \mathfrak B_3, \mathfrak B_4$ and plugging (\ref{eq_l91_y2}) into (\ref{eq_l91_x1}) yields the upper bound
\begin{equation}\label{eq_l91_y3}
\begin{aligned}
\mathbb{P}\bc{1\in\mathcal{F}\bc{\bm{T}_{n,t/n}[\bm{\theta}]\abc{1;}}\Delta \mathcal{F}\bc{\bm{T}_{n+1,t/n}[\bm{\theta}]\abc{1;}}} +\mathbb{P}\bc{1\in\mathcal{F}(\bm{T}_{n,t/n}[\bm{\theta}]^T\abc{1;})\Delta \mathcal{F}(\bm{T}_{n+1,t/n}[\bm{\theta}]^T\abc{1;})}+o_n(1)
\end{aligned}
\end{equation}
for $\mathbb{E}|\vy_{n,t/n}-\vy_{n+1,t/n}|$. 
\Cref{l71} then yields
\begin{align*}
\mathbb{E}|\vy_{n,t/n}-\vy_{n+1,t/n}|\leq  4(L+2)!\delta+2(L+2)\mathbb{P}\bc{\poi{d}\geq L}+o_{n,P}(1), \ \mbox{uniformly in $t \in [0,d]$}.
\end{align*}
Now the same limiting argument as in Case 1 yields
\begin{align*}
\mathbb{E}|\vy_{n,t/n}-\vy_{n+1,t/n}|=o_{n,P}(1),\ \mbox{uniformly in $t\in[0,d]$}.
\end{align*}

\textbf{Case 3}: ${\bm w}_{n,t/n} = \vz_{n,t/n}$.\\
\Cref{d2} of frozen variables yields the identity
\begin{align}\label{eq_l91_z2}
    \ind\cbc{1 \in \mathcal Z\bc{\bm{T}_{n,t/n}[\bm{\theta}]}} =  \ind\cbc{1 \notin \mathcal F\bc{\bm{T}_{n,t/n}[\bm{\theta}]} \cup \mathcal F\bc{\bm{T}_{n,t/n}[\bm{\theta}]^T} }.
\end{align}
Using $(\mathfrak B_1\cup \mathfrak B_2)\Delta(\mathfrak B_3\cup \mathfrak B_4)\subseteq (\mathfrak B_1\Delta \mathfrak B_3)\cup (\mathfrak B_2\Delta \mathfrak B_4)$ for any sets $\mathfrak B_1, \mathfrak B_2, \mathfrak B_3, \mathfrak B_4$ and plugging (\ref{eq_l91_z2}) into (\ref{eq_l91_x1}) yields the upper bound 
\begin{equation}\label{eq_l91_z3}
\begin{aligned}
\mathbb{P}\bc{1\in\mathcal{F}\bc{\bm{T}_{n,t/n}[\bm{\theta}]}\Delta \mathcal{F}\bc{\bm{T}_{n+1,t/n}[\bm{\theta}]}} +\mathbb{P}\bc{1\in\mathcal{F}\bc{\bm{T}_{n,t/n}[\bm{\theta}]^T}\Delta \mathcal{F}\bc{\bm{T}_{n+1,t/n}[\bm{\theta}]^T}}+o_n(1)
\end{aligned}
\end{equation}
for $\mathbb{E}|\vz_{n,t/n}-\vz_{n+1,t/n}|$.
\Cref{l7} then yields
\begin{align*}
\mathbb{E}|\vz_{n,t/n}-\vz_{n+1,t/n}|\leq  4(L+1)!\delta+2\mathbb{P}\left({\rm Po}(d)\geq L\right)+o_{n,P}(1), \ \mbox{uniformly in $t \in [0,d]$}.
\end{align*}
Now the same limiting argument as in Case 1 yields
\begin{align*}
\mathbb{E}|\vz_{n,t/n}-\vz_{n+1,t/n}|=o_{n,P}(1),\ \mbox{uniformly in $t\in[0,d]$}.
\end{align*}

\textbf{Case 4}: ${\bm w}_{n,t/n} = \vu_{n,t/n}$.\\
\Cref{d2} of frozen variables and \Cref{l1.1b} (ii) on row addition yield the identity
\begin{align}\label{eq_l91_u2}
    \ind\cbc{1 \in \mathcal U\bc{\bm{T}_{n,t/n}[\bm{\theta}]}} =  \ind\cbc{1 \in \mathcal F\bc{\bm{T}_{n,t/n}[\bm{\theta}]^T\abc{1;}} \setminus \mathcal F\bc{\bm{T}_{n,t/n}[\bm{\theta}]} }.
\end{align}
Using $(\mathfrak B_1\setminus \mathfrak B_2)\Delta(\mathfrak B_3\setminus \mathfrak B_4)\subseteq (\mathfrak B_1\Delta \mathfrak B_3)\cup (\mathfrak B_2\Delta \mathfrak B_4)$ for any sets $\mathfrak B_1, \mathfrak B_2, \mathfrak B_3, \mathfrak B_4$ and plugging (\ref{eq_l91_u2}) into (\ref{eq_l91_x1}) yields the upper bound 
\begin{equation}\label{eq_l91_u3}
\begin{aligned}
\mathbb{P}\bc{1\in\mathcal{F}\bc{\bm{T}_{n,t/n}[\bm{\theta}]}\Delta \mathcal{F}\bc{\bm{T}_{n+1,t/n}[\bm{\theta}]}} +\mathbb{P}\bc{1\in\mathcal{F}(\bm{T}_{n,t/n}[\bm{\theta}]^T\abc{1;})\Delta \mathcal{F}(\bm{T}_{n+1,t/n}[\bm{\theta}]^T\abc{1;})}+o_n(1).
\end{aligned}
\end{equation}
for $\mathbb{E}\abs{\vu_{n,t/n}-\vu_{n+1,t/n}}$. \Cref{l7,l71} then give
\begin{align*}
\mathbb{E}|\vu_{n,t/n}-\vu_{n+1,t/n}|\leq  
4(L+2)!\delta+(L+3)\mathbb{P}\bc{\poi{d}\geq L}+o_{n,P}(1), \ \mbox{uniformly in $t \in [0,d]$}.
\end{align*}
Now the same limiting argument as in Case 1 yields
\begin{align*}
\mathbb{E}|\vu_{n,t/n}-\vu_{n+1,t/n}|=o_{n,P,}(1),\ \mbox{uniformly in $t\in[0,d]$}.
\end{align*}

\textbf{Case 5}: ${\bm w}_{n,t/n} = \vv_{n,t/n}$.\\
This is completely analogous to Case 4.

\textbf{Case 6}: $\vw_{n,t/n}=\vze_{n,t/n}$.\\
This is an immediate consequence of Cases 1-5.
\end{proof}

\section{Type fixed point equations}\label{sec_pc}
As laid out in detail in \Cref{sec_rifpm}, for our lower bound on the rank to be tight, we need further means to restrict the potential values of the proportion $\bm\alpha_{n,t/n}$ of frozen variables. In this section, with the help of the stability properties of the types that we derived in \Cref{sec_prof}, we derive asymptotic fixed point equations for the proportions of finer types $\vy_{n,t/n}, \vu_{n,t/n}$ and $\vv_{t,t/n}$ in $\bm{T}_{n,t/n}[\bm{\theta}]$, as well as a lower bound for $\vz_{n,t/n}$. Correspondingly, the single main result of this section is \Cref{al11} below.

The proof of the characterisations in \Cref{al11} is based on a detailed analysis of the connection of the type of variable $n+1$ in the larger matrix $\bm{T}_{n+1,t/n}[\bm{\theta}]$ to the types of the non-zero entries of row $n+1$ in the smaller matrix $\bm{T}_{n,t/n}[\bm{\theta}]$. In this way, we can relate the proportions of types to certain functions of other proportions, that simply correspond to the choices of the non-zero entries of row $n+1$ and therefore comparatively easy to evaluate, see \Cref{sec_prosc12}. 
Of course, the details are considerably more involved, but indeed, this proof scheme is quite similar to the
deduction of the heuristic
fixed point equation in \Cref{sec_rifpm}, where we relate the type of the new coordinate to the types of its neighbours by the combination of \cref{adf,eq00}.

\subsection{Section overview}\label{sec_fe}
The main goal of this section is to derive the following fixed point equations for the types from \Cref{def_proportions}:
\begin{proposition}[Type fixed point equations]\label{al11}
For any $n\geq 0$ and $d>0$, 
\begin{align}
   & \vy_{n,t/n}=1-\phi_t\bc{\vx_{n,t/n}+\vy_{n,t/n}+\vu_{n,t/n}}-\phi_t\bc{\vx_{n,t/n}+\vy_{n,t/n}+\vv_{n,t/n}}+\phi_t\bc{\vx_{n,t/n}+\vy_{n,t/n}}+\oone;\label{eq_fpy0}\\
  &  \vu_{n,t/n}=\phi_t\bc{\vx_{n,t/n}+\vy_{n,t/n}+\vu_{n,t/n}}-\phi_t\bc{\vx_{n,t/n}+\vy_{n,t/n}}+\oone;\label{eq_fpu0}\\
 &   \vv_{n,t/n}=\phi_t\bc{\vx_{n,t/n}+\vy_{n,t/n}+\vv_{n,t/n}}-\phi_t\bc{\vx_{n,t/n}+\vy_{n,t/n}}+\oone;\label{eq_fpv0}\\
  &  \vz_{n,t/n}\geq \phi_t\bc{\vy_{n,t/n}}+\oone. \label{eq_fpz0}
\end{align}
\end{proposition}

For the proof of \Cref{al11}, in whose course we also work with a more general matrix model, we give names to the functions on the right hand sides of (\ref{eq_fpy0}) to (\ref{eq_fpv0}). We use the following suggestive notation:

\begin{definition}[Type functions]\label{defun}
Let $\mathcal G$ denote the set of non-decreasing functions $g:[0,1]\to [0,1]$ and $\Delta^4$ be the four-dimensional standard simplex. We then define the following three functions $Y,U,V:\Delta^4\times \mathcal{G} \to [0,1]$ by setting:
\begin{enumerate}[label=(\roman*)]
  \item $Y\bc{\zeta,g}=1-g(x+y+u)-g(x+y+v)+g(x+y)$ for $(\zeta,g) \in \Delta^4\times \mathcal{G}$;
  \item $U\bc{\zeta,g}=g(x+y+u)-g(x+y)$ for $(\zeta,g) \in \Delta^4\times \mathcal{G}$;
  \item $V\bc{\zeta,g}=g(x+y+v)-g(x+y)$ for $(\zeta,g) \in \Delta^4\times \mathcal{G}$.
\end{enumerate}
\end{definition}


The proof of \Cref{al11} is split into two main parts: \Cref{apexp} and \Cref{cs1}. First, \Cref{apexp} reduces the approximation of the types through the type functions to the separate approximation of conditional type probabilities in a larger matrix through the type functions and approximation of $\vze_{n+1,t/n}$ through $\vze_{n,t/n}$:

\begin{lemma}\label{apexp}
Let $t\in[0,d]$ with $d>0$. For any $W\in\cbc{Y,U,V}$ and $K\in \ZZ_{\geq 2}$,
\begin{align}\label{eq_eqyuv}
\mathbb{E}\abs{\vw_{n,t/n}-W\bc{\vze_{n,t/n},\phi_t}}\leq& \mathbb{E}\abs{\mathbb{P}\bc{n+1\in\mathcal{W}\bc{\bm{T}_{n+1,t/n}[\bm{\theta}]}\big|\vze_{n,t/n}}-W\bc{\vze_{n,t/n},\phi_t}} 
\\
& + (5+8K^2)\mathbb{E}\|\vze_{n,t/n}-\vze_{n+1,t/n}\|_\infty+4-4(1-1/K)^5+10\bc{1+2d}/K,\nonumber
\end{align}
and
\begin{align}\label{eq_eqz}
\mathbb{E}\brk{\bc{\vz_{n,t/n}-\phi_t\bc{\vy_{n,t/n}}}^-}\leq&\mathbb{E}\brk{\bc{\mathbb{P}\bc{n+1\in\mathcal{Z}\bc{\bm{T}_{n+1,t/n}[\bm{\theta}]} \big|\vze_{n,t/n} }-\phi_t\bc{\vy_{n,t/n}}}^-} 
\\
&+ (5+8K^2)\mathbb{E}\|\vze_{n,t/n}-\vze_{n+1,t/n}\|_\infty+4-4(1-1/K)^5+10\bc{1+2d}/K.\nonumber
\end{align}
\end{lemma}
The extensive proof of \Cref{apexp} is deferred to Appendix \ref{app_b}, since it is separate from our main proof strategy and rather technical. 
Since thanks to \Cref{l91}, we have good control over the differences $\mathbb{E}\|\vze_{n,t/n}-\vze_{n+1,t/n}\|_\infty$ already, it only remains to take care of the conditional probabilities in \Cref{apexp}. This is exactly what the second \Cref{cs1} does by illustrating the probabilistic interpretation of the type functions: It shows that the type functions (resp. $\phi_t$) are approximations (resp. lower bounds) of the probabilities that column $n+1$ in $\bm{T}_{n+1,t/n}[\bm{\theta}]$ has the type associated with the function (resp. type $\mathcal Z$), conditionally on $\vze_{n,t/n}$: 

\begin{lemma}[Conditional probabilities and type functions]\label{cs1}
For  any $W \in \cbc{Y,U,V}$ and $d >0$,
\begin{equation}\label{eq_fixyuv}
    \mathbb{P}\bc{n+1\in\mathcal{W}\bc{\bm{T}_{n+1,t/n}[\bm{\theta}]}\big|\vze_{n,t/n}}-W\bc{\vze_{n,t/n},\phi_t}=\oone,
\end{equation}
while
\begin{equation}\label{eq_ineqz}
\mathbb{P}\bc{n+1\in\mathcal{Z}\bc{\bm{T}_{n+1,t/n}[\bm{\theta}]}\big|\vze_{n,t/n}}-\phi_t\bc{\vy_{n,t/n}} \geq \oone.
\end{equation}
\end{lemma}

The proof of \Cref{cs1} is presented in  \Cref{sec_prosc1}. 
 With \Cref{cs1,apexp} in hand, we are ready to prove \Cref{al11}:

\begin{proof}[Proof of \Cref{al11} subject to \Cref{cs1,apexp}]
Let $W \in \{Y,U,V\}$ and $K \in \NN_{\geq 2}$. By \Cref{apexp}, an application of \Cref{l91} and \cref{eq_fixyuv} to the right hand side of \cref{eq_eqyuv} together with the fact $\| \cdot\|_{\infty} \leq \| \cdot\|_{1}$ gives
\begin{equation}\label{exh}\begin{aligned}
\mathbb{E}\abs{\vw_{n,t/n}-W\bc{\vze_{n,t/n},\phi_t}}\leq o_{n,P}(1) + 4 - 4(1-1/K)^{5} + 10(1+2d)/K \qquad \text{uniformly in $t \in [0,d]$}.
\end{aligned}\end{equation}
Since \cref{exh} holds true for any $K \geq 2$, and the left hand side does not depend on $K$, \cref{exh} implies that
\[\limsup_{P\to\infty}\limsup_{n\to\infty}\sup_{N\geq n,J_N\in\syN}\sup_{0\leq t\leq d}\mathbb{E}\abs{\vw_{n,t/n}-W\bc{\vze_{n,t/n},\phi_t}}=0,\]
which gives \cref{eq_fpy0} to \cref{eq_fpv0}.

Analogously, by \Cref{apexp}, an application of \Cref{l91} and \cref{eq_eqz} to the right hand side of \cref{eq_ineqz} together with the fact $\| \cdot\|_{\infty} \leq \| \cdot\|_{1}$ gives
\begin{equation}\label{exh2}\begin{aligned}
\mathbb{E}\brk{\bc{ \vz_{n,t/n}-\phi_t\bc{\vy_{n,t/n}} }^-}\leq o_{n,P}(1) + 4 - 4(1-1/K)^{5} + 10(1+2d)/K \qquad \text{uniformly in $t \in [0,d]$}.
\end{aligned}\end{equation}
As before , \cref{exh2} implies that
$\bc{\vz_{n,t/n}-\phi_t\bc{\vy_{n,t/n}}}^-=\oone$, i.e., $\vz_{n,t/n}-\phi_t\bc{\vy_{n,t/n}}\geq \oone.$
\end{proof}
It thus only remains to prove \Cref{cs1}. 

\subsection{Conditional probabilities and type functions: Proof of \Cref{cs1}}\label{sec_prosc1}
As outlined in \Cref{sec_fe}, the only ingredient in the proof of \Cref{al11} that is still lacking a proof is \Cref{cs1}. This gap is closed in the current section. In addition, we prove a more general version of \Cref{cs1}, which holds true for arbitrary square matrices and more general types of symmetric row- and column-attachment. We believe that this extension will prove useful in future applications of our strategy.


More precisely, for any positive integer $n$, let $A\in\FF^{n\times n}$ be an arbitrary square matrix and let $\bfh$ be an integer-valued random variable with probability generating function $\psi$. Given $\bfh$, let $\bm{h} \in \FF^{1\times n}$ be a random vector whose non-zero entries are chosen uniformly at random from the set $\binom{[n]}{\bfh}$. Throughout this section, we write
\[A^{\bm{h}}=\begin{pmatrix}A&\bm{h}^T\\\bm{h}&0\end{pmatrix}\]
to denote the matrix $A$ after symmetric row- and column-attachment of $\bm h$. 
Finally, we omit the explicit dependence of the proportions on the underlying matrix and simply write $\vx$ for $x(A[\bm{\theta}])$ throughout \Cref{sec_prosc1}. The quantities $\vy,\vz,\vu,\vv$ and $\vze$ are defined analogously. 

The main result of this section is the following generalised version of \Cref{cs1}:
\begin{proposition}[Approximating the type probabilities for column {$n+1$}]\label{ccs1}
For any $A\in\FF^{n\times n}$, $L\in \NN_{\geq 2}, \delta>0$ and $W \in\cbc{Y,U,V}$,
\begin{equation}\label{eq_ccs1_or}
    \mathbb{E}\abs{\mathbb{P}\bc{n+1\in\mathcal{W}\bc{A^{\bm{h}}[\bm{\theta}]}\big|\vze}-W(\vze,\psi)}\leq 6L!\delta+7\mathbb{P}\left(\bfh\geq L\right)+o_{n,P}(1),\
\end{equation}
and
\begin{equation}\label{eq_ccs1_z}
    \mathbb{E}\brk{\bc{\mathbb{P}\bc{n+1\in\mathcal{Z}\bc{A^{\bm{h}}[\bm{\theta}]}\big|\vze}-\psi\bc{\vy}}^-}\leq 2\mathbb{P}\left(\bfh\geq L\right)+o_{n,P}(1).\
\end{equation}
\end{proposition}

\begin{remark}[Error terms]
We emphasize that the error terms in \Cref{ccs1} and the rest of this section are uniform in $A$ and $\psi$. In the current more general setting, it becomes evident that the distribution of the type of the new column $n+1$ only depends on $A$ through the proportions of the types in $A[\bm{\theta}]$.  
\end{remark}
\Cref{cs1} is now a direct consequence of \Cref{ccs1}:
\begin{proof}[Proof of \Cref{cs1} subject to \Cref{ccs1}]
In the set-up of \Cref{ccs1}, let $A=\bm{T}_{n,t/n}$ and $\bm{h}=\bm{T}_{n+1,t/n}(n+1,)\abc{;n+1}$, such that $\psi$ becomes the probability generating function of a Bin$(n,t/n)$-variable. We first look at \Cref{eq_fixyuv}. For $W \in \{Y,U,V\}$, by the triangle inequality,
\begin{equation}\label{eq_triw}
    \begin{aligned}
    &\Erw\abs{\mathbb{P}\bc{n+1\in\mathcal{W}\bc{\bm{T}_{n+1,t/n}[\bm{\theta}]}\big|\vze_{n,t/n}}-W\bc{\vze_{n,t/n},\phi_t}}\\
    \leq \hspace{0.1 cm} &\Erw\abs{\mathbb{P}\bc{n+1\in\mathcal{W}\bc{\bm{T}_{n+1,t/n}[\bm{\theta}]}\big|\vze_{n,t/n}}-W\bc{\vze_{n,t/n},\psi}}+\Erw\abs{W\bc{\vze_{n,t/n},\psi}-W\bc{\vze_{n,t/n},\phi_t}}.
    \end{aligned}
\end{equation}
By \cite[Theorem 2.10]{van2017random}, 
\begin{equation}\label{eq_binpoi}
    \begin{aligned}
\sup_{r\in[0,1]}\abs{\psi(r)-\phi_t(r)}\leq\sum_{k=0}^{\infty}\abs{\mathbb{P}\bc{\text{Bin}(n,t/n)=k}-\mathbb{P}\bc{\poi{t}=k}}=\dtv{\bin{n,t/n},\poi{t}}\leq d^2/n.
\end{aligned}
\end{equation}
Equation \Cref{eq_fixyuv} then follows from \Cref{ccs1}, \cref{eq_triw} and \cref{eq_binpoi}. Inequality \Cref{eq_ineqz} follows anagolously.
\end{proof}
In the remainder of \Cref{sec_prosc1}, we carry out the proof of \Cref{ccs1}. For this, we first introduce five \textit{type events} in \Cref{sec_ctsoa} that establish a connection between the type of $n + 1$ and the types of $\supp{\bm h}$ in the underlying matrix $A^{\bm h}[\bm \theta]$. Since the non-zero coordinates of $\bm h$ are chosen uniformly given $\bfh$, we can then estimate the probabilities of the type events in \Cref{sec_prosc12}, and complete the proof of \Cref{ccs1} in \Cref{sec_proccs1}.

\subsubsection{Type events}\label{sec_ctsoa}
As announced, in this subsection, we introduce a number of ``type'' events that are solely defined in terms of $\supp{\bm h}$ and $A[\bm \theta]$. These events capture the main causes for variable $n+1$ to belong to a particular set $\mathcal{W}\bc{A^{\bm{h}}[\bm{\theta}]}$ in terms of whether all variables in $\supp{\bm{h}}$ are frozen with respect to $A[\bm{\theta}]$ or $A[\bm{\theta}]^T$. In this sense, we show that the probability that $n + 1$ does \textit{not} have a certain type on its matching type event is small in Lemmas \ref{l11} and \ref{l12} below.


As in Section \ref{sec_changefrozen}, throughout this section, we will frequently work on two good events. The first event $\mathfrak{P}_{n}$ will be the same as in \cref{event_P}, since it only involves the perturbation and ensures that the perturbations in $A[\bm{\theta}]$ and $A^{\bm{h}}[\bm{\theta}]$ agree. By \cref{oe}, $\mathbb{P}\left({\mathfrak{P}_n}\right)= 1+o_n(1)$. Secondly, and analogously to the definition of \cref{event_R}, we define an event ${\mathfrak R }$ that is used to  make our target matrix $(\delta,\ell)$-free, so that we can apply \Cref{l5}. More precisely, we denote by $\mathfrak R=\mathfrak R(\delta,L)$ the good  event that both $A[\bm{\theta}]$ and $A[\bm{\theta}]^T$ are $(\delta,\ell)$-free for $2\leq \ell\leq L$. \Cref{p1} gives that
\begin{equation} \label{def_R}
    \mathbb{P}\bc{\mathfrak R }\geq 1+o_{n,P}(1).
\end{equation}
Before we introduce the actual type events in \Cref{defq2}, we first define two basic events that are used to decide whether variable $n+1$ is firmly frozen in $A^{\bm{h}}[\bm{\theta}]$.

\begin{definition}[Basic events]\label{defq}
 Given $A\in\FF^{n\times n}$ and $\bm h$ as above, 
 we define the following events:
\begin{align}
\ffB &= \{\supp{\bm{h}}\subseteq \mathcal{F}\bc{A[\bm{\theta}]} \},\label{def_ffB} \\ 
\ffBT &= \{ \supp{\bm{h}}\subseteq \mathcal{F}\bc{A[\bm{\theta}]^T}\} \label{def_ffBT}.
\end{align}
\end{definition}

The following preparatory lemma then shows that the main reason for the new variable $n+1$ to be \textit{firmly frozen} in $A^{\bm{h}}[\bm{\theta}]$ is that not all of the variables in $\supp{\bm{h}}$ are frozen in $A[\bm{\theta}]^T$, and thus the event $\fBT$. This observation is later used to characterise the other possible types of $n+1$ in terms of the support of $\bm{h}$.

\begin{lemma}\label{l10}
For any $\delta>0, L\in \NN_{\geq 2}$ and $A \in \FF^{n \times n}$,
\begin{align} \label{eq_lem_l10_3}
\mathbb{P}\left(\text{$n+1$ is firmly frozen in $A^{\bm{h}}[\bm{\theta}]$}, \ffBT \right) = o_n(1),
\end{align}
and
\begin{align} \label{eq_lem_l10_1} 
 \mathbb{P}\left(\text{$n+1$ is not firmly frozen in $A^{\bm{h}}[\bm{\theta}]$}, \fBT\right) \leq L!\delta+\mathbb{P}\left(\bfh\geq L\right)+o_{n,P}(1).
\end{align}
\end{lemma}
%

\begin{proof}

\begin{enumerate}[label=(\roman*)]
\item We first show (\ref{eq_lem_l10_3}). 
By definition, $n+1$ is firmly frozen in $A^{\bm{h}}[\bm{\theta}]$ if and only if it is frozen in $A^{\bm{h}}[\bm{\theta}]\abc{n+1;}$. On the good event $\mathfrak{P}_n$, removal of row $n+1$ leaves us with the matrix $A[\bm \theta]$ plus the additional column $\bc{\bm{h}\quad 0_{1\times \bm{\theta}_r}}^T$. By \Cref{p4}, 
\begin{align*}
    \text{$n+1$ is firmly frozen in $A^{\bm{h}}[\bm{\theta}]$},\mathfrak{P}_n \  \Longrightarrow&\  \text{$\bc{\bm{h}\quad 0_{1\times \bm{\theta}_r}}^T$ cannot be linearly combined by the columns of $A[\bm{\theta}]$.}\\
\  \Longrightarrow&\ \text{$\supp{\bm{h}}\not\subseteq \mathcal{F}\bc{A[\bm{\theta}]^T}$.}
\end{align*}
On the other hand, on $\ffBT$, $\supp{\bm{h}}\subseteq\mathcal{F}\bc{A[\bm{\theta}]^T}$. 
Therefore, 
\[ \mathbb{P}\left(\text{$n+1$ is firmly frozen in $A^{\bm{h}}[\bm{\theta}]$}, \ffBT \right)= \mathbb{P}\left(\text{$n+1$ is firmly frozen in $A^{\bm{h}}[\bm{\theta}]$}, \ffBT, \mathfrak P_n \right) + o_n(1) = o_n(1),\]
as required.

 \item We next prove (\ref{eq_lem_l10_1}). 
    By definition, if $n+1$ is not firmly frozen in $A^{\bm{h}}[\bm{\theta}]$, it is not frozen in $A^{\bm{h}}[\bm{\theta}]\abc{n+1;}$. 
    On the good event $\mathfrak{P}_n$, 
    removal of row $n+1$ leaves us with the matrix $A[\bm \theta]$ plus the additional column $\bc{\bm{h}\quad 0_{1\times \bm{\theta}_r}}^T$. Since $n+1$ is not frozen in this matrix, \Cref{p4} gives that
\begin{align*}
    \text{$n+1$ is not firmly frozen in $A^{\bm{h}}[\bm{\theta}]$},\mathfrak{P}_n  \ \Longrightarrow\
\text{ $\bc{\bm{h}\quad 0_{1\times \bm{\theta}_r}}^T$ can be lin. combined by the columns of $A[\bm{\theta}]$.}
\end{align*}

On the other hand, on $\fBT$,  $\supp{\bm{h}}\not\subseteq\mathcal{F}\bc{A[\bm{\theta}]^T}$. This implies that both $\supp{\bm{h}}$ and $\supp{\bm{h}}\backslash \mathcal{F}\bc{A[\bm{\theta}]^T}$ are non-empty. If additionally, $\bc{\bm{h}\quad 0_{1\times \bm{\theta}_r}}^T$ can be linearly combined by the columns of $A[\bm{\theta}]$, $\supp{\bm{h}}\backslash \mathcal{F}\bc{A[\bm{\theta}]^T}$ is a relation of $A[\bm{\theta}]^T$. Hence, by \Cref{d2} (iii),
\begin{align*}
    \text{$n+1$ is not firmly frozen in $A^{\bm{h}}[\bm{\theta}]$}, \fBT, \mathfrak P_n \quad \Longrightarrow\quad
\text{$\supp{\bm{h}}$ is a proper relation of $A[\bm{\theta}]^T$.}
\end{align*}
By \Cref{l5} and (\ref{def_R}),
\[\begin{aligned}
\mathbb{P}\left(\text{$n+1$ is not firmly frozen in $A^{\bm{h}}[\bm{\theta}]$}, \fBT, \mathfrak P_n\right)
\leq& \hspace{0.1 cm }\mathbb{P}\left(\supp{\bm{h}} \text{ is a proper relation in } A[\bm{\theta}]^T\right)\\
\leq &\hspace{0.1 cm } L!\delta+\mathbb{P}\left(\bfh\geq L\right)+o_{n,P}(1),
\end{aligned}\]
as required.
\end{enumerate}
\end{proof}

With \Cref{l10}, we are now in the position to characterise the type of variable $n+1$ in terms of the role of the variables in $\supp{\bm h}$ in $A[\bm{\theta}]$ and $A[\bm{\theta}]^T$  through the following events.
\begin{definition}[Type events]\label{defq2}
With the notation of \Cref{defq}, let
\begin{align*}
{\mathfrak Y} &= \fB \cap \fBT,  \\
{\mathfrak U} &= \fB \cap \ffBT,  \\
{\mathfrak V} &= \ffB \cap \fBT, \\
{\mathfrak {XZ}} &= \ffB \cap \ffBT
\quad \text{and }\\
{\mathfrak Z_{\circ}} &= \{ \supp{\bm{h}}\subseteq \mathcal{Y}\bc{A[\bm{\theta}]}\}.
\end{align*}
\end{definition}

\begin{remark}\label{reuplus}
By construction, the four events $\mathfrak Y, \mathfrak U, \mathfrak V, \mathfrak{XZ}$ are pairwise disjoint, and their union $\mathfrak Y \uplus \mathfrak U \uplus \mathfrak V \uplus \mathfrak{XZ}$ gives the whole sample space.
\end{remark}

In the following two Lemmas \ref{l11} and \ref{l12}, we first show that for each choice of $\cW \in \{\cY, \cU, \cV, \cX\cZ\}$, the probability that $n+1$ does \textit{not} have the type corresponding to $\cW$ on ${\mathfrak W}$ is small and then that the probability that $n+1 \notin \cZ(A^{\bm h}[\bm \theta])$ on ${\mathfrak Z}_{\circ}$ is small. This offers an \textit{almost} complete description of the type of $n+1$ in terms of the events in \Cref{defq2}. \Cref{l11} deals with the simpler cases $\mathcal{W}\in\cbc{\mathcal{Y},\mathcal{U},\mathcal{V},\mathcal{X}\mathcal{Z}}$, where the type events are intersections of basic events.

\begin{lemma}\label{l11}
For any $\delta>0, L\in \NN_{\geq 2}$, $A \in \FF^{n \times n}$ and $W \in \{Y,U,V\}$,
\begin{align} \label{bound_QY}
\mathbb{P}\left(n+1\notin\mathcal{W}\bc{A^{\bm{h}}[\bm{\theta}]}, {\mathfrak W}\right) \leq 2 L!\delta+2\mathbb{P}\left(\bfh\geq L\right)+o_{n,P}(1),
\end{align}
as well as
\begin{align} \label{bound_QXZ}
\mathbb{P}\left(n+1\not\in\mathcal{X}\bc{A^{\bm{h}}[\bm{\theta}]}\cup \mathcal{Z}\bc{A^{\bm{h}}[\bm{\theta}]}, {\mathfrak{XZ}}\right) = o_n(1).
\end{align}
\end{lemma}

\begin{proof}
We show the claim for each of the possible variable types separately.

    \textit{Completely frozen variables - (\ref{bound_QY}) for $W=Y$}: 
    By definition, if $n+1$ is not completely frozen in $A^{\bm{h}}[\bm{\theta}]$, then it is not firmly frozen in $A^{\bm{h}}[\bm{\theta}]$ or not firmly frozen in $A^{\bm{h}}[\bm{\theta}]^T$. Since \Cref{l10} also applies to $A^{\bm h}[\bm \theta]^T$, a union bound gives
\[\begin{aligned}
\mathbb{P}\left(n+1\not\in\mathcal{Y}\bc{A^{\bm{h}}[\bm{\theta}]}, {\mathfrak Y}\right) & \leq \mathbb{P}\left(\text{$n+1$ not firmly frozen in $A^{\bm{h}}[\bm{\theta}]$}, \fBT \right)+\mathbb{P}\left(\text{$n+1$ not firmly frozen in $A^{\bm{h}}[\bm{\theta}]^T$}, \fB \right)\\
& \leq  2L!\delta+2\mathbb{P}\left(\bfh\geq L\right)+o_{n,P}(1).
\end{aligned}\]

\textit{One-sided firmly frozen variables  - (\ref{bound_QY}) for $W \in \{U,V\}$}:
    If $n+1\not\in\mathcal{U}\bc{A^{\bm{h}}[\bm{\theta}]}$, then, by definition, either $n+1$ is not firmly frozen in $A^{\bm{h}}[\bm{\theta}]^T$, or, if this is not the case, it is  frozen in $A^{\bm{h}}[\bm{\theta}]$ \textit{and} firmly frozen in $A^{\bm{h}}[\bm{\theta}]^T$. In the latter case, the symmetry of frailly frozen variables under transposition (see \Cref{c81}) implies that $n+1$ is also firmly frozen in $A^{\bm{h}}[\bm{\theta}]$. We conclude that if $n+1\not\in\mathcal{U}\bc{A^{\bm{h}}[\bm{\theta}]}$, then either $n+1$ is not firmly frozen in $A^{\bm{h}}[\bm{\theta}]^T$ or $n+1$ is firmly frozen in $A^{\bm{h}}[\bm{\theta}]$. Again, by a union bound and \Cref{l10},
\[\begin{aligned}
\mathbb{P}\left(n+1\not\in\mathcal{U}\bc{A^{\bm{h}}[\bm{\theta}]}, {\mathfrak U}\right)
&\leq \mathbb{P}\left(\text{$n+1$ not firmly frozen in $A^{\bm{h}}[\bm{\theta}]^T$}, \fB \right)+\mathbb{P}\left(\text{$n+1$ firmly frozen in $A^{\bm{h}}[\bm{\theta}]$}, \ffBT \right)\\
&\leq L!\delta+\mathbb{P}\left(\bfh\geq L\right)+o_{n,P}(1).
\end{aligned}\]

The claim for $W = V$ follows analogously.

\textit{Frailly frozen or two-sided non-frozen variables - (\ref{bound_QXZ})}: 
If $n+1\not\in \mathcal{X}\bc{A^{\bm{h}}[\bm{\theta}]}\cup \mathcal{Z}\bc{A^{\bm{h}}[\bm{\theta}]}$, then by definition, $n+1$ is firmly frozen in $A^{\bm{h}}[\bm{\theta}]$ or $A^{\bm{h}}[\bm{\theta}]^T$. By a union bound and \Cref{l10},
\[\begin{aligned}
&\mathbb{P}\left(n+1\not\in\mathcal{X}\bc{A^{\bm{h}}[\bm{\theta}]}\cup \mathcal{Z}\bc{A^{\bm{h}}[\bm{\theta}]}, {\mathfrak{XZ}}\right) \\
\leq & \hspace{0.1 cm} \mathbb{P}\left(\text{$n+1$ firmly frozen in $A^{\bm{h}}[\bm{\theta}]$}, \ffBT \right)+\mathbb{P}\left(\text{$n+1$ firmly frozen in $A^{\bm{h}}[\bm{\theta}]^T$}, \ffB \right)=o_n(1).
\end{aligned}\]
\end{proof}

We have the following analogous lemma for the event ${\mathfrak Z}_{\circ}$:
\begin{lemma}\label{l12}
For any $L\in \NN_{\geq 2}$ and $A \in \FF^{n \times n}$,
\[ \mathbb{P}\left(n+1\not\in\mathcal{Z}\bc{A^{\bm{h}}[\bm{\theta}]}, {\mathfrak{Z}}_\circ \right) \leq \mathbb{P}\left(\bfh\geq L\right)+ o_{n,P}(1).
\]
\end{lemma}

\begin{proof}
The first (and main) step is to prove that on the intersection of $\mathfrak Z_\circ$ with a sufficiently likely event, 
the $(n+1)$st row in $A^{\bm{h}}[\bm{\theta}]$ can be linearly combined by the other rows of $A^{\bm{h}}[\bm{\theta}]$, from which it follows through \Cref{p4} that $n+1$ is not frozen in $A^{\bm h}[\bm{\theta}]^T$.

On ${\mathfrak Z}_{\circ} \cap \mathfrak P_n$, $A^{\bm h}[\bm \theta](n+1,)= (\bm h \hspace{0.1 cm} 0_{1 \times (\bm \theta_c+1)})$ and all variables in $\supp{\bm{h}} =$ supp$(A^{\bm h}[\bm \theta](n+1,))$ are firmly frozen in $A[\bm{\theta}]$. Ideally, to derive the desired linear combination of $A^{\bm h}[\bm \theta](n+1,)$  by the other rows of $A^{\bm h}[\bm{\theta}]$, we would like to take one representation for each $i \in \supp{\bm{h}}$, and then simply sum over the representations. Alas, the matrix $A^{\bm h}[\bm{\theta}]$ has one more column than $A[\bm\theta]$, and it is not clear that for the existing representations, also the entries of column $n+1$ sum to zero. Therefore, we are looking for representations of $i \in \supp{\bm{h}}$ that expressly do not use one of the rows in $\supp{\bm h}$, if such representations exist. 

In fact, on ${\mathfrak Z}_{\circ}$, since any $i\in \supp{\bm{h}}$ is firmly frozen in $A[\bm{\theta}]$, there exists a representation of $i$ that does not use row $i$. To take care of the other rows corresponding to elements of $\supp{\bm h}$, we define the event 
\[\begin{aligned}
\mathfrak C=&\cbc{\text{for all $i\in \supp{\bm{h}}$, $i \notin \mathcal{F}\bc{A[\bm{\theta}]\abc{\supp{\bm{h}};}}\Delta \mathcal{F}\bc{A[\bm{\theta}]\abc{i;}}$}}.
\end{aligned}\]
The event $\mathfrak C$ is sufficiently likely for our purposes, as 
\[\begin{aligned}
\mathbb{P}\left(\mathfrak C^c\right)
\leq& \mathbb{P}\left(\bfh\geq L\right)+\sum_{i\in[n]}\sum_{k=2}^{L-1} \frac{L}{n}\mathbb{P}\left(i\in \mathcal{F}\bc{A[\bm{\theta}]\abc{\supp{\bm{h}};} }\Delta \mathcal{F}\bc{A[\bm{\theta}]\abc{i;}} \vert i\in \supp{\bm{h}}, \bfh = k\right) \\
&\cdot\mathbb{P}\bc{i\in \supp{\bm{h}}\bfh =k} \\
\leq& \mathbb{P}\left(\bfh\geq L\right)+\frac{L^2}{P}+o_n(1).
\end{aligned}\]
Here, in the last step, we have used \Cref{l9}, which states that for any $i\in [n]$ and $k\leq L$,
\[\mathbb{P}\left(i\in \mathcal{F}\bc{A[\bm{\theta}]\abc{\supp{\bm{h}};} }\Delta \mathcal{F}\bc{A[\bm{\theta}] \abc{i;}}| i\in\supp{\bm{h}}, \bfh =k\right)\leq \frac{L}{P}+o_n(1).\]


By design, on the event
$\mathfrak C \cap {\mathfrak Z}_{\circ}$, any $i \in \supp{\bm h}$ is frozen in $A[\bm{\theta}]\abc{\supp{\bm{h}};}$. In particular, there exists a representation of $\{i\}$ in $A[\bm{\theta}]\abc{\supp{\bm{h}};}$. On the good event $\mathfrak P_n$, each such representation can be extended to a representation $b=\bc{b_1,\ldots,b_{n+\bm{\theta}_r}}$ of $\{i\}$ in $A^{\bm h}[\bm\theta]\abc{n+1;}$ such that
\begin{align}\label{eq_l12_1}
b A^{\bm h}[\bm\theta]\abc{n+1;} =e_{n+\bm{\theta}_c}(i) \qquad\text{and}
\qquad b_k=0 \text{ for $k\in \supp{\bm{h}}$}.   
\end{align}
Thus, on the event ${\mathfrak Z}_{\circ}\cap \mathfrak{C}\cap \mathfrak P_n$, any $i \in \supp{\bm{h}}$ is frozen in $A^{\bm{h}}[\bm{\theta}]\abc{n+1;}$.
We conclude that the $(n+1)$st row in $A^{\bm{h}}[\bm{\theta}]$ can be linearly combined by the other rows of $A^{\bm{h}}[\bm{\theta}]$ (this is also true if $\supp{\bm h} = \emptyset$). Therefore, by \Cref{p4}, $n+1$ is not frozen in $A^{\bm{h}}[\bm{\theta}]^T$, which only leaves the possibility $n+1\in \mathcal{V}\bc{A^{\bm{h}}[\bm{\theta}]}\cup \mathcal{Z}\bc{A^{\bm{h}}[\bm{\theta}]}$ on  ${\mathfrak Z}_{\circ}\cap \mathfrak{C}\cap \mathfrak P_n$. 

On the other hand, since $\mathfrak Z_{\circ} \subseteq \ffBT$, by (\ref{eq_lem_l10_3}), $n+1$ cannot be firmly frozen in $A^{\bm h}[\bm{\theta}]$ on the event ${\mathfrak Z}_{\circ}\cap \mathfrak{C}\cap \mathfrak P_n$.
Therefore, $n+1\in  \mathcal{Z}\bc{A^{\bm{h}}[\bm{\theta}]}$ and we arrive at
\[
    \mathbb{P}\left(n+1\not\in\mathcal{Z}\bc{A^{\bm{h}}[\bm{\theta}]}, {\mathfrak Z}_{\circ}, \mathfrak{C} \right)=o_n(1),
\]
i.e.,
\[\begin{aligned}
\mathbb{P}\left(n+1\not\in\mathcal{Z}\bc{A^{\bm{h}}[\bm{\theta}]}, {\mathfrak Z}_{\circ}\right)
\leq \mathbb{P}\left(\mathfrak C^c\right)+o_n(1)
\leq \mathbb{P}\left(\bfh\geq L\right)+\frac{L^2}{P}+o_n(1) = \mathbb{P}\left(\bfh\geq L\right)+ o_{n,P}(1).
\end{aligned}\]
This yields the claim.
\end{proof}

\subsubsection{Probabilities of type events}\label{sec_prosc12}
In \Cref{sec_ctsoa}, we have related the type of $n+1$ in $A^{\bm h}[\bm \theta]$ to the occurrence of a bunch of type events, which are formulated in terms of $\supp{\bm h}$. We now approximate the conditional probabilities of the type events through the corresponding functions $Y,U,V$ from \Cref{defun} and $\psi$. In this way, we build the connection between the event $\cbc{n+1\in\mathcal{W}\bc{A^{\bm{h}}[\bm{\theta}]}}$ and $W(\vze,\psi)$.  For the current section, recall that we use boldface letters $\vw$ to abbreviate the proportions $w(A[\bm{\theta}])$. We then show that conditionally on the vector $\vze$, for any $W\in\{Y,U,V\}$, the function $W(\vze,\psi)$ is a good approximation of the probability of ${\mathfrak W}$, while $\psi$ is a good approximation of $\mathfrak{Z}_{\circ}$. Since the type events are defined solely in terms of the membership of $\supp{h}$ in the sets $\cW(A[\bm \theta])$ and $\bm h$ is chosen independently of $A[\bm \theta]$, this basically reduces to a comparison between drawing $\supp{\bm h}$ with and without replacement.

\begin{lemma}\label{al1}
For any $L\in \NN_{\geq 2}$, $W \in \{Y, U, V\}$,
\begin{equation}\label{eq_css1_7}
        \abs{\mathbb{P}\bc{{\mathfrak W}|\vze}-W(\vze,\psi)} \leq \mathbb{P}\left(\bfh\geq L\right)+o_{n,P}(1)
\end{equation}
and
\begin{equation}\label{eq_css1_7_new}
        \abs{\mathbb{P}\bc{{\mathfrak Z}_{\circ}|\vze}-\psi\bc{\bm y}} \leq \mathbb{P}\left(\bfh\geq L\right)+o_{n,P}(1).
\end{equation}
\end{lemma}

\begin{proof}
\begin{itemize}
\item[(i)] We first prove \cref{eq_css1_7} for $W=Y$.

Recall from \Cref{defq2} that ${\mathfrak Y} = \fB \cap \fBT$. By the inclusion-exclusion principle,
\begin{align}\label{eq_al1_5}
    \PP\bc{\mathfrak Y}= \PP\bc{\fB} + \PP\bc{\fBT} - \PP\bc{\fB \cup \fBT} = 1 - \PP\bc{\ffB} - \PP\bc{\ffBT} + \PP\bc{\ffB \cap \ffBT}.
\end{align}
Moreover, by Definitions \ref{defq} and \ref{dxyzuv},
\begin{enumerate}
  \item[(a)] $\ffB$ coincides with the event that $\{\supp{\bm{h}}\subseteq \mathcal{X}\bc{A[\bm{\theta}]}\cup \mathcal{Y}\bc{A[\bm{\theta}]}\cup \mathcal{V}\bc{A[\bm{\theta}]}\}$,
  \item[(b)] $\ffBT$ coincides with the event that $\{\supp{\bm{h}}\subseteq \mathcal{X}\bc{A[\bm{\theta}]}\cup \mathcal{Y}\bc{A[\bm{\theta}]}\cup \mathcal{U}\bc{A[\bm{\theta}]}\}$ and
  \item[(c)] $\ffB \cap \ffBT$ coincides with the event that $\{\supp{\bm{h}}\subseteq \mathcal{X}\bc{A[\bm{\theta}]}\cup \mathcal{Y}\bc{A[\bm{\theta}]}\}$.
\end{enumerate}
Given the \textit{number} of non-zero entries $\bfh$ of $\bm h$, the \textit{positions} of these non-zero entries are chosen uniformly at random from all $\bfh$-subsets of $[n]$, and independently of $A[\bm \theta]$. Moreover, by \cite{freedman1977sampling}, for any $k \geq 0$,
\begin{align}
\abs{ \frac{\binom{(\bm x+\bm y+\bm v)n}{n}}{\binom{n}{k}} - (\bm x+\bm y+\bm v)^k + \frac{\binom{(\bm x+ \bm y+\bm u)n}{n}}{\binom{n}{k}}  -(\bm x+\bm y+\bm u)^k - \frac{\binom{(\bm x+\bm y)n}{n}}{\binom{n}{k}}  + (\bm x+\bm y)^k}
 \leq \hspace{0.1 cm}  \frac{3k (k-1)}{2n}.
\end{align}
Thus
\begin{align*}
& \left|\mathbb{P}\bc{{\mathfrak Y}|\bm \zeta}-Y(\bm \zeta,\psi)\right| \\
 \leq & \hspace{0.1 cm} \sum_{k=0}^{\infty}\PP\bc{\bfh=k}\abs{ \frac{\binom{(\bm x+\bm y+\bm v)n}{n}}{\binom{n}{k}} - (\bm x+\bm y+\bm v)^k + \frac{\binom{(\bm x+ \bm y+\bm u)n}{n}}{\binom{n}{k}}  -(\bm x+\bm y+\bm u)^k - \frac{\binom{(\bm x+\bm y)n}{n}}{\binom{n}{k}}  + (\bm x+\bm y)^k}\\
 \leq & \hspace{0.1 cm} \mathbb{P}\left(\bfh\geq L\right) + \frac{3L(L-1)}{2n} = \mathbb{P}\left(\bfh\geq L\right) + o_{n,P}(1).
\end{align*}

\item[(ii)] We next prove \cref{eq_css1_7} for $W=U$.
 
 Recall from \Cref{defq2} that ${\mathfrak U} = \fB \cap \ffBT$. Moreover,
\begin{enumerate}
  \item[(a)] $\fB$ coincides with the event that $\{\supp{\bm{h}}\cap ( \mathcal{U}\bc{A[\bm{\theta}]}\cup \mathcal{Z}\bc{A[\bm{\theta}]}) \not= \emptyset\}$ and
  \item[(b)] $\ffBT$ coincides with the event that $\{\supp{\bm{h}}\subseteq \mathcal{X}\bc{A[\bm{\theta}]}\cup \mathcal{Y}\bc{A[\bm{\theta}]}\cup \mathcal{U}\bc{A[\bm{\theta}]}\}$.
\end{enumerate}
In other words, ${\mathfrak U}$ coincides with the event that $\supp{\bm{h}}$ is a subset of $ \mathcal{X}\bc{A[\bm{\theta}]}\cup \mathcal{Y}\bc{A[\bm{\theta}]}\cup \mathcal{U}\bc{A[\bm{\theta}]}$, but not of $\mathcal{X}\bc{A[\bm{\theta}]}\cup \mathcal{Y}\bc{A[\bm{\theta}]}$.
As before, using the total variation estimate between sampling with and without replacement of \cite{freedman1977sampling}, for any $k \geq 0$,
\begin{align}
  \abs{\frac{\binom{(\bm x+\bm y+\bm u)n}{n}}{\binom{n}{k}} - \frac{\binom{(\bm x+\bm y)n}{n}}{\binom{n}{k}} - \bc{(\bm x+\bm y+\bm u)^k-(\bm x+\bm y)^k}} \leq  \frac{2k(k-1)}{2n}.
\end{align}
Thus,
\begin{align*}
\left|\mathbb{P}\bc{{\mathfrak U}|\bm \zeta}-U(\zeta,\psi)\right|
& = \sum_{k=0}^{\infty}\PP\bc{\bfh=k}\abs{\frac{\binom{(\bm x+\bm y+\bm u)n}{n}}{\binom{n}{k}} - \frac{\binom{(\bm x+\bm y)n}{n}}{\binom{n}{k}} - \bc{(\bm x+\bm y+\bm u)^k-(\bm x+\bm y)^k}} \\
& \leq \mathbb{P}\left(\bfh\geq L\right)+ \frac{L(L-1)}{n} = \mathbb{P}\left(\bfh\geq L\right)+o_{n,P}(1).
\end{align*}

\item[(iii)] By symmetry, \cref{eq_css1_7} for $W=V$ follows as in (ii).
\item[(iv)]  We finally prove (\ref{eq_css1_7_new}).

Given the \textit{number} of non-zero entries $\bfh$ of $\bm h$, the \textit{positions} of these non-zero entries are chosen uniformly at random from all $\bfh$-subsets of $[n]$, and independently of $A[\bm \theta]$. Thus, conditionally on $\bfh$ and the proportions of types $\bm \zeta$ in $A[\bm \theta]$, the event $\mathfrak{Z}_{\circ}$ holds if and only if all of these $\bfh$ positions are chosen from the set $\cY(A[\bm \theta])$.
 Therefore, 
\begin{equation}\label{eq_al1_1}
    \mathbb{P}\bc{{\mathfrak Z}_{\circ}|\bm \zeta,\bfh}=\frac{\binom{n\bm y}{\bfh}}{\binom{n}{\bfh}}.
\end{equation}
On the other hand, by \cite{freedman1977sampling}, for any fixed $k \geq 0$,
\begin{align}\label{eq_al1_2}
   \left|\frac{\binom{n \bm y}{k}}{\binom{n}{k}}-\bm y^{k} \right| \leq \frac{k (k -1)}{2n}.
\end{align} 
Therefore, for any $L \in \NN_{\geq 2}$,
\[\begin{aligned}
 \abs{\mathbb{P}\bc{{\mathfrak Z}_{\circ}|\bm \zeta}-\psi(\bm{y})} &\leq\sum_{k=0}^{\infty} \mathbb P\bc{\bfh=k}\abs{\frac{\binom{n\bm y}{k}}{\binom{n}{k}} -\bm y^{k}}
 \leq \mathbb{P}\left(\bfh\geq L\right)+\sup_{0\leq k\leq L} \abs{\frac{\binom{n\bm y}{k}}{\binom{n}{k}} -\bm y^{k}}  \\
& \leq \mathbb{P}\left(\bfh\geq L\right) + \frac{L (L -1)}{2n} = \mathbb{P}\left(\bfh\geq L\right) + o_{n,P}(1).
\end{aligned}\]
\end{itemize}
\end{proof}

\subsubsection{Approximating the type probabilities for column {$n+1$}: Proof of \Cref{ccs1}}\label{sec_proccs1}
With the results of the previous two subsections, we are now in the position to prove \Cref{ccs1}.
\begin{proof}[Proof of \Cref{ccs1}]
\textit{At least one-sided firmly frozen variables - proof of \cref{eq_ccs1_or}:} For $W\in\cbc{Y,U,V}$, by the triangle inequality,
\begin{align} \label{eq_css1_1}
  &  \mathbb{E}\abs{\mathbb{P}\bc{n+1\in\mathcal{W}\bc{A^{\bm{h}}[\bm{\theta}]}\big|\vze}-W(\vze,\psi)} 
\leq \mathbb{E}\abs{\mathbb{P}\bc{n+1\in\mathcal{W}\bc{A^{\bm{h}}[\bm{\theta}]}\big|\vze}-\mathbb{P}\bc{{\mathfrak W}|\vze}} + \mathbb{E}\abs{\mathbb{P}\bc{{\mathfrak W}|\vze}-W(\vze,\psi)}.
\end{align}
We bound both summands on the right hand side of (\ref{eq_css1_1}) separately, beginning with the first. By conditional Jensen's inequality and the tower property,
\begin{align}\label{eq_css1_2}
\mathbb{E}\left|\mathbb{P}\bc{n+1\in\mathcal{W}\bc{A^{\bm{h}}[\bm{\theta}]}\Big|\vze}-\mathbb{P}\bc{\mathfrak W|\vze}\right|
 &\leq  \mathbb{E}\left|\teo{n+1\in\mathcal{W}\bc{A^{\bm{h}}[\bm{\theta}]}}-\ind \mathfrak W \right|  \\
 &\leq   \mathbb{P}\left(n+1\in\mathcal{W}\bc{A^{\bm{h}}[\bm{\theta}]},\mathfrak W^c\right) + \mathbb{P}\left(n+1\not\in\mathcal{W}\bc{A^{\bm{h}}[\bm{\theta}]},\mathfrak W \right). \nonumber
\end{align}
Now, let $\mathcal{I}_{W} = \{Y, U, V, XZ\} \setminus \{W\}$. Since the type events apart from ${\mathfrak Z}_{\circ}$ are pairwise disjoint (see \Cref{reuplus}), 
\begin{align}\label{eq_css1_3}
{\mathfrak W}^c = \biguplus_{\cW' \in \mathcal{I}_{\cW}} {\mathfrak W}'. 
\end{align}
Thus, with (\ref{eq_css1_3}) and using the abbreviation  $\mathcal{XZ}\bc{A^{\bm{h}}[\bm{\theta}]}$ to denote the union $\mathcal{X}\bc{A^{\bm{h}}[\bm{\theta}]} \cup \mathcal{Z}\bc{A^{\bm{h}}[\bm{\theta}]}$, we obtain
\begin{align}\label{eq_css1_4}
    &\mathbb{P}\left(n+1\in\mathcal{W}\bc{A^{\bm{h}}[\bm{\theta}]},\mathfrak W^c\right)
    \leq \sum_{W' \in \mathcal{I}_{W}}\mathbb{P}\left(n+1\in\mathcal{W}\bc{A^{\bm{h}}[\bm{\theta}]},{\mathfrak W}'\right)
    \leq \sum_{W' \in \mathcal{I}_{W}}\mathbb{P}\left(n+1\notin\mathcal{W}'\bc{A^{\bm{h}}[\bm{\theta}]},{\mathfrak W}'\right).
\end{align}
Plugging (\ref{eq_css1_4}) into (\ref{eq_css1_2}) and using \Cref{l11} on all four summands yields 
\begin{align}\label{eq_css1_6}
    &\mathbb{E}\left|\mathbb{P}\bc{n+1\in\mathcal{W}\bc{A^{\bm{h}}[\bm{\theta}]}\Big|\vze}-\mathbb{P}\bc{\mathfrak W|\vze}\right|
    \leq 6L!\delta + 6 \PP(\bfh\geq L) + o_{n,P}(1).
\end{align}
Finally, the upper bound on the second summand $\mathbb{E}\abs{\mathbb{P}\bc{{\mathfrak W}|\vze}-W(\vze,\psi)}$ on the right hand side of (\ref{eq_css1_1}) follows immediately from (\ref{eq_css1_7}) in \Cref{al1}.
Plugging the two bounds  (\ref{eq_css1_7}) and (\ref{eq_css1_6}) into (\ref{eq_css1_1}) gives \cref{eq_ccs1_or}.

\textit{Nowhere frozen variables - proof of \cref{eq_ccs1_z}:} 
 Since $(a+b)^-\leq a^- +b^-$ and $a^-\leq \abs{a}$,
\begin{align*} 
  \mathbb{E}\brk{\bc{\mathbb{P}\bc{n+1\in\mathcal{Z}\bc{A^{\bm{h}}[\bm{\theta}]}\big|\vze}-\psi\bc{\vy}}^-}
  \leq  &  \mathbb{E}\brk{\bc{\mathbb{P}\bc{n+1\in\mathcal{Z}\bc{A^{\bm{h}}[\bm{\theta}]}\big|\vze}-\mathbb{P}\bc{{\mathfrak Z}_{\circ}|\vze}}^-} + \mathbb{E}\abs{\mathbb{P}\bc{{\mathfrak Z}_{\circ}|\vze}-\psi\bc{\vy}}\\
  \leq & \mathbb{P}\left(n+1\not\in\mathcal{Z}\bc{A^{\bm{h}}[\bm{\theta}]}, {\mathfrak Z}_{\circ}\right)+\mathbb{E}\abs{\mathbb{P}\bc{{\mathfrak Z}_{\circ}|\vze}-\psi\bc{\vy}}.
\end{align*} 
Equation \cref{eq_ccs1_z} now follows from \Cref{l12} and \cref{eq_css1_7_new} in \Cref{al1}.

\end{proof}

\section{Analysis of the rank-difference}\label{sec_newapp}

In \Cref{sec_prov}, we have reduced the lower bound of \Cref{t1} to Propositions \ref{lem_rkdif_rank}, \ref{imp_pro} and \Cref{fact_1}. This section is devoted to the proof of those three results, of which \Cref{imp_pro} requires the most efforts. 
Our starting points here are \Cref{al11}, the fixed point equations for the proportions of frozen types, and \cref{eq_sumone}:
\begin{align}
   & \vy_{n,t/n}=1-\phi_t(\vx_{n,t/n}+\vy_{n,t/n}+\vu_{n,t/n})-\phi_t(\vx_{n,t/n}+\vy_{n,t/n}+\vv_{n,t/n})+\phi_t(\vx_{n,t/n}+\vy_{n,t/n})+\oone;\label{eq_fpy}\\
  &  \vu_{n,t/n}=\phi_t(\vx_{n,t/n}+\vy_{n,t/n}+\vu_{n,t/n})-\phi_t(\vx_{n,t/n}+\vy_{n,t/n})+\oone;\label{eq_fpu}\\
 &   \vv_{n,t/n}=\phi_t(\vx_{n,t/n}+\vy_{n,t/n}+\vv_{n,t/n})-\phi_t(\vx_{n,t/n}+\vy_{n,t/n})+\oone;\label{eq_fpv}\\
  &  \vz_{n,t/n}\geq \phi_t(\vy_{n,t/n})+\oone;\label{eq_fpz}\\
  & \vx_{n,t/n}+\vy_{n,t/n}+\vz_{n,t/n}+\vu_{n,t/n}+\vv_{n,t/n}=1.\label{eq_fpsum}
\end{align}
The combination of \cref{eq_fpsum,eq_fpy,eq_fpu} gives that
\begin{align}\label{eq_fpxz}
    \vx_{n,t/n}+\vz_{n,t/n}=1-\vy_{n,t/n}-\vu_{n,t/n}-\vv_{n,t/n}=\phi_t\bc{\vx_{n,t/n}+\vy_{n,t/n}}+\oone;
\end{align}
\Cref{eq_fpy,eq_fpu,eq_fpv,eq_fpz,eq_fpsum,eq_fpxz}, as well as \Cref{l91}, are the main results from the previous sections and the proofs in this section highly depend on them.

\subsection{The rank increase: Proof of \Cref{lem_rkdif_rank}}\label{sec_pr_r}
Recall the function $h_t:[0,1] \to \RR, h_t\bc{\alpha}=\alpha+1-\phi_t\bc{\alpha}$ from \cref{def_ht} as well as \Cref{lem_rkdif_rank} from \Cref{sec_rifpm}:
\randif*
\begin{proof}
Recall the good event $\mathfrak{P}_n$ from \cref{event_P}. On $\mathfrak{P}_n$, the matrix $\bm{T}_{n,t/n}[\bm{\theta}]$ arises from the matrix $\bm{T}_{n+1,t/n}[\bm{\theta}]$ through removal of the $(n+1)$st row and column, and therefore, \Cref{lr} gives the following representation of their rank difference in terms of the type of $n+1$:
\begin{align*}
    \rk{\bm{T}_{n+1,t/n}[\bm{\theta}]}-\rk{\bm{T}_{n,t/n}[\bm{\theta}]}
    =&\ind\cbc{n+1 \in \cX\bc{\bm{T}_{n+1,t/n}[\bm{\theta}]}}+2\cdot \ind\cbc{n+1 \in \cY\bc{\bm{T}_{n+1,t/n}[\bm{\theta}]}}\\
    &+\ind\cbc{n+1 \in \cU\bc{\bm{T}_{n+1,t/n}[\bm{\theta}]}}
    +\ind\cbc{n+1 \in \cV\bc{\bm{T}_{n+1,t/n}[\bm{\theta}]}}.
\end{align*}
On the other hand, in any case, $\abs{\rk{\bm{T}_{n+1,t/n}[\bm{\theta}]}-\rk{\bm{T}_{n,t/n}[\bm{\theta}]}}\leq 2P+2$, since both matrices can be obtained from $\bm{T}_{n,t/n}$ by adding at most $P+1$ rows and at most $P+1$ columns. By \cref{oe}, the above equation holds with high probability.   Hence,
 \begin{equation}\label{eq_n1sum}\begin{aligned}
    &\Erw\brk{\rk{\bm{T}_{n+1,t/n}[\bm{\theta}]}-\rk{\bm{T}_{n,t/n}[\bm{\theta}]}}\\
    =& \hspace{0.1 cm} \PP\bc{n+1 \in \cX\bc{\bm{T}_{n+1,t/n}[\bm{\theta}]}} +2\cdot \PP\bc{n+1 \in \cY\bc{\bm{T}_{n+1,t/n}[\bm{\theta}]}}+\PP\bc{n+1 \in \cU\bc{\bm{T}_{n+1,t/n}[\bm{\theta}]}}\\
    &+\PP\bc{n+1 \in \cV\bc{\bm{T}_{n+1,t/n}[\bm{\theta}]}}+o_n(1).
\end{aligned}\end{equation}
On the other hand, by \Cref{cor_proportions}, for any $\cW\in\cbc{\cX,\cY,\cZ,\cU,\cV}$ and any $i\in [n+1]$, $\PP(i \in \cW\bc{\bm{T}_{n+1,t/n}[\bm{\theta}]})=\Erw[\vw_{n+1,t/n}]$,
 and \Cref{l91} shows that \[\Erw\brk{\vw_{n+1,t/n}}=\Erw\brk{\vw_{n,t/n}}+o_{n,P}(1),\ \mbox{uniformly in $t\in[0,d]$}.\]
Therefore, \cref{eq_n1sum} reduces to
\begin{align*}
    \Erw\brk{\rk{\bm{T}_{n+1,t/n}[\bm{\theta}]}-\rk{\bm{T}_{n,t/n}[\bm{\theta}]}}
    =\Erw\brk{\vx_{n,t/n}+2\vy_{n,t/n}+\vu_{n,t/n}+\vv_{n,t/n}}+o_{n,P}(1),
\end{align*}
uniformly in $t\in[0,d]$. 
Since $\val_{n,t/n}=\vx_{n,t/n}+\vy_{n,t/n}+\vv_{n,t/n}$
as observed in in \cref{eq_alphaxyuv}, the combination of \cref{eq_fpy,eq_fpu} gives that
\begin{align*}
    \Erw\brk{\rk{\bm{T}_{n+1,t/n}[\bm{\theta}]}-\rk{\bm{T}_{n,t/n}[\bm{\theta}]}}
    =&\Erw\brk{\bc{\vx_{n,t/n}+\vy_{n,t/n}+\vv_{n,t/n}}+\bc{\vy_{n,t/n}+\vu_{n,t/n}}}+o_{n,P}(1)\\
    =&\Erw\brk{\val_{n,t/n}+1-\phi_t\bc{\val_{n,t/n}}}+o_{n,P}(1)=\Erw\brk{h_t\bc{\val_{n,t/n}}}+o_{n,P}(1),
\end{align*}
uniformly in $t\in[0,d]$, as desired.
\end{proof}

\subsection{Lower bound on the rank increase: Proof of \Cref{imp_pro}}\label{sec_pr_imp}

In this section, we prove \Cref{imp_pro}:
\imp*

The proof of \Cref{imp_pro} heavily depends on the properties of the function $G_t$ defined in (\ref{fun_g}) and its zeroes: Recall that $G_t:[0,1] \to \mathbb{R}$,
$$
    G_t(\alpha) = \alpha+ \phi_t\bc{1-\phi_t\bc{\alpha}}-1$$
and $\alpha_\star(t)$ and $\alpha^\star(t)$ were defined as the smallest and the largest zeroes of $G_t$ in $[0,1]$, respectively. Moreover, $\alpha_0(t)$ denotes the unique zero of the increasing function $\Xi_t:[0,1]\to \RR$, $\Xi_t(\alpha) = \alpha+\phi_t(\alpha)-1$, which is also always a zero of $G_t$ (see \Cref{lem_easy_fix}). With this terminology, we note the following properties of $G_t$ and its zeroes:


\begin{restatable}[Useful properties of $G_t$ and its zeroes; see {\cite[Section 3]{coja2022sparse}}]{lemma}{lemproal}\label{lem_proal}
\hspace{4 cm}

\begin{enumerate}
    \item\label{it_1} For $t\in [0,\eul]$, $G_t$ is strictly increasing and has a unique zero: $\alpha_\star(t)=\alpha_0(t)=\alpha^\star(t)$.

    \item\label{it_3} For $t \in (\eul, \infty)$, $G_t$ has exactly three distinct zeroes $\alpha_\star(t)<\alpha_0(t)<\alpha^\star(t)$, and $\alpha_0(t)\geq 1-\ln t/t$.
    
    \item\label{it_4} For all $t \geq 0$, $\alpha_\star(t)=1-\phi_t\bc{\alpha^\star(t)}$ and $\alpha^\star(t)=1-\phi_t\bc{\alpha_\star(t)}$.
    
    \item\label{it_5} For $t \in (\eul, \infty)$, $G_t$ is positive on $(\alpha_\star(t),\alpha_0(t)) \cup (\alpha^\star(t),1]$ and negative on $[0,\alpha_\star(t)) \cup (\alpha_0(t),\alpha^\star(t))$. Moreover, $G_t$ is strictly increasing on $[\alpha^\star(t),1]$.
    
    \item\label{it_6} For $t\not= \eul$, $G_t$ and $G'_t$ have no common zero. For $t=\eul$, their unique common zero is given by $\alpha_0(\eul)=1-1/\eul$.

\item\label{it_8} For all $t >0$ and $\alpha \in [0,1]\setminus\{\alpha_\star(t), \alpha^\star(t)\}$, $R_t(\alpha^\star(t))=R_t(\alpha_\star(t)) < R_t(\alpha)$.

\item\label{it_7} The functions $t \mapsto \alpha_\star(t)$, $t \mapsto \alpha_0(t)$ and $t \mapsto \alpha^\star(t)$ are differentiable on $[0,\infty)$ with continuous derivatives on $(0,\eul)\cup(\eul,\infty)$.

\item \label{it_9} Let $(\bm{b}_{n,P,N,J_N,t})_{n, P,N\in \ZZ^+, J_N \in \syN, t\in [0,d]} \subseteq [0,1]$ be an arbitrary family of 
random variables. If $ G_t(\bm{b}_{n,P,N,J_N,t})=\bar{o}_{\mathbb{P}}(1)$, then also
\begin{align*}
  \min\cbc{\abs{\bm{b}_{n,P,N,J_N,t}-\alpha_\star(t)},\abs{\bm{b}_{n,P,N,J_N,t}-\alpha_0(t)},\abs{\bm{b}_{n,P,N,J_N,t}-\alpha^\star(t)}}=\bar{o}_{\mathbb{P}}(1).
    \end{align*}
\end{enumerate}
\end{restatable}

We emphasize that a large part of \Cref{lem_proal} is covered by the results of \cite[Section 3]{coja2022sparse}. On the other hand, several of the specific properties that we need only arise in proofs, and are correspondingly difficult to cite. For the sake of completeness and easy reference, we therefore give a proof of all properties that we need in Appendix \ref{app_proal}.

As a last preparation for the proof of \Cref{imp_pro}, we prove two short lemmas on the evaluation of $h_t$ at specific points:
 The first lemma shows that $\alpha_\star(t)$ and $\alpha^\star(t)$ minimize $ h_t$ among the zeroes of $G_t$. It is a direct consequence of \Cref{lem_proal}:
\begin{lemma}\label{imp_lem}
For any $t\geq 0$, 
\begin{align}\label{eq_h_al}
    h_t\bc{\alpha_\star(t)}=h_t\bc{\alpha^\star(t)}\leq h_t\bc{\alpha_0(t)}.
\end{align}
\end{lemma}

\begin{proof}
By \cref{it_1,it_3} in \Cref{lem_proal},
    \[
    \alpha_\star(t)= \alpha_0(t)= \alpha^\star(t)\mbox{ for $t\leq \eul \qquad $ and }\qquad \alpha^\star(t)>\alpha_0(t)\geq 1-\ln t/t \mbox{ for $t> \eul$}.
    \]
Taking the derivative of $h_t$ w.r.t.\ $\alpha$, we have $h_t'(\alpha)=1-t \phi_t(\alpha)=1-t\eul^{t(\alpha-1)}$,
so 
\begin{equation}\label{mon_ht}
    \mbox{$\alpha \mapsto h_t(\alpha)$ is a strictly increasing function on $[0,1-\ln t/t]$ and a strictly decreasing function on $[1-\ln t/t,1]$}
\end{equation}
and thus $h_t\bc{\alpha^\star(t)}\leq h_t\bc{\alpha_0(t)}$. It thus only remains to show that $h_t\bc{\alpha_\star(t)}=h_t\bc{\alpha^\star(t)}$.
By \cref{it_4} in \Cref{lem_proal}, $1-\alpha_\star(t) = \phi_t(\alpha^\star(t))$ and $1-\alpha^\star(t) = \phi_t(\alpha_\star(t)).$
It now directly follows that
\[h_t\bc{\alpha_\star(t)}=\alpha_\star(t)+1-\phi_t\bc{\alpha_\star(t)}  = \alpha^\star(t)+\alpha_\star(t)=\alpha^\star(t)+1-\phi_t\bc{\alpha^\star(t)}=h_t\bc{\alpha^\star(t)}.\]
\end{proof}

The second lemma is a consequence of the type fixed point equations (\ref{eq_fpy}) to (\ref{eq_fpxz}): 

\begin{lemma}\label{imp_lem0}
For any $d>0$, 
\begin{equation}\label{eq_eqineq}
\begin{aligned}
    h_t\bc{\vx_{n,t/n}+\vy_{n,t/n}+\vu_{n,t/n}}=& \hspace{0.1 cm}  h_t\bc{\vx_{n,t/n}+\vy_{n,t/n}+\vv_{n,t/n}}+\oone
    = h_t\bc{\vx_{n,t/n}+\vy_{n,t/n}}+\oone\\
    \leq& \hspace{0.1 cm}  h_t\bc{\vy_{n,t/n}}+\oone.
    \end{aligned}
\end{equation}
\end{lemma}

\begin{proof}
The first and second equalities in \cref{eq_eqineq} follow directly from \cref{eq_fpu,eq_fpv}.

On the other hand, the combination of \cref{eq_fpxz,eq_fpz} gives
    \begin{align*}
            h_t\bc{\vx_{n,t/n}+\vy_{n,t/n}}=& \hspace{0.1 cm} \vx_{n,t/n}+\vy_{n,t/n}+1-\phi_t\bc{\vx_{n,t/n}+\vy_{n,t/n}}=\vy_{n,t/n}+1-\vz_{n,t/n}+\oone\\
            \leq&  \hspace{0.1 cm} \vy_{n,t/n}+1-\phi_t\bc{\vy_{n,t/n}}+\oone=h_t\bc{\vy_{n,t/n}}+\oone,
    \end{align*}
and thus the last inequality in \cref{eq_eqineq} follows.
\end{proof}


With \Cref{imp_lem,imp_lem0} in hand, we are finally in the position to prove \Cref{imp_pro}. 

\begin{proof}[Proof of \Cref{imp_pro}]
Define
    \begin{align}\label{def_vtavet}
    \bvta_n=\teo{h_t'(\vx_{n,t/n}+\vy_{n,t/n})\geq 0}\quad \mbox{and} \quad\bvet_n=1-\bvta_n=\teo{h_t'(\vx_{n,t/n}+\vy_{n,t/n})<0}.
    \end{align}
Since $\bvta_n+\bvet_n=1$, we divide equation \cref{eq_eqeqh} into two parts as follows:
\begin{align}
   &\bvta_n\bc{ h_t\bc{\alpha^\star(t)} - h_t\bc{\val_{n,t/n}}}\leq \oone,\label{eq_taur}\\
\mbox{and}\quad&\bvet_n  \bc{h_t\bc{\alpha^\star(t)}-h_t\bc{\val_{n,t/n}}}\leq \oone.\label{eq_etar}
\end{align}
In the absence of the error terms $\oone$ and under the assumption that $\bvta_n\equiv 1$ (or $\bvet_n\equiv 1$), the proof of \Cref{imp_pro} would amount to an analytic treatment of the properties of $h_t$. Unfortunately, we have to deal with the error terms and both cases. In the ensuing argument, we therefore fall back upon Taylor's Theorem with Lagrange Remainder, \cref{it_9} in \Cref{lem_proal} and the following two facts:
\begin{itemize}\label{eq_01eq}
    \item[(i)] For any function $g$ and $\upsilon\in\cbc{0,1}$,  if $\upsilon a=\upsilon b$, then $\upsilon g(a)=\upsilon g(b)$;
    \item[(ii)] For a family of differentiable functions $\bc{g_t}_{t\in[0,d]}$, if there exists a uniform bound $b$ such that $\sup_{t\in [0,d],s\in[0,1]}\abs{g'_t(s)}\leq b$, then $g_t(\bm{a}_t')=g_t(\bm{a}_t)+\oone$ for any random variables $\bm{a}_t,\bm{a}'_t\in [0,1]$, $\bm{a}_t'=\bm{a}_t+\oone\in [0,1]$ and $t\in [0,d]$,  since $\abs{g_t(\bm{a}_t')-g_t(\bm{a}_t)}\leq b \abs{\bm{a}_t'-\bm{a}_t}$.
\end{itemize}

\begin{enumerate}
    \item \begin{proof}[Proof of \Cref{eq_taur}]
  By definition of $\bvta_n$, $\bvta_n h_t'\bc{\vx_{n,t/n}+\vy_{n,t/n}}\geq 0$. 
  Fix $\varepsilon\in (0,1/d)$ and let $b=\inf_{t\in [\varepsilon d,d]}\cbc{t^2 \eul^{-t}}>0$, such that $\sup_{\alpha \in [0,1]} h_t''\bc{\alpha}\leq-b$ for $t \in [\varepsilon d, d]$. Then by Taylor's Theorem with Lagrange remainder, for $t\in [\varepsilon d,d]$,
    \begin{align*}
            \bvta_n h_t(\vy_{n,t/n}) 
            \leq& \hspace{0.1 cm} \bvta_n \bc{ h_t\bc{\vx_{n,t/n}+\vy_{n,t/n}}- h_t'\bc{\vx_{n,t/n}+\vy_{n,t/n}} \vx_{n,t/n}-\frac{b}{2} \vx_{n,t/n}^2}\\
            \leq& \hspace{0.1 cm} \bvta_n h_t\bc{\vx_{n,t/n}+\vy_{n,t/n}}
    -\bvta_n\teo{t\in [\varepsilon d,d]}\frac{b}{2}\vx_{n,t/n}^2.
    \end{align*}
 Analogously, let $c=1-\varepsilon d>0$ such that $\inf_{\alpha\in [0,1]} h_t'(\alpha)\geq c$ for $t\in [0,\varepsilon d)$. Then by Taylor's Theorem with Lagrange remainder, for $t\in [0,\varepsilon d)$,
 \begin{align*}
            \bvta_n h_t(\vy_{n,t/n})\leq&  \bvta_n\bc{ h_t\bc{\vx_{n,t/n}+\vy_{n,t/n}}- c \vx_{n,t/n}}.
    \end{align*}
On the other hand, \cref{eq_eqineq}  shows that for all $t \in [0,d]$,
\begin{align} \label{eq_noname}
    \bvta_n h_t(\vy_{n,t/n})\geq \bvta_nh_t\bc{\vx_{n,t/n}+\vy_{n,t/n}}+\oone.
\end{align}
Since $\vx_{n,t/n}\in [0,1]$, (\ref{eq_noname}) implies that
\[ \min\cbc{\frac{b}{2},c}\bvta_n\vx_{n,t/n}^2\leq \bvta_n\teo{t\in [\varepsilon d,d]}\frac{b}{2}\vx_{n,t/n}^2+\bvta_n\teo{t\in [0,\varepsilon d)}c \vx_{n,t/n}\leq \oone,\]
and we conclude that that $\bvta_n\vx_{n,t/n}^2 = \oone$. The Cauchy-Schwarz inequality $\Erw[\bvta_n\vx_{n,t/n}] \leq \Erw[\bvta_n\vx_{n,t/n}^2]^{1/2}$ then yields 
\begin{align}\label{eq_vtavx}
    \bvta_n\vx_{n,t/n}= \oone.
\end{align}
Since $\val_{n,t/n} = \vx_{n,t/n} + \vy_{n,t/n} + \vv_{n,t/n}$ and $\hval_{n,t/n} = \vx_{n,t/n} + \vy_{n,t/n} + \vu_{n,t/n}$, \cref{eq_vtavx} in combination with \cref{eq_fpy} and \cref{eq_fpv} implies that 
\begin{align*}
    \bvta_n\val_{n,t/n}=& \hspace{0.1 cm}\bvta_n\bc{\vy_{n,t/n}+\vv_{n,t/n}+\oone}=\bvta_n\bc{1-\phi_t\bc{\vx_{n,t/n}+\vy_{n,t/n}+\vu_{n,t/n}}+\oone}\\
    =& \hspace{0.1 cm} \bvta_n(1-\phi_t(\hval_{n,t/n})+\oone).
\end{align*}
Analogously, \cref{eq_vtavx} in combination with \cref{eq_fpy} and \cref{eq_fpu} implies that $\bvta_n\hval_{n,t/n}=\bvta_n\bc{1-\phi_t\bc{\val_{n,t/n}}+\oone}$.
Hence, 
    \begin{align*}
            \bvta_n\val_{n,t/n}=\bvta_n(1-\phi_t(\hval_{n,t/n})+\oone)=\bvta_n(1-\phi_t\bc{1-\phi_t\bc{\val_{n,t/n}}}+\oone),
    \end{align*}
i.e.,
\begin{align}\label{eq_Gat_al}
\bvta_n G_t\bc{\val_{n,t/n}}=\oone.
\end{align}
Let $\vbe_{n,t/n}=\bvta_n\val_{n,t/n}+\bvet_n\alpha^\star(t)$. Since $G_t(\alpha^\star(t))=0$, \cref{eq_Gat_al} implies that
    \begin{align}\label{eq_vbe}
   G_t(\vbe_{n,t/n})= \bvta_n G_t\bc{\val_{n,t/n}}+\bvet_n G_t\bc{\alpha^\star(t)}=\oone.
    \end{align}
Hence, \cref{it_9} in \Cref{lem_proal} implies that
    \[
    \min\cbc{\abs{\vbe_{n,t/n}-\alpha_\star(t)},\abs{\vbe_{n,t/n}-{\alpha_0(t)}},\abs{\vbe_{n,t/n}-\alpha^\star(t)}}=\oone.
    \]
By \Cref{imp_lem}, $h_t\bc{\alpha_\star(t)}=h_t\bc{\alpha^\star(t)}\leq h_t\bc{\alpha_0(t)}$, so 
    \[
    h_t\bc{\alpha^\star(t)}\leq h_t\bc{\vbe_{n,t/n}}+\oone=\bvta_n h_t\bc{\val_{n,t/n}}+\bvet_n h_t\bc{\alpha^\star(t)}+\oone,\]
and \cref{eq_taur} follows immediately.
\end{proof}

\item \begin{proof}[Proof of \Cref{eq_etar}]
By definition of $\bvet_n$, $\bvet_n h_t'\bc{\vx_{n,t/n}+\vy_{n,t/n}}< 0$. Since the function $h_t$ is strictly increasing on $[0,1-\ln(t)/t]$, this implies that $\bvet_n \bc{\vx_{n,t/n}+\vy_{n,t/n}}\geq \bvet_n \bc{1-\ln t/t}$, so that
    \begin{align}
    \label{eq_vallow}
    \bvet_n\val_{n,t/n}\geq\bvet_n\bc{\vx_{n,t/n}+\vy_{n,t/n}}\geq \bvet_n\bc{1-\ln t/t}.
    \end{align}

Another application of Taylor's Theorem with Lagrange remainder to \cref{eq_fpu} and \cref{eq_fpv} as in the argument leading to \cref{eq_vtavx} 
yields that $\bvet_n\vu_{n,t/n}=\oone$ and $\bvet_n\vv_{n,t/n}=\oone$. 
Hence, by \cref{eq_fpy}, 
    \begin{align}
    \label{eq_fixvet}
    \bvet_n\vy_{n,t/n}=\bvet_n(1-\phi_t(\vx_{n,t/n}+\vy_{n,t/n})+\oone)\ \mbox{and }\ \bvet_n\val_{n,t/n}=\bvet_n(\vx_{n,t/n}+\vy_{n,t/n}+\oone).
    \end{align}

Let
    \[
    \vbe_{n,t/n}'=\bvet_n\teo{\val_{n,t/n}> {\alpha^\star(t)}}\val_{n,t/n}+\bc{1-\bvet_n\teo{\val_{n,t/n}> {\alpha^\star(t)}}}\alpha^\star(t).
    \]
Then $\vbe_{n,t/n}'\geq \alpha^\star(t)$. 
Since $G_t(\alpha^\star(t))=0$, by \cref{eq_fixvet,eq_fpz,eq_fpsum},
    \[
    \begin{aligned}
    G_t(\vbe_{n,t/n}')=&\bvet_n\teo{\val_{n,t/n}> {\alpha^\star(t)}}\bc{\val_{n,t/n}+\phi_t\bc{1-\phi_t\bc{\val_{n,t/n}}}-1}\\
    \leq& \bvet_n\teo{\val_{n,t/n}> {\alpha^\star(t)}}\bc{\vx_{n,t/n}+\vy_{n,t/n}+\vz_{n,t/n}-1+\oone}\leq  \oone.
    \end{aligned}
    \]

On the other hand, by \cref{it_5} in \Cref{lem_proal}, $G_t$ is strictly increasing on $[\alpha^\star(t),1]$. Hence     \[
    G_t(\vbe_{n,t/n}')\geq G_t\bc{\alpha^\star(t)}=0, \qquad \text{so that} \qquad    G_t(\vbe_{n,t/n}')=\oone.
    \]
Then the combination of \cref{it_9} in \Cref{lem_proal} and $\vbe'_{n,t/n}\geq \alpha^\star(t)$ yields that $\vbe'_{n,t/n}=\alpha^\star(t)+\oone$, which leads to
    \begin{align}
    \label{eq_valup}
    \bvet_n\val_{n,t/n}\leq \bvet_n\vbe'_{n,t/n}=\bvet_n\bc{\alpha^\star(t)+\oone}.
    \end{align}
Hence, by \cref{mon_ht}, \cref{eq_vallow} and \cref{eq_valup}, 
    \[
    \bvet_n h_t\bc{\val_{n,t/n}}\geq \bvet_n h_t\bc{\vbe_{n,t/n}'}=\bvet_n\bc{h_t\bc{\alpha^\star(t)}+\oone},
    \]
and thus \cref{eq_etar} holds.
\end{proof}
\end{enumerate}
The combination of \cref{eq_taur,eq_etar} gives \cref{eq_eqeqh}.
\end{proof}

\subsection{Integral evaluation: Proof of \Cref{fact_1}}\label{sec_eqqdrd}
In this section, we prove 
 \Cref{fact_1} from Section \ref{sec_lower_exp}:
\qeqaulr*
\begin{proof}
Throughout the proof, we use the abbreviations
\[
    q(d) = \int_{0}^{d} h_t(\alpha^\star(t))\dif t = \int_{0}^{d}\bc{\alpha^\star(t)-\phi_t\bc{\alpha^\star(t)}+1}\dif t\]
    and
   \[ r(d) = d \cdot R_d(\alpha^\star(d)) = 2d-d\phi_d\bc{1-\phi_d(\alpha^\star(d))}-d\phi_d\bc{\alpha^\star(d)}-d^2\phi_d\bc{\alpha^\star(d)}\bc{1-\alpha^\star(d)}.
    \]
We have $q(0)=r(0)=0$. Moreover, by \cref{it_7} in \Cref{lem_proal}, the function $t \mapsto \alpha^\star(t)$ is continuous on $[0,\infty)$, which then transfers to the functions $d \mapsto q(d)$ and $d \mapsto r(d)$.
  In order to prove that $q(d)=r(d)$ for all $d \geq 0$, it is thus sufficient to certify that $q'(d)=r'(d)$ for all $d\in (0,\eul)\cup(\eul,\infty)$.

Recall that the derivative of $d \mapsto \alpha^\star(d)$  is continuous on $(0,\eul)\cup(\eul,\infty)$ by  \cref{it_7} in \Cref{lem_proal}. To derive an expression for $r'(d)$, we compute the partial derivatives of the function $(d,\alpha) \mapsto \phi_d\bc{1-\phi_d(\alpha)}$ at $(d,\alpha^\star(d))$ for $d \not= \eul$, where we use that $\alpha^\star(d)$ is a zero of $G_d$, i.e., $\phi_d\bc{1-\phi_d\bc{\alpha^\star(d)}}=1-\alpha^\star(d)$:
\begin{align*}
    \frac{\partial }{\partial d}\phi_d\bc{1-\phi_d(\alpha)}\Big|_{\alpha=\alpha^\star(d)}=&-\phi_d\bc{\alpha^\star(d)}\phi_d\bc{1-\phi_d\bc{\alpha^\star(d)}}\bc{1-d\bc{1-\alpha^\star(d)}}\\
    =&-\phi_d\bc{\alpha^\star(d)}\bc{1-\alpha^\star(d)}\bc{1-d\bc{1-\alpha^\star(d)}}
\end{align*}
as well as
\[\frac{\partial }{\partial \alpha}\phi_d\bc{1-\phi_d(\alpha)}\Big|_{\alpha=\alpha^\star(d)}=-d^2\phi_d\bc{\alpha^\star(d)}\phi_d\bc{1-\phi_d\bc{\alpha^\star(d)}}=-d^2\phi_d\bc{\alpha^\star(d)}(1-\alpha^\star(d)).\]
Hence,
\begin{align}
    \label{eq_rd}r'(d)=&2-\phi_d\bc{1-\phi_d(\alpha^\star(d))}-d\frac{\partial }{\partial d}\phi_d\bc{1-\phi_d(\alpha)}\Big|_{\alpha=\alpha^\star(d)}-d\frac{\partial }{\partial \alpha}\phi_d\bc{1-\phi_d(\alpha)}\Big|_{\alpha=\alpha^\star(d)}\frac{\dif \alpha^\star(d)}{\dif d}\\
    &- \phi_d\bc{\alpha^\star(d)}-d\phi_d\bc{\alpha^\star(d)}\bc{\alpha^\star(d)-1+d\frac{\dif \alpha^\star(d)}{\dif d}}-2d\phi_d\bc{\alpha^\star(d)}\bc{1-\alpha^\star(d)}\nonumber\\
&-d^2\phi_d\bc{\alpha^\star(d)}\bc{\alpha^\star(d)-1+d\frac{\dif \alpha^\star(d)}{\dif d}}\bc{1-\alpha^\star(d)}+d^2\phi_d\bc{\alpha^\star(d)}\frac{\dif \alpha^\star(d)}{\dif d}.\nonumber
\end{align}
Substituting the two partial derivatives of $(d,\alpha) \mapsto \phi_d\bc{1-\phi_d(\alpha)}$ into \cref{eq_rd}, we see that the sum of the terms with ${\dif \alpha^\star(d)}/{\dif d}$ on the right hand side vanishes and 
\begin{align*}
    r'(d)=2-\phi_d\bc{1-\phi_d(\alpha^\star(d))}-\phi_d\bc{\alpha^\star(d)}.
\end{align*}

On the other hand, $q'(d)=\alpha^\star(d)-\phi_d\bc{\alpha^\star(d)}+1=r'(d)$ since $\phi_d\bc{1-\phi_d\bc{\alpha^\star(d)}}=1-\alpha^\star(d)$.
Thus $r(d)=q(d)$ for all $d\geq 0$, as desired.
\end{proof}

\paragraph{\bf Acknowledgement.}This research was supported by the European Union’s Horizon 2020 research and innovation programme under the Marie Skłodowska-Curie grant agreement no. 945045, and by the NWO Gravitation project NETWORKS under grant no. 024.002.003.

\bibliographystyle{abbrv}
\DeclareRobustCommand{\VAN}[3]{#3}
\bibliography{references}

\begin{thebibliography}{10}

\bibitem{addario2014hitting}
L.~Addario-Berry and L.~Eslava.
\newblock Hitting time theorems for random matrices.
\newblock {\em Combinatorics, Probability and Computing}, 23(5):635--669, 2014.

\bibitem{aizenman2003extended}
M.~Aizenman, R.~Sims, and S.~L. Starr.
\newblock Extended variational principle for the {Sherrington-Kirkpatrick}
  spin-glass model.
\newblock {\em Physical Review B}, 68(21):214403, 2003.

\bibitem{aronson1998maximum}
J.~Aronson, A.~Frieze, and B.~G. Pittel.
\newblock Maximum matchings in sparse random graphs: {Karp–Sipser} revisited.
\newblock {\em Random Structures \& Algorithms}, 12(2):111--177, 1998.

\bibitem{ayre2020satisfiability}
P.~Ayre, A.~Coja-Oghlan, P.~Gao, and N.~M{\"u}ller.
\newblock The satisfiability threshold for random linear equations.
\newblock {\em Combinatorica}, 40(2):179--235, 2020.

\bibitem{basak2021sharp}
A.~Basak and M.~Rudelson.
\newblock Sharp transition of the invertibility of the adjacency matrices of
  sparse random graphs.
\newblock {\em Probability theory and related fields}, 180(1):233--308, 2021.

\bibitem{bauer2001exactly}
M.~Bauer and O.~Golinelli.
\newblock Exactly solvable model with two conductor-insulator transitions
  driven by impurities.
\newblock {\em Physical review letters}, 86(12):2621, 2001.

\bibitem{benaych2019largest}
F.~Benaych-Georges, C.~Bordenave, and A.~Knowles.
\newblock Largest eigenvalues of sparse inhomogeneous {Erd\H{o}s–Rényi}
  graphs.
\newblock {\em The Annals of Probability}, 47(3), 2019.

\bibitem{bordenave2011rank}
C.~Bordenave, M.~Lelarge, and J.~Salez.
\newblock {The rank of diluted random graphs}.
\newblock {\em The Annals of Probability}, 39(3):1097 -- 1121, 2011.

\bibitem{coja2022sparse}
A.~Coja-Oghlan, O.~Cooley, M.~Kang, J.~Lee, and J.~B. Ravelomanana.
\newblock The sparse parity matrix.
\newblock In {\em Proceedings of the 2022 Annual ACM-SIAM Symposium on Discrete
  Algorithms (SODA)}, pages 822--833. SIAM, 2022.

\bibitem{coja2022rank}
A.~Coja-Oghlan, A.~A. Erg{\"u}r, P.~Gao, S.~Hetterich, and M.~Rolvien.
\newblock The rank of sparse random matrices.
\newblock {\em Random Structures \& Algorithms}, 2022.

\bibitem{coja2021cut}
A.~Coja-Oghlan and M.~Hahn-Klimroth.
\newblock The cut metric for probability distributions.
\newblock {\em SIAM Journal on Discrete Mathematics}, 35(2):1096--1135, 2021.

\bibitem{coja2023XORSAT}
A.~Coja-Oghlan, M.~Kang, L.~Krieg, and M.~Rolvien.
\newblock The k-{XORSAT} threshold revisited, 2023.

\bibitem{coja2017info}
A.~Coja-Oghlan, F.~Krzakala, W.~Perkins, and L.~Zdeborová.
\newblock Information-theoretic thresholds from the cavity method.
\newblock In {\em Proceedings of the 49th Annual ACM SIGACT Symposium on Theory
  of Computing}, pages 146--157, 2017.

\bibitem{cooper2019rank}
C.~Cooper, A.~Frieze, and W.~Pegden.
\newblock On the rank of a random binary matrix.
\newblock {\em The Electronic Journal of Combinatorics}, 2019.

\bibitem{costello2006random}
K.~P. Costello, T.~Tao, and V.~Vu.
\newblock Random symmetric matrices are almost surely nonsingular.
\newblock {\em Duke Mathematical Journal}, 135(2):395--413, 2006.

\bibitem{costello2010rank}
K.~P. Costello and V.~Vu.
\newblock On the rank of random sparse matrices.
\newblock {\em Combinatorics, Probability and Computing}, 19(3):321--342, 2010.

\bibitem{costello2008rank}
K.~P. Costello and V.~H. Vu.
\newblock The rank of random graphs.
\newblock {\em Random Structures \& Algorithms}, 33(3):269--285, 2008.

\bibitem{demichele2021rank}
P.~DeMichele, M.~Glasgow, and A.~Moreira.
\newblock On the rank, kernel, and core of sparse random graphs.
\newblock {\em arXiv preprint: arXiv:2105.11718}, 2021.

\bibitem{dietzfelbinger2010tight}
M.~Dietzfelbinger, A.~Goerdt, M.~Mitzenmacher, A.~Montanari, R.~Pagh, and
  M.~Rink.
\newblock Tight thresholds for cuckoo hashing via {XORSAT}.
\newblock In {\em Automata, Languages and Programming: 37th International
  Colloquium, ICALP 2010, Bordeaux, France, July 6-10, 2010, Proceedings, Part
  I 37}, pages 213--225. Springer Berlin Heidelberg, 2010.

\bibitem{dubois2002XORSAT}
O.~Dubois and J.~Mandler.
\newblock The 3-{XORSAT} threshold.
\newblock {\em Comptes Rendus Mathematique}, 335(11):963--966, 2002.

\bibitem{erdos2012spectral}
L.~Erd\H{o}s, A.~Knowles, H.-T. Yau, and J.~Yin.
\newblock Spectral statistics of {Erd\H{o}s-Rényi} graphs ii: Eigenvalue
  spacing and the extreme eigenvalues.
\newblock {\em Communications in Mathematical Physics}, 314(3), 2012.

\bibitem{erdos2013spectral}
L.~Erd\H{o}s, A.~Knowles, H.-T. Yau, and J.~Yin.
\newblock Spectral statistics of {Erd\H{o}s–Rényi} graphs i: Local
  semicircle law.
\newblock {\em The Annals of Probability}, 41(38), 2013.

\bibitem{ferber2021singularity}
A.~Ferber, M.~Kwan, A.~Sah, and M.~Sawhney.
\newblock Singularity of the k-core of a random graph, 2021.

\bibitem{freedman1977sampling}
D.~Freedman.
\newblock A remark on the difference between sampling with and without
  replacement.
\newblock {\em Journal of the American Statistical Association}, 72(359), 1977.

\bibitem{van2017random}
R.~{\VAN{Hofstad}{van der}{van der}}~Hofstad.
\newblock {\em Random graphs and complex networks}, volume~43.
\newblock Cambridge university press, 2017.

\bibitem{karp1981maximum}
R.~M. Karp and M.~Sipser.
\newblock Maximum matchings in sparse random graphs.
\newblock In {\em 22nd Annual Symposium on Foundations of Computer Science
  (sfcs 1981)}, pages 364--375. IEEE, 1981.

\bibitem{andrea2008estimating}
A.~Montanari.
\newblock Estimating random variables from random sparse observations.
\newblock {\em European Transactions on Telecommunications}, 19(4):385--403,
  2008.

\bibitem{pittel2016sat}
B.~Pittel and G.~B. Sorkin.
\newblock The satisfiability threshold for k-{XORSAT}.
\newblock {\em Combinatorics, Probability and Computing}, 25(2):236--268, 2016.

\bibitem{wigner1955characteristic}
E.~Wigner.
\newblock Characteristic vectors of bordered matrices with infinite dimensions.
\newblock {\em Annals of Mathematics}, 62(3), 1955.

\end{thebibliography}

\begin{appendices}
\section{Useful properties of the functions $G_t$ and $R_t$}\label{app_proal}
For $t \geq 0$, recall the rank function $R_t:[0,1] \to \RR$,
\[R_t(\alpha)=2-\phi_t\bc{1-\phi_t(\alpha)}-(1+t(1-\alpha))\phi_t(\alpha),\]
defined in (\ref{em0}), as well as $G_t:[0,1] \to \mathbb{R}$, 
\[G_t(\alpha)=\alpha+\phi_t\bc{1-\phi_t(\alpha)}-1,\]
defined in (\ref{fun_g}). The following lemma shows that for all $t \geq 0$, $G_t$ has at least one zero $\alpha_0(t)$, that additionally satisfies the equation $\alpha_0(t) = 1 - \phi_t(\alpha_0(t))$:

\begin{lemma}[See {\cite[Section 3]{coja2022sparse}}] \label{lem_easy_fix}
For any $t \geq 0$, the function $\Xi_t:[0,1] \to \mathbb{R}$, $\Xi_t(\alpha)=\alpha+\phi_t(\alpha)-1$, has a unique zero $\alpha_0(t)$. Moreover,
\begin{align}
    G_t(\alpha_0(t)) = 0.
\end{align}
\end{lemma}
\begin{proof}
We have $\Xi_t'(\alpha) = 1 + t\phi_t(\alpha) >0$ as well as $\Xi_t(0)=\eul^{-t}-1 \leq 0$ and $\Xi_t(1)=1>0$. Thus, $\Xi_t$ has a unique zero $\alpha_0 \in [0,1]$. Moreover,
\begin{align*}
    G_t(\alpha_0(t)) = \alpha_0(t) - 1 + \phi_t(1-\phi_t(\alpha_0(t))) = -\phi_t(\alpha_0(t)) + \phi_t(\alpha_0(t)) = 0.
    \end{align*}
\end{proof}

Recall that $\alpha_\star(t)$ and $\alpha^\star(t)$ denote the smallest and largest zero of $G_t$, respectively. We now prove \Cref{lem_proal}:

\lemproal*

\begin{proof}

\vspace{0.1 cm}
\begin{itemize}
    \item[1.] First observe that $G_0(\alpha)=\alpha$, so $G_0$ is strictly increasing with a unique zero in $\alpha=0$. 

Taking the first and second derivative of ${\alpha\mapsto G_t}$, we have
\[G_t'(\alpha)=1-t^2\phi_t(\alpha)\phi_t\bc{1-\phi_t(\alpha)}\quad\mbox{and}\quad G_t''(\alpha)=t^3\phi_t(\alpha)\phi_t\bc{1-\phi_t\bc{\alpha}}(t\phi_t(\alpha)-1).\]
Since $t>0$, $G_t''(\alpha)<0$ precisely when $\alpha<1-\ln t/t$, and $G_t''(\alpha)>0$ precisely when $\alpha>1-\ln t/t$. Thus, for any $t>0$, the first derivative $\alpha\mapsto G_t'(\alpha)$ is strictly decreasing on $[0,1-\ln t/t)$ and strictly increasing on $(1-\ln t/t,1]$. When $t\in (0,1]$, $1-\ln t/t \geq 1$ and for all $\alpha \in [0,1]$, $G'_t(\alpha)\geq G'_t(1)=1-t^2\eul^{-t}> 0$. When $t \in (1,\eul]$, $1-\ln t/t \in (0,1)$ and for all $\alpha \in [0,1]$,
\begin{align}\label{eq_lower_dG}
G_t'(\alpha)\geq G_t'(1-\ln t/t)=1-t/\eul.
\end{align}
We conclude that for all $t\in [0,\eul)$ and $\alpha \in [0,1]$, $G'_t(\alpha)>0$ and that $G_t$ is strictly increasing. In this case, $G_t$ has at most one zero, which is given by $\alpha_0(t)$ from \Cref{lem_easy_fix}. \\
For $t=\eul$, $G_\eul'$ has exactly one zero in $\alpha = 1-1/\eul$ and is positive otherwise. Thus $G_\eul$ is strictly increasing and has a unique zero, which is given by $\alpha_0(\eul)=1-1/\eul$ from \Cref{lem_easy_fix}.

\item[2. \& 3.] 
Suppose that  $t \in (\eul,\infty)${, so} that $1-\ln t/t \in (0,1)$. We first show that $G_t$ has \textit{at most} three zeroes. The proof of \cref{it_1} shows that $G_t'$ is strictly decreasing on $[0, 1-\ln t/t)$ and strictly increasing on $(1-\ln t/t, 1]$ with $G_t'(1-\ln t/t)=1-t/\eul<0$. Moreover, 
\[G_t'(0)=1-t^2\eul^{-t}\eul^{-t\eul^{-t}}\geq 1-t^2\eul^{-t}=G'_t(1)  >0,\]
where we have used that $\eul^{t/2}= \eul\eul^{t/2-1}>\eul\bc{1+t/2-1}>t$ for $t>\eul$ in the last step. The intermediate value theorem now implies that $G_t'$ has exactly two zeroes in $[0,1]$, which we denote by $\alpha_1(t)<\alpha_2(t)$. By the above, 
\[\alpha_1(t)<1-\ln t/t<\alpha_2(t).\]
This implies for $G_t$ that
\begin{equation}\label{mon_gt}
\begin{aligned}
    \mbox{$G_t$ is strictly increasing on $[0,\alpha_1(t)) \cup (\alpha_2(t),1]$}
    \mbox{ and strictly decreasing on $(\alpha_1(t),\alpha_2(t))$.}
    \end{aligned}
\end{equation}
By the intermediate value theorem and \Cref{mon_gt}, $G_t$ has at most three zeroes in $[0,1]$. 

We next argue that for $t>\eul$, $G_t$ has \textit{exactly three} zeroes. For this, observe that 
\[G_t(0)=\eul^{-t\eul^{-t}}-1<0\quad \mbox{and}\quad G_t(1)=\eul^{-t}>0,\]
\[G_t(1-\ln t/t)=\eul^{-1}-\ln t/t=\eul^{-1}\bc{1-\ln t/\eul^{\ln t-1}}>\eul^{-1}\bc{1-\ln t/\bc{1+\ln t-1}}=0, \]
and
\[G_t(1-1/t)=-1/t+\eul^{-t\eul^{-1}}= -1/t+1/\bc{\eul\eul^{t\eul^{-1}-1}}< -1/t+1/\bc{\eul\bc{1+t\eul^{-1}-1}}=0.\]
By the intermediate value theorem, $G_t$ has at least one zero in each of the intervals $(0,1-\ln t/t)$, $(1-\ln t/t,1-1/t)$ and $(1-1/t,1)$. It follows that  $G_t$ has exactly three zeroes. 

We next show that for $t>\eul$, $\alpha_0(t)$ is neither the largest nor the smallest zero, such that the zeroes $\alpha_\star(t), \alpha^\star(t)$ and $\alpha_0(t)$ are distinct. 
Let $t \geq 0$. If $\dot{\alpha}$ is any zero of $G_t$, then 
$\dot{\alpha}=1-\phi_t(1 - \phi_t(\dot{\alpha}))$. This implies that
\[G_t(1-\phi_t(\dot{\alpha}))= 1-\phi_t(\dot{\alpha})+\phi_t(1-\phi_t(1-\phi_t(\dot{\alpha})))-1=1-\phi_t(\dot{\alpha})+\phi_t(\dot{\alpha})-1=0.\]
Therefore, $1-\phi_t(\dot{\alpha})$ is also a zero of $G_t$, for any $t \geq 0$.

For $t > \eul$, let $\alpha_M(t)$ be the zero of $G_t$ that is contained in the non-empty interval $(1-\ln t/t,1-1/t)$. Since $G_t$ has exactly the three zeros $\alpha_\star(t)< \alpha_M(t)<\alpha^\star(t)$, by the above, $1-\phi_t(\alpha^\star(t))<1-\phi_t(\alpha_M(t))<1-\phi_t(\alpha_\star(t))$ are also three distinct zeroes of $G_t(\alpha)$. Thus we must have that
\begin{align}\label{eq_alm}
   \alpha_\star(t)=1-\phi_t(\alpha^\star(t)),\quad\alpha_M(t)=1-\phi_t(\alpha_M(t))\quad\mbox{and}\quad\alpha^\star(t)=1-\phi_t(\alpha_\star(t)), 
\end{align}
which implies that $\alpha_M(t) = \alpha_0(t)$ (see \Cref{lem_easy_fix}), as well as $\alpha_0(t)\in (1-\ln t/t,1-1/t)$. In particular, the proof of \cref{it_3} is now complete.

Moreover, \cref{eq_alm} proves \cref{it_4} for $t>\eul$. For $t\in [0,\eul]$ on the other hand, \cref{it_4} follows from the fact that $\alpha_\star(t)= \alpha_0(t)=\alpha^\star(t)$ and \Cref{lem_easy_fix}.

\item[4.] As shown in the proof of \cref{it_3}, for $t>\eul$, $\alpha_\star(t) \in (0,1-\ln t/t)$, $\alpha_0(t) \in (1-\ln t/t,1-1/t)$ and $\alpha^\star(t) \in (1-1/t,1)$ with $G_t(0) <0, G_t(1-\ln t/t) > 0, G_t(1-1/t) < 0$ and $G_t(1)>0$. This implies the first part of the claim. Moreover, recall the two zeroes $\alpha_1(t)<\alpha_2(t)$ of $G_t'$ as well as observation \cref{mon_gt}. Since there can be at most one zero in each of the intervals $[0,\alpha_1(t)]$, $[\alpha_1(t),\alpha_2(t)]$ and $[\alpha_2(t),1]$ and $G_t$ has exactly three zeroes, we must have that $\alpha_\star(t)\in [0,\alpha_1(t)]$, $\alpha_0(t)\in [\alpha_1(t),\alpha_2(t)]$ and $\alpha^\star(t)\in [\alpha_2(t),1]$. The second part of \cref{it_5} now follows from \cref{mon_gt}.


\item[5.] Assume that $\bar{\alpha}(t)\in [0,1]$ is a common zero of $G_t$ and $G_t'$. Then $\bar{\alpha}(t)<1$, since $G_t(1)=\eul^{-t}$. Let $b=\phi_t(\bar{\alpha}(t))/(1-\bar{\alpha}(t))>0$. We distinguish the following two cases:
\begin{itemize}
    \item[i)] If $b=1$, then $\bar{\alpha}(t)=1-\phi_t(\bar{\alpha}(t))$. Since $\bar{\alpha}(t)$ is a zero of $G_t'$, $0 = G_t'(\bar{\alpha}(t)) = 1-t^2\phi_t^2(\bar{\alpha}(t))$, i.e. $t>0$ and $\phi_t(\bar{\alpha}(t))=1/t$. By definition of $\phi_t$, $\phi_t(\bar{\alpha}(t))=\eul^{t(\bar{\alpha}(t)-1)}$, so $\bar{\alpha}(t)=1-\ln t/t$. On the other hand, $\bar{\alpha}(t)=1-\phi_t(\bar{\alpha}(t))=1-1/t$. This is only possible for $t=\eul$. In this case, $\bar{\alpha}(\eul)=1-1/\eul = \alpha_0(\eul)$, and indeed, $G_{\eul}(\alpha_0(\eul))=G_{\eul}'(\alpha_0(\eul))=0$.
    \item[ii)] If $b\neq 1$, then 
\[\phi_t(\bar{\alpha}(t))/b=1-\bar{\alpha}(t)=\phi_t(1-\phi_t(\bar{\alpha}(t)))=\eul^{-t\phi_t(\bar{\alpha}(t))}=\bc{\eul^{t(\bar{\alpha}(t)-1)}}^b=\phi_t(\bar{\alpha}(t))^b>0,\]
where we have used $G_t(\bar{\alpha}(t))=0$ in the second step. 
Therefore,
\begin{align}\label{eq_common_zeros}
    \phi_t(\bar{\alpha}(t))=b^{-1/(b-1)}\quad \mbox{and}\quad 1-\bar{\alpha}(t)=b^{-1}\phi_t(\bar{\alpha}(t))=b^{-b/(b-1)}.
    \end{align}
Hence, by definition of $\phi_t$ and (\ref{eq_common_zeros}),
\begin{align}\label{eq_common_zeros2}
t=\bc{\ln \phi_t(\bar{\alpha}(t))}/(\bar{\alpha}(t)-1)=b^{b/(b-1)}\ln b/(b-1).
\end{align}
Since also $G_t'(\bar{\alpha}(t))=0$, we have that 
\[0=1-t^2\phi_t(\bar{\alpha}(t))\phi_t\bc{1-\phi_t(\bar{\alpha}(t))}\stackrel{G_t(\bar{\alpha}(t))=0}{\scalebox{8}[1]{=}}1-t^2\phi_t(\bar{\alpha}(t))(1-\bar{\alpha}(t))=1-b\bc{\ln b}^2/(b-1)^2,\]
where we have used (\ref{eq_common_zeros}) and (\ref{eq_common_zeros2}) in the last step. Thus $(b-1)^2-b\bc{\ln b}^2=0$. Hence $(b-1)b^{-1/2}-\ln b=0$ since $b-1$ and $\ln b $ have the same sign.

Let $l(c)=(c-1)c^{-1/2}-\ln c$ for $c>0$. Then $l(b)=0$ and $l(1)=0$. Taking the derivative of $l$, we have
\[l'(c)=c^{-1/2}/2+c^{-3/2}/2-1/c\geq 0,\]
with equality only if $c=1$. Therefore, $l$ is strictly increasing on $(0,\infty) $. Since $b \not=1$, we conclude that $l(b)\neq 0$, which is a contradiction. 
\end{itemize}

\item[6.] By \cref{it_4},
\begin{align*}
    R_t(\alpha^\star(t))&= 2-\phi_t\bc{1-\phi_t(\alpha^\star(t))}-(1+t(1-\alpha^\star(t)))\phi_t(\alpha^\star(t))\\
    &= 2-\phi_t\bc{\alpha_\star(t)}-(1+t\phi_t\bc{\alpha_\star(t)})\phi_t(\alpha^\star(t))\\
    &= 2-\phi_t\bc{\alpha_\star(t)}-\phi_t\bc{\alpha^\star(t)}-t\phi_t\bc{\alpha_\star(t)}\phi_t\bc{\alpha^\star(t)}.
\end{align*}
Analogously, $R_t(\alpha_\star(t))=2-\phi_t\bc{\alpha_\star(t)}-\phi_t\bc{\alpha^\star(t)}-t\phi_t\bc{\alpha_\star(t)}\phi_t\bc{\alpha^\star(t)}$, so $R_t(\alpha_\star(t))=R_t(\alpha^\star(t))$.

One the other hand,
\[R_t'(\alpha)=t^2\phi_t\bc{\alpha}G_t(\alpha),\]
so for $t>0$, the sign of $R_t'(\alpha)$ is equal to the sign of $G_t(\alpha)$ for $\alpha\in [0,1]$.

\begin{itemize}
    \item[i)] For $t\leq \eul$, as shown under \cref{it_1}, $G_t$ is strictly increasing with a unique zero in $\alpha^\star(t) = \alpha_\star(t)$. Therefore, $R_t$ obtains its unique minimum in $\alpha=\alpha^\star(t) = \alpha_\star(t)$.
    \item[ii)] For $t> \eul$, \cref{it_5} shows that $R_t$ is strictly decreasing on $(0,\alpha_\star(t)) \cup (\alpha_0(t),\alpha^\star(t))$ and strictly increasing on $(\alpha_\star(t),\alpha_0(t)) \cup (\alpha^\star(t),1)$. Therefore, $R_t$ attains its minimum either in $\alpha^\star(t)$ or in $\alpha_\star(t)$. 
\end{itemize}
We conclude that for all $\alpha \notin \{\alpha^\star(t),\alpha_\star(t)\}$, $R_t(\alpha) > R_t(\alpha_\star(t))=R_t(\alpha^\star(t))$.

\item[7.] We apply the implicit function theorem. Consider the two-variable function $\tilde G:\RR_{\geq 0} \times [0,1] \to \RR$, 
\[\tilde G(t, \alpha) = G_t(\alpha) = \alpha - 1 + \eul^{-t \eul^{t(\alpha - 1)}}.\]
$\tilde G$ has continuous partial derivatives and thus is differentiable on $\RR_{>0} \times (0,1)$. By Items 1 and 2, for any $t_0 \geq 0$,
\begin{align*}
    \tilde G(t_0, \alpha_\star(t_0)) = 0, \qquad  \tilde G(t_0, \alpha_0(t_0)) = 0 \qquad \text{ and } \qquad  \tilde G(t_0, \alpha^\star(t_0)) = 0,
\end{align*}
and, by \cref{it_6}, for $t_0\neq \eul$, the partial derivative $\partial_{\alpha}\tilde G$ does not vanish in the respective zeroes:
\begin{align*}
    \partial_\alpha \tilde G(t_0,\alpha^\star(t_0))\not= 0, \qquad \partial_\alpha \tilde G(t_0,\alpha_0(t_0))\not= 0 \qquad \text{and} \qquad \partial_\alpha \tilde G(t_0,\alpha_\star(t_0))\not= 0.
\end{align*}
For $t_0 \in (0,\eul) \cup (\eul,\infty)$, the implicit function theorem provides the existence of continuously differentiable functions $t \mapsto\beta_\star(t)$, $t \mapsto\beta_0(t)$ and $t \mapsto\beta^\star(t)$ defined on an open set $t_0 \in T\subset [0,\infty)$ such that
 \begin{align}\label{eq_allalbe}
 \beta_\star(t_0)=\alpha_\star(t_0),\qquad\beta_0(t_0)=\alpha_0(t_0),\qquad\beta^\star(t_0)=\alpha^\star(t_0),
  \end{align}
  and
 \begin{align}\label{eq_allgb}
\tilde G(t,\beta_\star(t))=\tilde G(t,\beta_0(t))=\tilde G(t,\beta^\star(t)) =0 \qquad \mbox{ for all $t\in T$.}
 \end{align}
Assume now that $t_0\in (0,\eul)$. In this case, \cref{eq_allgb} and \cref{it_1} imply that on $T\cap(0,\eul)$, $\beta_\star, \beta_0$ and $\beta^\star$ are identical to $\alpha_0$, since for any $t\in (0,\eul)$, $G_t$ has exactly one zero $\alpha_0(t)$. 
Therefore, for any $t_0 \in (0,\eul)$, the function $t \mapsto\alpha_0(t)$ is continuously differentiable in $t_0$.
 
 Let now $t_0\in (\eul,\infty)$. Since in this case, the three zeroes of $G_t$ are distinct by Item 2, \cref{eq_allalbe} implies that 
 \[\beta_\star(t_0)<\beta_0(t_0)<\beta^\star(t_0).\]
 Since $\beta^\star(t),\beta_0(t)$ and $\beta_\star(t)$ are continuous functions, we can further restrict $T$ such that $T\subset (\eul,\infty)$ and
$\beta_\star(t)<\beta_0(t)<\beta^\star(t)$ for all $t\in T$. Then the combination of \cref{eq_allgb} and \cref{it_3} gives that on $T$,
\[\alpha^\star(t)=\beta^\star(t),\qquad\alpha_0(t)=\beta_0(t)\qquad\mbox{and}\qquad\alpha_\star(t)=\beta_\star(t).\]
Therefore, for any $t_0 \in (\eul, \infty)$, the functions  $t \mapsto\alpha_\star(t)$, $t \mapsto\alpha_0(t)$ and $t \mapsto\alpha^\star(t)$ are continuously differentiable in $t_0$.

We finally consider continuity of $t \mapsto \alpha_\star(t)$ in the point $t=\eul$. Let $a:=\limsup_{t \to \eul}\alpha_\star(t)\in [0,1]$ and suppose that $ a \not= \alpha_\star(\eul)$. Since $\alpha_\star(\eul)$ is the only zero of $G_\eul$, $\tilde G(\eul, a) \not=0$. As $\tilde G$ is a continuous function, there exists $\delta>0$ such that for all $(t,\alpha) \in U_\delta:= [\eul-\delta, \eul+\delta] \times [a+\delta, a-\delta]$, $|\tilde G(t,\alpha) - \tilde G(\eul,a)| \leq |\tilde G(\eul,a)|$. In particular, $\tilde G(t,\alpha) \not= 0$ for all $(t,\alpha) \in U_\delta$. On the other hand, by definition of $a$, there exists $t_\delta \in  [\eul-\delta, \eul+\delta] \setminus \{\eul\}$ with $\alpha_\star(t_\delta) > a-\delta$, such that $(t_\delta,\alpha_\star(t_\delta)) \in U_\delta$. But $\tilde G(t_\delta,\alpha_\star(t_\delta))=0$, which gives the desired contradiction. We conclude that $\limsup_{t \to \eul}\alpha_\star(t) = \alpha_\star(\eul)$.

Analogously, it can be shown that $\liminf_{t \to \eul}\alpha_\star(t) = \alpha_\star(\eul)$ and therefore, $t\mapsto \alpha_\star(t)$ is continuous in $t=\eul$. 
Similarly, one can show that $t\mapsto \alpha_\star(t)$ is continuous in $t=0$. The results for $\alpha_0(t)$ and $\alpha^\star(t)$ follow along the same lines.

\item[8.] 
As in the proof of \cref{it_7}, we use $\tilde G$ to denote the two-variable function $(t,\alpha) \mapsto G_t(\alpha)$. 
For $\varepsilon\geq 0$, let \[U_\varepsilon=\{(t,\alpha)\in[0,d]\times [0,1]:|\tilde G(t,\alpha)|\leq\varepsilon\}\subseteq \RR^2,\]
such that
\[U_0=\{(t,\alpha)\in[0,d]\times [0,1]:\alpha \in \{\alpha_\star(t),\alpha_0(t),\alpha^\star(t)\}\}.\]

For $b\in \RR^2$ and $A\subset\RR^2$, let $\rd\bc{b,A}:=\inf_{a\in A}\|a-b\|_2$. We next argue that  
\begin{align}\label{duep}
    \lim_{\varepsilon\to 0}\sup_{x\in U_\varepsilon} \rd\bc{x,U_0}=: \lim_{\varepsilon\to 0} \Delta_\varepsilon  = 0.
\end{align}
Indeed, suppose that \cref{duep} does not hold. Then there exist $\delta>0$, $\varepsilon_n\downarrow 0$ and $x_n\in U_{\varepsilon_n}$ such that for all $n \geq 1$,
\[\rd\bc{x_n,U_0}\geq\delta.\]
As a uniformly bounded sequence, $\bc{x_n}_{n\geq 1}$ has a convergent subsequence $\bc{x_{n_k}}_{k\geq 1}$ with limit $x^\ast$, and 
\[\rd\bc{x^\ast,U_0}=\lim_{k\to\infty}\rd\bc{x_{n_k},U_0}\geq \delta.\]
However, since $\tilde G$ is continuous, $|\tilde G(x^\ast)|=\lim_{k\to\infty}|\tilde G\bc{x_{n_k}}|\leq \lim_{k\to\infty}\varepsilon_{n_k}=0$, i.e., $\rd\bc{x^\ast,U_0}=0$, which is a contradiction. Therefore, \cref{duep} holds. 

Recall that we assume that $G_t(\bm{b}_{n,P,N,J_N,t})= \oone$ for a family of random variables $(\bm{b}_{n,P,N,J_N,t})_{n, P,N\in \ZZ^+,J_N \in \syN,t\in [0,d]}\subseteq [0,1]$. This is equivalent to
\[ \liminf_{P\to\infty}\liminf_{n\to\infty}\inf_{\substack{N\geq n, J_N\in \syN}}\inf_{t\in [0,d]}\PP\bc{\abs{G_t(\bm{b}_{n,P,N,J_N,t})}\leq \varepsilon}=1 \qquad \text{for all }\varepsilon >0,\]
since for any $\varepsilon>0$ and $B:=\sup_{n,P,N\in \NN, J_N \in \syN, t\in [0,d]}|G_t(\bm{b}_{n,P,N,J_N,t})|<\infty$, $\varepsilon\PP\bc{\abs{G_t(\bm{b}_{n,P,N,J_N,t})}> \varepsilon}\leq \mathbb{E}\abs{G_t(\bm{b}_{n,P,N,J_N,t})}\leq  \varepsilon+B\PP\bc{\abs{G_t(\bm{b}_{n,P,N,J_N,t})}> \varepsilon}.$

We conclude that for any $\varepsilon >0$,
\[ \liminf_{P\to\infty}\liminf_{n\to\infty}\inf_{\substack{N\geq n,J_N\in \syN}}\inf_{t\in [0,d]}\PP\bc{(t,\bm{b}_{n,P,N,\syN,t}) \in U_\varepsilon} = 1. \]
The definition of $\Delta_{\varepsilon}$ then implies that for all $\varepsilon >0,$
\begin{align}\label{ineq_disdel}
    \liminf_{P\to\infty}\liminf_{n\to\infty}\inf_{\substack{N\geq n,J_N\in \syN}}\inf_{t\in [0,d]}\PP\bc{\rd\bc{(t,\bm{b}_{n,P,N,\syN,t}),U_0}\leq \Delta_\varepsilon}=1.
\end{align}
On the event $\{\rd\bc{(t,\bm{b}_{n,P,N,J_N,t}),U_0}\leq \Delta_\varepsilon\}$, there exists $(\tilde{t},\tilde{\alpha})\in U_0$ with $|t-\tilde t|\leq \Delta_\eps$ and $|\bm{b}_{n,P,N,J_N,t}-\tilde\alpha|\leq \Delta_\eps$. We next argue that for $\eps$ chosen small enough and thus $\tilde t$ close to $t$, also the value of $\tilde \alpha \in \{\alpha_\star(\tilde t), \alpha_0(\tilde t), \alpha^\star(\tilde t)\}$ is close to one of $\alpha_\star(t), \alpha_0(t)$ or $\alpha^\star(t)$: Observe that since the functions $s \mapsto\alpha_\star(s)$, $s\mapsto \alpha_0(s)$ and $s \mapsto \alpha^\star(s)$ are uniformly continuous on $[0,d]$ by \cref{it_7}, for any $\upsilon >0$, there exists $\omega>0$ such that  for any $t_1,t_2\in [0,d]$ with $\abs{t_1-t_2}\leq \omega$, 
\begin{align}\label{eq_up_al}
    \max\cbc{\abs{\alpha_0(t_1)-\alpha_0(t_2)},\abs{\alpha_\star(t_1)-\alpha_\star(t_2)},\abs{\alpha^\star(t_1)-\alpha^\star(t_2)}}\leq \upsilon/2.
\end{align}
As $\lim_{\varepsilon\downarrow 0}\Delta_\varepsilon=0$, we can choose $\varepsilon$ such that $\Delta_\varepsilon<\min\cbc{\upsilon/2,\omega}$, 
such that in particular
 \begin{align}\label{del_eps_l}
     \abs{t-\tilde{t}}\leq \omega\quad\mbox{and}\quad \abs{\bm{b}_{n,P,N,J_N,t}-\tilde{\alpha}}\leq \upsilon/2.
 \end{align}
Then by equations (\ref{eq_up_al}) and (\ref{del_eps_l}), 
\begin{align*}
    \max\cbc{\abs{\alpha_0(t)-\alpha_0(\tilde{t})},\abs{\alpha_\star(t)-\alpha_\star(\tilde{t})},\abs{\alpha^\star(t)-\alpha^\star(\tilde{t})}}\leq \upsilon/2.
\end{align*}
Combining the {above inequality} with \cref{del_eps_l}, we conclude that
\[\bm{M}_{n,P,N,J_N,t} := \min\cbc{\abs{\bm{b}_{n,P,N,J_N,t}-\alpha_0(t)},\abs{\bm{b}_{n,P,N,J_N,t}-\alpha_\star(t)},\abs{\bm{b}_{n,P,N,J_N,t}-\alpha^\star(t)}}\leq \upsilon.\]
Hence, for any $\upsilon>0$ and $\varepsilon$ sufficiently small, by \cref{ineq_disdel}, 
\begin{align*}
   & \liminf_{P\to\infty}\liminf_{n\to\infty}\inf_{\substack{N\geq n,\\J_N\in \syN}}\inf_{t\in [0,d]}\PP\bc{\bm{M}_{n,P,N,J_N,t} \leq \upsilon} \\
     \geq\ & \liminf_{P\to\infty}\liminf_{n\to\infty}\inf_{\substack{N\geq n,\\J_N\in \syN}}\inf_{t\in [0,d]}\PP\bc{\rd\bc{(t,\bm{b}_{n,P,N,J_N,t}),U_0}\leq \Delta_\varepsilon}=1,
\end{align*}
which, by the above equivalent characterisation of $\oone$-convergence, gives the claim: $\bm{M}_{n,P,N,J_N,t}  = \oone$. 
\end{itemize}
\end{proof}

\section{Upper bound via leaf-removal: Derivation of \Cref{t_upper}}\label{app_leaf}

In this section, we briefly explain how to derive \Cref{t_upper} from known results about the Karp-Sipser core of sparse Erd\H{o}s-Rényi random graphs \cite{ aronson1998maximum,karp1981maximum}. To obtain the Karp-Sipser core of a graph $G$, we iteratively remove vertices of degree one and their unique neighbors from $G$, until only isolated vertices and a subgraph of minimum degree at least two remain. 
We call the number of isolated vertices in the reduced graph $I_{\text{KS}}(G)$. The derivation of \Cref{t_upper} rests on the following result:

\begin{theorem}[\cite{ aronson1998maximum,karp1981maximum}]\label{thm_isolated}
    For any $d>0$,
    \begin{align*}
        \frac{I_{\text{KS}}(\bm{G}_{n,d/n})}{n} \stackrel{\PP}{\longrightarrow} \frac{\gamma^\star+\gamma_\star+\gamma^\star\gamma_\star}{d} - 1, \qquad n \to \infty.
    \end{align*}
Here, $\gamma_\star$ is the smallest root of the equation $x=d\exp(-d\exp(-x))$ and $\gamma^\star= d\exp(-\gamma_\star)$.
\end{theorem}

\Cref{thm_isolated} is of importance in our setting since crucially, removal of degree-one vertices and their neighbours does not change the nullity of the corresponding adjacency matrix (see \cite{bauer2001exactly}). Therefore, irrespective of the field or the matrix entries,
\begin{align*}
    \frac{\rk{\bm{A}_{n,d/n}}}{n} = 1 -  \frac{\text{nul}_{\FF}(\bm{A}_{n,d/n})}{n} \leq  1 -  \frac{I_{\text{KS}}(\bm{G}_{n,d/n})}{n} \qquad \text{a.s.}
\end{align*}
and \Cref{thm_isolated} thus implies that for any $d,\varepsilon>0$ and any field $\FF$,
\begin{equation*}
\lim_{n\to \infty}\PP\bc{\sup_{J_n\in\syn} \frac{\rk{\bm{A}_{n,d/n}}}{n}\leq 2 -\frac{\gamma^\star+\gamma_\star+\gamma^\star\gamma_\star}{d}  +\varepsilon}=1.
    \end{equation*}
Thus, it remains to relate the limit from \Cref{thm_isolated} to the rank function $R_d$.

For this, observe that any root $x^\star$ of the equation $x=d\exp(-d\exp(-x))$ satisfies $x^\star \in (0,d)$ as well as $G_d(1 - x^\star/d) = 0$.  \Cref{it_1,it_3} in \Cref{lem_proal} thus imply that
\begin{align*}
    \gamma_\star = d(1-\alpha^\star) \qquad \text{and} \qquad  \gamma^\star = d(1-\alpha_\star). 
\end{align*}
Moreover, \cref{it_4} in \Cref{lem_proal} as well as $G_d(\alpha^\star)=0$ give that
\begin{align*}
2 -\frac{\gamma^\star+\gamma_\star+\gamma^\star\gamma_\star}{d} &= 2 -  \frac{d(1-\alpha_\star)+d(1-\alpha^\star)+d^2(1-\alpha_\star)(1-\alpha^\star)}{d} \\
&= 2 - \phi_d(1-\phi_d(\alpha_\star)) - \phi_d(\alpha_\star) - d\phi_d(\alpha_\star)(1-\alpha_\star) = R_d(\alpha_\star) = \min_{\alpha \in [0,1]}R_d(\alpha).
\end{align*}
Here, in the last step, we have used \Cref{lem_proal}, part 6. This concludes the derivation of \Cref{t_upper} from \Cref{thm_isolated}.

\section{Difference approximation via conditional expectations: proof of \Cref{apexp}}\label{app_b}
For a differentiable function $f:\mathbb{R}^k\mapsto\mathbb{R}$, let $\nabla f$ be the gradient of $f$. We prove the following more general version of \Cref{apexp}: 
\begin{proposition}\label{pexp}
Fix a dimension $k \in \NN$ and $K>1$. Let $\bm{Z}_1,\bm{Z}_2$ and $\bm{X}$ be defined on the same probability space with convex codomains
$\mathcal{R}_{\bm{Z}_1},\mathcal{R}_{\bm{Z}_2}\subset \RR^k$ and $\mathcal{R}_{\bm{X}} \subset \RR$ respectively, such that $\mathcal{R}_{\bm{X}}$ is bounded. Then for any differentiable functions $f,g:\mathbb{R}^k\to\mathbb{R}$,
{\small\begin{equation}\label{eh0}
\begin{aligned}
&\mathbb{E}\abs{f(\bm{Z}_2)-g(\bm{Z}_2)}\\
\leq& \bc{\sup_{\zeta\in\mathcal{R}_{\bm{Z}_1}}\abs{f(\zeta)}+\sup_{x\in \mathcal{R}_{\bm{X}}}\abs{x}}\bc{4K^2\mathbb{E}\|\bm{Z}_1-\bm{Z}_2\|_\infty+2-2(1-1/K)^k}
+k\sup_{\zeta\in\mathcal{R}_{\bm{Z}_2}}\dabs{\nabla f(\zeta)}_\infty\mathbb{E}\|\bm{Z}_1-\bm{Z}_2\|_\infty\\
&+\mathbb{E}\abs{f(\bm{Z}_1)-\mathbb{E}\brk{\bm{X}|\bm{Z}_1}}
+\mathbb{E}\abs{\mathbb{E}\brk{\bm{X}|\bm{Z}_2}-g(\bm{Z}_2)}+\frac{2k}{K}\bc{\sup_{\zeta\in\mathcal{R}_{\bm{Z}_1}}\dabs{\nabla f(\zeta)}_\infty+\sup_{\zeta\in\mathcal{R}_{\bm{Z}_2}}\dabs{\nabla g(\zeta)}_\infty}
\end{aligned}
\end{equation}}
and
{\small\begin{equation}\label{eh0_new}
\begin{aligned}
&\mathbb{E}\brk{\bc{f(\bm{Z}_2)-g(\bm{Z}_2)}^{-}}\\
\leq& \bc{\sup_{\zeta\in\mathcal{R}_{\bm{Z}_1}}\abs{f(\zeta)}+\sup_{x\in \mathcal{R}_{\bm{X}}}\abs{x}}\bc{4K^2\mathbb{E}\|\bm{Z}_1-\bm{Z}_2\|_\infty+2-2(1-1/K)^k}
+k\sup_{\zeta\in\mathcal{R}_{\bm{Z}_2}}\dabs{\nabla f(\zeta)}_\infty\mathbb{E}\|\bm{Z}_1-\bm{Z}_2\|_\infty\\
&+\mathbb{E}\brk{\bc{f(\bm{Z}_1)-\mathbb{E}\brk{\bm{X}|\bm{Z}_1}}^-}
+\mathbb{E}\brk{\bc{\mathbb{E}\brk{\bm{X}|\bm{Z}_2}-g(\bm{Z}_2)}^-}+\frac{2k}{K}\bc{\sup_{\zeta\in\mathcal{R}_{\bm{Z}_1}}\dabs{\nabla f(\zeta)}_\infty+\sup_{\zeta\in\mathcal{R}_{\bm{Z}_2}}\dabs{\nabla g(\zeta)}_\infty}.
\end{aligned}
\end{equation}}
\end{proposition}

The philosophy behind \Cref{pexp} is that for sufficiently nice functions $f$ and $g$, good control over  $\mathbb{E}\|\bm{Z}_1-\bm{Z}_2\|_\infty$, $\mathbb{E}\abs{\mathbb{E}\brk{\bm{X}|\bm{Z}_1}-f(\bm{Z}_1)}$ and $\mathbb{E}\abs{\mathbb{E}\brk{\bm{X}|\bm{Z}_2}-g(\bm{Z}_2)}$ 
allows to bound the difference of $f(\bm{Z}_2)$ and $g(\bm{Z}_2)$ in expectation.
Indeed, \Cref{pexp} is designed to deal with situations where $\bm{Z}_1$ and $\bm{Z}_2$ are close, and it might be helpful to keep this in mind during the following proofs.

To prove \Cref{pexp}, we partition $\mathbb{R}^k$ into small hypercubes: For $K>0$, let $\ZZ^k/K=\cbc{q\in \RR^k:Kq\in \ZZ^k}$.
We define the $k$-dimensional, half-open hypercube with side-length $r>0$ and center $s=(s_1,s_2,\ldots,s_k) \in \RR^k$ as
\[D_s(r)=\left\{(t_1,t_2,\ldots,t_k): t_i-s_i\in [-r/2,r/2),i=1,2,\ldots,k\right\}.\]
The following lemma shows that if $\bm{Z}_1$ and $\bm{Z}_2$ are close, they are likely to be found within the same box, after the application of a random uniform translation. This random translation ensures that the rather arbitrary random variables $\bm{Z}_1, \bm{Z}_2$ do not always take values in the boundary of the partitioning hypercubes.

\begin{lemma}\label{ls1}
Fix a dimension $k \in \NN$ and hypercube edge length $1/K>0$. Then for any two random vectors $\bm{Z}_1,\bm{Z}_2 \in \RR^k$ that are defined on the same probability space, an independently and uniformly chosen ``shift'' vector $\bm{\xi}_K \in (0,1/K]^k$ and $0<\varepsilon<1/K$,
\begin{equation}\label{els}
\small \sum_{q\in \ZZ^k/K}\mathbb{E}\brk{\abs{\ensuremath{\mathds{1}}{\left\{\bm{Z}_1 - \bm{\xi}_K \in  D_{q}\bc{1/K}\right\}}-\ensuremath{\mathds{1}}{\left\{\bm{Z}_2- \bm{\xi}_K\in  D_{q}\bc{1/K}\right\}}}}\leq \frac{4}{\varepsilon}\mathbb{E}\|\bm{Z}_1-\bm{Z}_2\|_\infty+2-2(1-K\varepsilon)^k.
\end{equation}
\end{lemma}

\begin{proof}
For the sake of brevity, we omit the range of summation from $\sum_{q\in \ZZ^k/K}$ throughout this proof.

Fix any hypercube $D_{q}\bc{1/K}$ and let $j \in \{1,2\}$. For $\bm{Z}_j- \bm{\xi}_K$ to fall into $D_{q}\bc{1/K}$ and $\bm{Z}_{3-j}- \bm{\xi}_K$ to fall into a distinct box, one of the following two cases must happen: 
\begin{itemize}
 \item[(a)] $\bm{Z}_j- \bm{\xi}_K$ is in the ``inner part'' $D_{q}\bc{1/K-\varepsilon}$ of the box, but $\bm{Z}_{3-j} - \bm{\xi}_K \notin D_{q}\bc{1/K}$ (``\textbf{separation}''), or
  \item[(b)] $\bm{Z}_{j}- \bm{\xi}_K$ is in the ``$\varepsilon$-boundary''  $D_{q}\bc{1/K} \setminus D_{q}\bc{1/K-\varepsilon}$ of the box (``\textbf{boundary}'').
\end{itemize}
We call the separation event $\mathfrak{S}_q^{(j)}$, and the boundary event $\mathfrak{B}_q^{(j)}$, 
which yields the almost sure upper bound
\begin{equation}\label{eq_ls1_0}
\small\begin{aligned}
&\abs{\ensuremath{\mathds{1}}{\left\{\bm{Z}_1- \bm{\xi}_K\in  D_{q}\bc{1/K}\right\}}-\ensuremath{\mathds{1}}{\left\{\bm{Z}_2- \bm{\xi}_K\in  D_{q}\bc{1/K}\right\}}} \leq \ind \mathfrak{S}_q^{(1)} + \ind\mathfrak{B}_q^{(1)} + \ind\mathfrak{S}_q^{(2)} + \ind\mathfrak{B}_q^{(2)}.
\end{aligned}
\end{equation}
It thus remains to upper bound the right hand side of (\ref{eq_ls1_0}) in expectation and then sum over $q \in \mathbb{Z}^k/K$.

\textbf{Separation}: 
Deterministically, for $j \in \{1,2\}$,
{\small\begin{equation}\label{eq_ls1_1}
    \begin{aligned} 
\ind\mathfrak{S}_q^{(j)}
\leq \ensuremath{\mathds{1}}{\left\{\bm{Z}_j- \bm{\xi}_K\in  D_{q}\bc{1/K-\varepsilon}, \|\bm{Z}_2-\bm{Z}_1\|_\infty\geq \varepsilon/2\right\}}
\leq \ensuremath{\mathds{1}}{\left\{\bm{Z}_j- \bm{\xi}_K \in  D_{q}\bc{1/K-\varepsilon}\right\}} \frac{2}{\varepsilon}\|\bm{Z}_2-\bm{Z}_1\|_\infty.
\end{aligned}
\end{equation}}

Summing over $q \in \ZZ^k/K$ in (\ref{eq_ls1_1}) and taking expectation gives
\begin{equation}\label{eq_b1}
    \begin{aligned}
\mathbb{E}\brk{\sum \ind \mathfrak{S}_q^{(j)}} \leq \frac{2}{\varepsilon}\Erw\brk{\|\bm{Z}_2-\bm{Z}_1\|_\infty}.
\end{aligned}
\end{equation}


\textbf{Boundary}: 
This is the case where the benefit of the random translation $\bm{\xi}_K$ becomes apparent. Again, let $j \in \{1,2\}$ and fix $q \in \ZZ^k/K$. Conditionally on $\bm{Z}_j$, the random variable $\bm{Z}_j-\bm{\xi}_K-q$ is uniformly distributed over the box $\prod_{i=1}^k[\bc{\bm{Z}_j}_i-q_i-1/K,\bc{\bm{Z}_j}_i-q_i)$. Therefore,
\begin{align}\label{eq_lebe}
   \PP\bc{\mathfrak B_q^{(j)}\big \vert \bm{Z}_j} 
   = K^k \lambda\bc{\prod_{i=1}^k[\bc{\bm{Z}_j}_i-q_i-1/K,\bc{\bm{Z}_j}_i-q_i) \cap \bc{D_{0}\bc{1/K} \setminus D_{0}\bc{1/K-\varepsilon}}},
\end{align}
where $\lambda$ denotes the $k$-dimensional Lebesgue measure.

Now, since also the boxes $\prod_{i=1}^k[\bc{\bm{Z}_j}_i-q_i-1/K,\bc{\bm{Z}_j}_i-q_i)$, $q \in \ZZ^k/K$, partition $\RR^k$, (\ref{eq_lebe}) further yields that
\begin{equation}\label{eq_b3}
\begin{aligned}
\Erw\brk{\sum \ind\mathfrak B_q^{(j)}} = \Erw\brk{\sum\PP\bc{\mathfrak{B}_q^{(j)} \big \vert \bm{Z}_j}} 
= K^k \lambda\bc{D_{0}\bc{1/K} \setminus D_{0}\bc{1/K-\varepsilon}}
=1-(1-K\varepsilon)^k.
\end{aligned}
\end{equation}


The claim now follows from summing (\ref{eq_ls1_0}) over $q \in \ZZ^k/K$,  \cref{eq_b1} and \cref{eq_b3}.
\end{proof}

We next turn to the proof of \Cref{pexp}. 

\begin{proof}[Proof of \Cref{pexp}]
Again, for brevity, we omit the range of summation from $\sum_{q\in \ZZ^k/K}$ throughout the proof. As in \Cref{ls1}, let $\bm{\xi}_K \in (0,1/K]^k$ be a uniformly chosen ``shift'' vector that is independent of $(\bm{Z}_1, \bm{Z}_2, \bm{X})$. We first distinguish the possible hypercube-locations for $\bm{Z}_2-\bm{\xi}_K$ and apply the tower property to get
\begin{equation}\label{eq_n_pb1}\mathbb{E}\abs{f(\bm{Z}_2)-g(\bm{Z}_2)}=\sum\Erw\brk{\Erw\brk{\abs{f(\bm{Z}_2)-g(\bm{Z}_2)}\ensuremath{\mathds{1}}{\left\{\bm{Z}_2-\bm{\xi}_K \in  D_{q}\bc{1/K}\right\}}|\bm{\xi}_K}}.
\end{equation}

Given $\bm{\xi}_K$, on the event $\{\bm{Z}_2-\bm{\xi}_K \in D_{q}\bc{1/K}\}$, $\bm{Z}_2$ is located in the hypercube $D_{q+\bm{\xi}_K}\bc{1/K}$ of sidelength $1/K$.
Since the hypercubes are small and $f,g$ are continuous, the values of $f$ and $g$ should not fluctuate too much on $D_{q+\bm{\xi}_K}\bc{1/K}$.
More precisely, let $t \in D_{q+\bm{\xi}_K}\bc{1/K}\cap \mathcal{R}_{\bm{Z}_2}$ be arbitrary. If $D_{q+\bm{\xi}_K}\bc{1/K}\cap \mathcal{R}_{\bm{Z}_2} = \emptyset$, let $t=0$. Then by the mean value theorem,
{\[\mathbb{E}\brk{\abs{f(\bm{Z}_2)-f(t)} \ensuremath{\mathds{1}}{\left\{\bm{Z}_2-\bm{\xi}_K \in  D_{q}\bc{1/K}\right\}}|\bm{\xi}_K}\leq \frac{k}{K}\sup_{\zeta\in \mathcal{R}_{\bm{Z}_2}}\dabs{\nabla f(\zeta)}_\infty\mathbb{P}\bc{\bm{Z}_2-\bm{\xi}_K \in  D_{q}\bc{1/K}|\bm{\xi}_K},\]}
and
{\[\mathbb{E}\brk{\abs{g(\bm{Z}_2)-g(t)} \ensuremath{\mathds{1}}{\left\{\bm{Z}_2-\bm{\xi}_K \in  D_{q}\bc{1/K}\right\}}|\bm{\xi}_K}\leq \frac{k}{K}\sup_{\zeta\in \mathcal{R}_{\bm{Z}_2}}\dabs{\nabla g(\zeta)}_\infty\mathbb{P}\bc{\bm{Z}_2-\bm{\xi}_K \in  D_{q}\bc{1/K}|\bm{\xi}_K}.\]}
In the last two displays, both sides are zero if $D_{q+\bm{\xi}_K}\bc{1/K}\cap \mathcal{R}_{\bm{Z}_2} = \emptyset$. By the triangle inequality, we get
\begin{align}
&\mathbb{E}\brk{\abs{f(\bm{Z}_2)-g(\bm{Z}_2)}\ensuremath{\mathds{1}}{\left\{\bm{Z}_2-\bm{\xi}_K \in  D_{q}\bc{1/K}\right\}}|\bm{\xi}_K} \label{eq_pb8_0} \\
\leq& \bc{ \abs{f(t)-g(t)} +\frac{k}{K}\sup_{\zeta\in \mathcal{R}_{\bm{Z}_2}}\bc{\dabs{\nabla f(\zeta)}_\infty+\dabs{\nabla g(\zeta)}_\infty}}\mathbb{P}\bc{\bm{Z}_2-\bm{\xi}_K \in  D_{q}\bc{1/K}|\bm{\xi}_K} \label{eq_pb8_new} \\
\leq
&\abs{\mathbb{E}\brk{\bc{f(\bm{Z}_2)-g(\bm{Z}_2)}\ensuremath{\mathds{1}}{\left\{\bm{Z}_2-\bm{\xi}_K \in  D_{q}\bc{1/K}\right\}}|\bm{\xi}_K}} \label{eq_pb8} \\
& \quad
+\frac{2k}{K}\sup_{\zeta\in \mathcal{R}_{\bm{Z}_2}}\bc{\dabs{\nabla f(\zeta)}_\infty+\dabs{\nabla g(\zeta)}_\infty}\mathbb{P}\bc{\bm{Z}_2-\bm{\xi}_K \in  D_{q}\bc{1/K}|\bm{\xi}_K}, \label{eq_pb8_1}
\end{align}
where now the modulus is outside of the expectation in \Cref{eq_pb8} in comparison to \Cref{eq_pb8_0}.
Summing \Cref{eq_pb8} over $q \in \ZZ^k/K$ and applying the triangle inequality together yield that
\begin{align}
&\sum \abs{\mathbb{E}\brk{\bc{f(\bm{Z}_2)-g(\bm{Z}_2)}\ensuremath{\mathds{1}}{\left\{\bm{Z}_2-\bm{\xi}_K \in  D_{q}\bc{1/K}\right\}}|\bm{\xi}_K}} \label{eq_n_pb2_new} \\
\leq&\sum \mathbb{E}\brk{\abs{f(\bm{Z}_2)-f(\bm{Z}_1)}\ensuremath{\mathds{1}}{\left\{\bm{Z}_2-\bm{\xi}_K \in  D_{q}\bc{1/K}\right\}}|\bm{\xi}_K}\label{eq_n_pb2_1} \\
&+\sum\mathbb{E}\brk{\abs{f(\bm{Z}_1)}\abs{\ensuremath{\mathds{1}}{\left\{\bm{Z}_2-\bm{\xi}_K \in  D_{q}\bc{1/K}\right\}}-\ensuremath{\mathds{1}}{\left\{\bm{Z}_1-\bm{\xi}_K \in  D_{q}\bc{1/K}\right\}}}|\bm{\xi}_K} \label{eq_n_pb2_2} \\
&+\sum\abs{\Erw\brk{f(\bm{Z}_1)\ensuremath{\mathds{1}}{\left\{\bm{Z}_1-\bm{\xi}_K \in  D_{q}\bc{1/K}\right\}}-g(\bm{Z}_2)\ensuremath{\mathds{1}}{\left\{\bm{Z}_2-\bm{\xi}_K \in  D_{q}\bc{1/K}\right\}}|\bm{\xi}_K}}.\label{eq_n_pb2_3}
\end{align}
For \cref{eq_n_pb2_1}, since $\sum\ensuremath{\mathds{1}}{\left\{\bm{Z}_2-\bm{\xi}_K \in  D_{q}\bc{1/K}\right\}}=1$, again the mean value theorem implies that
{\small \begin{equation}\label{eq_pb6}
\begin{aligned}
\sum &\mathbb{E}\brk{\abs{f(\bm{Z}_2)-f(\bm{Z}_1)}\ensuremath{\mathds{1}}{\left\{\bm{Z}_2-\bm{\xi}_K \in  D_{q}\bc{1/K}\right\}}|\bm{\xi}_K}
\leq \mathbb{E}\abs{f(\bm{Z}_2)-f(\bm{Z}_1)}\leq k \sup_{\zeta \in\mathcal{R}_{\bm{Z}_2} }\dabs{\nabla f(\zeta)}_\infty \mathbb{E}\|\bm{Z}_1-\bm{Z}_2\|_\infty.
\end{aligned}
\end{equation}}
Taking expectation in \Cref{eq_n_pb2_2}, then an application of \Cref{ls1} gives that
\begin{equation}\label{eq_pb2}
\begin{aligned}
&\sum\mathbb{E}\brk{\abs{f(\bm{Z}_1)}\abs{\ensuremath{\mathds{1}}{\left\{\bm{Z}_2-\bm{\xi}_K \in  D_{q}\bc{1/K}\right\}}-\ensuremath{\mathds{1}}{\left\{\bm{Z}_1-\bm{\xi}_K \in  D_{q}\bc{1/K}\right\}}}}\\
&\leq \sup_{\zeta\in \mathcal{R}_{\bm{Z}_1}}|f(\zeta)|\bc{\frac{4}{\varepsilon}\mathbb{E}\|\bm{Z}_1-\bm{Z}_2\|_\infty+2-2(1-K\varepsilon)^k}.
\end{aligned}
\end{equation}
Finally, using the triangle inequality once more, \cref{eq_n_pb2_3} can again be divided into three sub-parts as follows:
\begin{align}
&\sum\abs{\Erw\brk{f(\bm{Z}_1)\ensuremath{\mathds{1}}{\left\{\bm{Z}_1-\bm{\xi}_K \in  D_{q}\bc{1/K}\right\}}-g(\bm{Z}_2)\ensuremath{\mathds{1}}{\left\{\bm{Z}_2-\bm{\xi}_K \in  D_{q}\bc{1/K}\right\}}|\bm{\xi}_K}}\label{eq_n_pb3_new} \\
\leq& \sum\mathbb{E}\brk{\abs{f(\bm{Z}_1)-\mathbb{E}\brk{\bm{X}|\bm{Z}_1}}\ensuremath{\mathds{1}}{\left\{\bm{Z}_1-\bm{\xi}_K \in  D_{q}\bc{1/K}\right\}}|\bm{\xi}_K}\label{eq_n_pb3_1} \\
&+\sum\mathbb{E}\brk{\abs{\mathbb{E}\brk{\bm{X}|\bm{Z}_2}-g(\bm{Z}_2)}\ensuremath{\mathds{1}}{\left\{\bm{Z}_2-\bm{\xi}_K \in  D_{q}\bc{1/K}\right\}}|\bm{\xi}_K}\label{eq_n_pb3_2} \\
&+\sum\abs{\mathbb{E}\brk{\mathbb{E}\brk{\bm{X}|\bm{Z}_1}\ensuremath{\mathds{1}}{\left\{\bm{Z}_1-\bm{\xi}_K \in  D_{q}\bc{1/K}\right\}}|\bm{\xi}_K}-\mathbb{E}\brk{\mathbb{E}\brk{\bm{X}|\bm{Z}_2}\ensuremath{\mathds{1}}{\left\{\bm{Z}_2-\bm{\xi}_K \in  D_{q}\bc{1/K}\right\}}|\bm{\xi}_K}}.\label{eq_n_pb3_3}
\end{align}
Since $\sum\ensuremath{\mathds{1}}{\left\{\bm{Z}_2-\bm{\xi}_K \in  D_{q}\bc{1/K}\right\}}=1$ and $\bm{\xi}_K$ and $\bc{\bm{Z}_1,\bm{Z}_2,\bm{X}}$ are independent, \cref{eq_n_pb3_1} and \cref{eq_n_pb3_2} reduce to
\begin{equation}\label{eq_pb4}
\begin{aligned}
&\sum\mathbb{E}\brk{\abs{f(\bm{Z}_1)-\mathbb{E}\brk{\bm{X}|\bm{Z}_1}}\ensuremath{\mathds{1}}{\left\{\bm{Z}_1-\bm{\xi}_K \in  D_{q}\bc{1/K}\right\}}|\bm{\xi}_K}
=\mathbb{E}\brk{\abs{f(\bm{Z}_1)-\mathbb{E}\brk{\bm{X}|\bm{Z}_1}}},
\end{aligned}
\end{equation}
and
\begin{equation}\label{eq_pb5}
\begin{aligned}
\sum\mathbb{E}\brk{\abs{\mathbb{E}\brk{\bm{X}|\bm{Z}_2}-g(\bm{Z}_2)}\ensuremath{\mathds{1}}{\left\{\bm{Z}_2-\bm{\xi}_K \in  D_{q}\bc{1/K}\right\}}|\bm{\xi}_K}
=\mathbb{E}\brk{\abs{\mathbb{E}\brk{\bm{X}|\bm{Z}_2}-g(\bm{Z}_2)}}.
\end{aligned}
\end{equation}
Let now $i \in \{1,2\}$. Again, since $\bm{\xi}_K$ and $\bc{\bm{Z}_1,\bm{Z}_2,\bm{X}}$ are independent, each expectation in \cref{eq_n_pb3_3} can be simplified as
\begin{align}\label{eq_pb100}
\mathbb{E}\brk{\mathbb{E}\brk{\bm{X}|\bm{Z}_i}\ensuremath{\mathds{1}}{\left\{\bm{Z}_i-\bm{\xi}_K \in  D_{q}\bc{1/K}\right\}}|\bm{\xi}_K} & = \mathbb{E}\brk{\mathbb{E}\brk{\bm{X}|\bm{Z}_i, \bm{\xi_K}}\ensuremath{\mathds{1}}{\left\{\bm{Z}_i-\bm{\xi}_K \in  D_{q}\bc{1/K}\right\}}|\bm{\xi}_K} \nonumber\\
& = \mathbb{E}\brk{\bm{X}\cdot\ensuremath{\mathds{1}}{\left\{\bm{Z}_i-\bm{\xi}_K \in  D_{q}\bc{1/K}\right\}}|\bm{\xi}_K}.
\end{align}
Plugging identity \cref{eq_pb100} into \cref{eq_n_pb3_3} and the triangle inequality yield
\begin{align}\label{eq_pb3}
&\sum\abs{\mathbb{E}\brk{\mathbb{E}\brk{\bm{X}|\bm{Z}_1}\ensuremath{\mathds{1}}{\left\{\bm{Z}_1-\bm{\xi}_K \in  D_{q}\bc{1/K}\right\}}|\bm{\xi}_K}-\mathbb{E}\brk{\mathbb{E}\brk{\bm{X}|\bm{Z}_2}\ensuremath{\mathds{1}}{\left\{\bm{Z}_2-\bm{\xi}_K \in  D_{q}\bc{1/K}\right\}}|\bm{\xi}_K}} \nonumber \\
\leq&\sum\mathbb{E}\brk{\abs{\bm{X}}\abs{\ensuremath{\mathds{1}}{\left\{\bm{Z}_1-\bm{\xi}_K \in  D_{q}\bc{1/K}\right\}}-\ensuremath{\mathds{1}}{\left\{\bm{Z}_2-\bm{\xi}_K \in  D_{q}\bc{1/K}\right\}}}|\bm{\xi}_K}.
\end{align}
Now again, by \Cref{ls1},
    \begin{align}\label{eq_pb1}
&\sum\mathbb{E}\brk{\mathbb{E}\brk{\abs{\bm{X}}\abs{\ensuremath{\mathds{1}}{\left\{\bm{Z}_1-\bm{\xi}_K \in  D_{q}\bc{1/K}\right\}}-\ensuremath{\mathds{1}}{\left\{\bm{Z}_2-\bm{\xi}_K \in  D_{q}\bc{1/K}\right\}}}|\bm{\xi}_K}} \nonumber \\
\leq &\sup_{x\in \mathcal{R}_{\bm{X}}}|x|\bc{\frac{4}{\varepsilon}\mathbb{E}\|\bm{Z}_1-\bm{Z}_2\|_\infty+2-2(1-K\varepsilon)^k}.
\end{align}
\cref{eh0} now follows by combining the bounds \cref{eq_n_pb1} -- \cref{eq_pb1} and the choice $\varepsilon = 1/K^2$.

The proof of \cref{eh0_new} follows along the same lines, since the triangle inequality $(a+b)^-\leq a^- +b^-$ and Jensen's inequality $\bc{\mathbb{E}\brk{\bm{a}}}^-\leq\mathbb{E}[\bc{\bm{a}}^-]$ hold for the negative part, as well as $a^-\leq \abs{a}$. Indeed, the only difference between the proofs is that we replace all absolute values $\abs{\cdot}$ in \cref{eq_n_pb1,eq_pb8_0,eq_pb8_new,eq_pb8,eq_n_pb2_new,eq_n_pb2_3,eq_n_pb3_new,eq_n_pb3_2,eq_pb5} by the corresponding negative parts, while we keep the absolute values in all other bounds.
\end{proof}

\begin{proof}[Proof of \Cref{apexp}]
\Cref{apexp} is an immediate consequence of \Cref{pexp}: In the notation of \Cref{pexp}, let $k=5$ and fix any $K>1$. We choose $\bm{Z}_1=\vze_{n+1,t/n}$, $\bm{Z}_2=\vze_{n,t/n}$ and  $\bm{X}=\teo{n+1\in\mathcal{W}\bc{\bm{T}_{n+1,t/n}[\bm{\theta}]}}$ for $\cW \in \{\cY, \cU, \cV\}$ with codomains $\mathcal{R}_{\bm{Z}_1}=\mathcal{R}_{\bm{Z}_2}=[0,1]^5$ and $\mathcal{R}_{\bm{X}}=[0,1]$, respectively. Next, let $f:\RR^{k} \to \RR$ be the projection onto the coordinate of $\zeta$ corresponding to $w \in \{y,u,v\}$, i.e. $f(\zeta)= f((x,y,z,u,v)) = w$, and $g:\RR^{k} \to \RR$, $g(\zeta)=W\bc{\zeta,\phi_t}$. Then \cref{eq_eqyuv} follows from \cref{eh0} by checking that
\begin{itemize}
  \item[(i)] $\mathbb{E}\abs{f(\bm{Z}_2)-g(\bm{Z}_2)}=\mathbb{E}\abs{\vw_{n,t/n}-W(\vze_{n,t/n}, \phi_t)}$;
  \item[(ii)] $\mathbb{E}\brk{\bm{X}|\bm{Z}_1}=\vw_{n+1,t/n}=f(\bm{Z}_1)$ by \Cref{cla_exc};
  \item[(iii)] $\sup_{\zeta\in[0,1]^5}\abs{f(\zeta)}=1$;
  \item[(iv)] $\sup_{x\in [0,1]}\abs{x}=1$;
  \item[(v)] $\sup_{\zeta\in[0,1]^5}\dabs{\nabla f(\zeta)}_\infty=1$;
  \item[(vi)] $\sup_{\zeta\in[0,1]^5}\dabs{\nabla g(\zeta)}_\infty\leq 2d$.
\end{itemize}

Analogously, \cref{eq_eqz} follows from \cref{eh0_new} by choosing $f\bc{\zeta}=z$, $g\bc{\zeta}=\phi_t\bc{y}$ and $\bm{X}=\teo{n+1\in\mathcal{Z}\bc{\bm{T}_{n+1,t/n}[\bm{\theta}]}}$, while the other parameters are as in the derivation of \cref{eq_eqyuv}.
\end{proof}

\end{appendices}

\end{document}